\documentclass[12pt]{amsart} 
\usepackage{verbatim, latexsym, amssymb, amsmath,color}
\usepackage{enumitem}
\usepackage{epsfig}

\usepackage{geometry}
\usepackage{hyperref}

\DeclareMathOperator{\Ric}{Ric}

\DeclareMathOperator{\Vol}{Vol}
\DeclareMathOperator{\area}{area}
\DeclareMathOperator{\genus}{genus}
\DeclareMathOperator{\ind}{index}

\DeclareMathOperator{\dist}{dist}
\DeclareMathOperator{\injrad}{injrad}
\DeclareMathOperator{\rad}{rad}
\DeclareMathOperator{\Sing}{Sing}

\newtheorem{theo}{Theorem}[]
\newtheorem{prop}[theo]{Proposition}
\newtheorem{lemme}[theo]{Lemma}
\newtheorem{definition}{Definition}[section]
\newtheorem{coro}[theo]{Corollary}
\newtheorem{remarque}{Remark}[section]

\begin{document}
\title[Bounded area 
minimal hypersurfaces]
{Morse index, Betti numbers and singular set of bounded area minimal hypersurfaces}

\author{Antoine Song}

\address{Department of Mathematics\\University of California, Berkeley\\Berkeley, CA 94720}
\email{aysong@math.caltech.edu}

\maketitle

\begin{abstract} 
We introduce a combinatorial argument to study closed minimal hypersurfaces of bounded area and high Morse index. Let $(M^{n+1},g)$ be a closed Riemannian manifold and $\Sigma\subset M$ be a closed embedded minimal hypersurface with area at most $A>0$ and with a singular set of Hausdorff dimension at most $n-7$. We show the following bounds: there is $C_A>0$ depending only on $n$, $g$, and $A$ so that
$$\sum_{i=0}^n b^i(\Sigma) \leq C_A \big(1+\ind(\Sigma)\big) \quad \text{ if $3\leq n+1\leq 7$},$$
$$\mathcal{H}^{n-7}\big(\Sing(\Sigma)\big) \leq C_A  \big(1+\ind(\Sigma)\big)^{7/n} \quad \text{ if $n+1\geq 8$},$$
where $b^i$ denote the Betti numbers over any field, $\mathcal{H}^{n-7}$ is the $(n-7)$-dimensional Hausdorff measure and $\Sing(\Sigma)$ is the singular set of $\Sigma$. In fact in dimension $n+1=3$, $C_A$ depends linearly on $A$. We list some open problems at the end of the paper.

\end{abstract}

\section*{Introduction}

In this paper we study families of closed embedded minimal hypersurfaces with a uniform bound on their $n$-volume but arbitrarily high Morse index. For reasons that will be explained later, we are interested in bounding the geometric complexity of a minimal hypersurface from above by the index. 

Let $(M^{n+1},g)$ be a closed Riemannian $n$-manifold. Our main results are divided into two cases: the low dimensional case $3\leq n+1 \leq 7$ and the higher dimensional case $n+1\geq 8$. In the following result, the area of a hypersurface in $M$ means its $n$-dimensional volume. We fix a field $F$ and denote by $b^i(.) = b^i(.,F)$ the dimension over $F$ of the cohomology groups $H^i(.,F)$ (in the statement $b^i(.)$ can in fact be replaced by the dimensions $b_i(.)$ for the homology groups $H_i(.,F)$). 

\begin{theo}\label{maina}
Let $(M,g)$ be a closed $(n+1)$-dimensional Riemannian manifold, with $2\leq n\leq 6$. For any $A>0$, there is a constant $C_A=C_A(M,g,A)$ depending only on the metric and $A$, such that for any smoothly embedded closed  minimal hypersurface $\Sigma\subset M$ of area at most $A$, 
$$ \sum_{i=0}^n b^i(\Sigma) \leq C_A (1+\ind(\Sigma)).$$
In dimension $n+1=3$, $C_A$ depends linearly on $A$, i.e.  $C_A\leq C.A$ for some constant $C$ depending only on the metric $g$.
\end{theo}

In our next main theorem, we say that $\Sigma$ is a closed minimal hypersurface smooth outside a singular set of Hausdorff dimension at most $n-7$ if $\Sigma$ is a stationary multiplicity one integral $n$-varifold whose support is smoothly embedded outside a set of dimension at most $n-7$. This assumption is natural in view of the fact that area minimizing hypersurfaces have this regularity, as well as minimal hypersurfaces produced by min-max theories. The $n$-mass of this varifold $\Sigma$ is called area. Its Morse index is defined to be the index of its regular part.

\begin{theo}\label{mainb}
Let $(M,g)$ be a closed $(n+1)$-dimensional Riemannian manifold, with $ n \geq 7$. For any $A>0$, there is a constant $C_A=C_A(M,g,A)$ depending only on the metric and $A$, such that for any closed embedded minimal hypersurface $\Sigma\subset M$ smooth outside a singular set $\Sing(\Sigma)$ of Hausdorff dimension at most $n-7$ and of area at most $A$, 
$$ \mathcal{H}^{n-7}(\Sing(\Sigma)) \leq C_A  (1+\ind(\Sigma))^{7/n}.$$
\end{theo}

In both theorems, we expect our inequalities to be sharp up to the non-explicit factor $C_A$ (it is indeed the case in dimension $n+1=3$). In dimension $n+1=3$, we conjecture a better general bound when the area is unbounded (see Section \ref{open}, Conjecture $\mathbf{C}_1$). For examples of infinite sequences of minimal surfaces with uniformly bounded area in the round $3$-sphere, see for instance \cite{Lawson, KapouleasYang, Kapouleasdoubling, Wiygulstacking, ChoeSoret, Ketoverequivariant, KeMaNe, KapouleasWiygulClifford, KapouleasMcGrathdoubling} (see also \cite{Kapouleas1, Kapouleas2}). In the statement of Theorem \ref{mainb}, the assumptions can probably be weakened by the regularity theory of N. Wickramasekera \cite{Wickramasekera}.

In dimensions $3\leq n+1\leq 7$, Theorem \ref{maina} can be thought of as a quantification of the following ``compactness'' results which hold under a uniform index bound assumption. R. Schoen and L. Simon \cite{SchoenSimon} proved that the set of stable minimal hypersurfaces (i.e. with Morse index $0$) with a uniform area bound is smoothly compact. This was extended by Sharp \cite{Sharp} to a varifold compactness result for any set of minimal hypersurfaces with uniform area and index bounds. By studying bounded index minimal hypersurfaces, H.-Z. Li and X. Zhou \cite{LiZhou} constructed general examples of sequence of unbounded index minimal hypersurfaces (see also \cite{Carlottolargeindex,Aiexnon}). O. Chodosh, D. Ketover and D. Maximo \cite{ChodKetMax} described the degenerations occurring in the compactness result of \cite{Sharp} and among other things showed that any space of minimal hypersurfaces with uniformly bounded area and index contains only finitely many diffeomorphism types (in particular the total Betti number is bounded). Another related article is A. Carlotto \cite{Carlottofinite}. Similar finiteness results for the number of diffeomorphism types, the total Betti number and the total curvature were obtained by R. Buzano and B. Sharp \cite{BuzSha} (see also \cite{AmbBuzCarSha}). In \cite{Maximonote}, D. Maximo proved Theorem \ref{maina} in dimension $3$, in the special case where the index is uniformly bounded. 

On the other hand Theorem \ref{mainb} quantifies the work of A. Naber and D. Valtorta \cite{NaberValtorta}, which bounds the size of the singular set of area minimizing hypersurfaces in dimensions $n+1\geq 8$.

One of our main motivations comes from the existence theory for minimal hypersurfaces. For simplicity let us focus on closed Riemannian $3$-manifolds. Many recent works revolving around S.-T. Yau's conjecture for minimal surfaces, which asks for the existence of infinitely many minimal surfaces in any closed $3$-manifold \cite{Yauproblemsection}, led to the realization that closed embedded minimal surfaces abound in closed $3$-manifolds. F.C. Marques and A. Neves \cite{MaNeinfinity} initiated a program to develop the min-max theory of Almgren-Pitts as an approach to the conjecture. Later K. Irie, F.C. Marques and A. Neves \cite{IrieMaNe}, O. Chodosh and C. Mantoulidis \cite{ChodMant}, X. Zhou \cite{Zhoumultiplicityone} proved the generic case of Yau's conjecture, and we settled the general case in \cite{AntoineYau}. One important aspect of min-max theory is that in a closed $3$-manifold, one can construct of a sequence $\{\Sigma_p\}_{p\in \mathbb{N}}$ of closed embedded minimal surfaces playing the role of non-linear geometric eigenfunctions. How geometrically complicated are these minimal hypersurfaces? As a Morse theoretic heuristic suggests, each $\Sigma_p$ has natural Morse index bounds \cite{MaNeindexbound, ChodMant, Zhoumultiplicityone}. It means that, in order to understand $\Sigma_p$, it is crucial to study the relation between the geometry of a minimal surface and its Morse index. The following conjecture of F.C. Marques, A. Neves and R. Schoen fits into that picture: if a closed manifold $(M^{n+1},g)$ has positive Ricci curvature, then any closed smoothly embedded minimal hypersurface $\Sigma^n \subset M$ should satisfy for some $C$ depending only on $(M,g)$:
\begin{equation} \label{conjMNS}
b^1(\Sigma) \leq C \ind(\Sigma)
\end{equation}
where $b^1(\Sigma)$ is the first Betti number over $\mathbb{R}$. In this paper, without assuming conditions on the ambient curvature, we find the first upper bound on the genus of these minimal surfaces $\Sigma_p$: since $\Sigma_p$ has area growing as $p^{1/3}$ and index at most $p$, we have as a corollary of Theorem \ref{maina} that for a constant $C=C(g)$,
$$\genus(\Sigma_p)\leq C p^{4/3}.$$
We conjecture that the optimal bound should be linear (see Section \ref{open},Conjecture $\mathbf{C}_1$). By \cite{Zhoumultiplicityone, MaNemultiplicityone} and \cite{EjiriMicallef}, it can be shown that generically
$$\genus(\Sigma_p) \geq C p.$$


Another motivation is to better understand families of minimal surfaces in $3$-manifolds which have bounded area but unbounded genus and index. For bounded index minimal surfaces with possibly unbounded area, we have already mentioned \cite{LiZhou,ChodKetMax,Carlottofinite,Maximonote}. The case where the area can be unbounded but the genus is uniformly bounded has been extensively studied by T. H. Colding and W. P. Minicozzi in a series of papers that give an essentially complete picture of what happens under these assumptions \cite{C&Mboundedgenus1,C&Mboundedgenus2,C&Mboundedgenus3,C&Mboundedgenus4,C&Mboundedgenus5}. For earlier work when $\Ric>0$, see  H. I. Choi-Schoen \cite{ChoiSchoen}. On the other hand, when the area is bounded but not the genus or the index, the situation remains foggy. The quantified results we propose here imply that genus and index are actually equivalent in this setting. Indeed as a corollary of Theorem \ref{maina} and \cite{EjiriMicallef}, we obtain:
\begin{coro} \label{genusindex}
Let $(M^3,g)$ be a closed $3$-manifold. For any $A>0$, there is a constant $C_A=C_A(M,g,A)>0$ such that for any closed embedded minimal surface $\Sigma$ of area at most $A$:
$$C_A^{-1}(\genus(\Sigma)+1)\leq \ind(\Sigma)+1\leq C_A (\genus(\Sigma)+1)$$
\end{coro} This is consistent with the index computations by N. Kapouleas-D. Wiygul of some Lawson surfaces \cite{KapouleasWiygul}.

Let us give some further background on quantified index bounds. In dimension $n=2$, optimal \textit{upper} bounds for the index in terms of the area and genus for $2$-dimensional minimal surfaces have been proved by N. Ejiri-M. Micallef \cite{EjiriMicallef}: they showed that for any closed minimal surface $\Sigma$ in a $3$-manifold $(M^3,g)$, there is $C=C(g)$ so that $$\ind(\Sigma)\leq C(\area(\Sigma)+\genus(\Sigma)).$$ 
There are no such upper index bounds for higher dimensional minimal hypersurfaces as it is shown by the minimal $3$-spheres with bounded area and unbounded index constructed by W.-Y. Hsiang \cite{HsiangWY1,HsiangWY2}. 

When it comes to \textit{lower} index bounds, starting with an idea of A. Ros \cite{Rostrick} there have been many articles verifying the conjecture of Marques-Neves-Schoen (\ref{conjMNS}) mentioned earlier for ambient spaces $(M,g)$ carrying special metrics: see \cite{Rostrick,AmbCarShanote} for flat tori; A. Savo \cite{Savo} for round spheres; F. Urbano \cite{Urbanosecond} for $S^1\times S^2$; L. Ambrozio, A. Carlotto and B. Sharp \cite{ACS} for compact rank one symmetric spaces and more; C. Gorodski, {R. A. E.} Mendes and M. Radeschi \cite{GMR} etc. These papers are based on subtle refinements of Ros' method and consequently their results do not depend on an area upper bound but require the metric to be very symmetric. Note that for a general metric $g$, Theorem \ref{maina} is false if we remove the dependency of $C_A$ on the area bound $A$ (there are examples of sequence of stable minimal surfaces with unbounded genus, see Example 1.2 in \cite{MatthiesenSiffertsystole}). To improve the understanding of high index minimal hypersurfaces in general Riemannian manifolds and prove Theorem \ref{maina}, we will not use A. Ros' technique but instead we introduce a new quantified covering argument that we explain in Section \ref{first}. 

In high dimensions $n+1\geq 8$, we cannot get a similar bound for the total Betti number as in Theorem \ref{maina}, and as a matter of fact we conjecture that there should be counterexamples (see Section \ref{open}). Theorem \ref{mainb} is the first bound on the size of the singular set by the Morse index. 

Finally we mention that in Section \ref{reamrks}, we introduce a two-piece decomposition for any $2$-dimensional minimal surface $\Sigma$ embedded in a $3$-manifold. This decomposition divides $\Sigma$ into a non-sheeted part and a sheeted part; it is analogous to the thick-thin decomposition for hyperbolic manifolds or manifolds with bounded sectional curvature. We will see that the non-sheeted part has area and genus controlled by the Morse index of $\Sigma$, independently of the total area of $\Sigma$ which can be arbitrarily large.

\subsection*{Organization} In Section \ref{first}, we outline the combinatorial arguments used to prove the main theorems, and list the main constants used throughout the paper. Sections \ref{sectionbetti} and \ref{sectiongenus} are focused on the proof of Theorem \ref{maina}. In Section \ref{reamrks} we introduce a two-piece decomposition for minimal surfaces and we improve Theorem \ref{maina} in dimension $3$.  In Section \ref{undostres}, Theorem \ref{mainb} is proved. Finally in the last section we list some open problems related to the geometric complexities of minimal hypersurfaces.

\subsection*{Acknowledgement} 
I am grateful to Fernando Cod\'{a} Marques and Andr\'{e} Neves for their continued support. This work benefited from extended discussions with Jonathan J. Zhu. I thank Otis Chodosh, Chao Li, Davi Maximo, Brian White, Xin Zhou for interesting conversations, and Hans-Joachim Hein for explaining to me his unpublished work with Aaron Naber on certain constructions of K\"{a}hler-Einstein metrics. I also thank Giada Franz and  Santiago Cordero Misteli for correcting a mistake in Lemma \ref{fin}. I am indebted to the reviewers, whose constructive comments and numerous corrections substantially improved the writing of this article.

The author was partially supported by NSF-DMS-1509027. This research was partially conducted during the period the author served as a Clay Research Fellow.

\section{Outline of proofs} \label{first}
In this first section, we explain our method, which relies on a geometric covering argument. Covering arguments implying topological or geometric ``finiteness'' properties had already been used in Differential Geometry, some instances include Gromov \cite[Chapter 13]{BallGromSchr} \cite{GromovBetti}, Naber-Valtorta \cite{NaberValtorta}. 
In our case this covering is given by using estimates for stable minimal hypersurfaces. The novelty in our paper is to relate in a quantified way the number of balls in the covering to the Morse index while dealing with the issue that the relative sizes and positions of these balls are uncontrolled. Our arguments are based on counting arguments which are not specific to minimal hypersurfaces and may apply to other variational objects.

\subsection{Definition of the folding number}\label{folding} \label{fofo}

In the following, $(M^{n+1},g)$ is a closed $(n+1)$-dimensional Riemannian manifold. We introduce the following useful quantity:

\begin{definition} 
Let $\Sigma$ be a closed embedded minimal hypersurface in $M$. The folding number of $\Sigma$, denoted by $\mathbf{f}(\Sigma)$, is the maximal number of disjoint open subsets $b\subset \Sigma$ of $\Sigma$ such that each $b$ is unstable. If $\Sigma$ is stable, we define by convention $\mathbf{f}(\Sigma)=0$.
\end{definition}
Note the following simple inequality 
$$\mathbf{f}(\Sigma)\leq \ind(\Sigma)$$
(see for instance \cite[Lemma 3.1]{Sharp}).
We chose this name to suggest that in lower dimensions, unstable minimal surfaces are generally folded while stable minimal surfaces are not folded because of the curvature estimates of Schoen \cite{Sc}, Schoen-Simon-Yau \cite{SSY}, Schoen-Simon \cite{SchoenSimon}. 

\subsection{Outline of proof of Theorem \ref{maina}}

There is a standard topological lemma which enables us to estimate the Betti numbers of a manifolds once we are given a cover by open sets whose finite intersections are topologically simple,  see Lemma \ref{lemmetopo} in the Appendix. It states that if we can find a family of $m$ open subsets $\{U_i\}_{i=1}^m$ with uniformly bounded ``overlap'', such that finite intersections have uniformly bounded Betti numbers then the Betti numbers of the manifold are bounded by a linear function in the number $m$ of open sets.

Fix $(M,g)$ of dimension $3$ to $7$, and let $\bar{r}\ll \injrad_M$,  $\lambda\gg1$. Let $\Sigma$ be a smoothly embedded closed minimal hypersurface in $(M,g)$ with $n$-volume at most $A$. We define the stability radius of $\Sigma$ as follows: for all $x\in M$, set
$$\mathbf{r}_{\mathrm{stab}}(x):=\sup\{r\leq\bar{r} ; B(x,\lambda r)\cap \Sigma \text{ is stable}\}.$$
The stability radius is in fact equivalent to a curvature radius in the sense that if $\mathbf{r}_{\mathrm{stab}}$ is small then the curvature of $\Sigma$ is large somewhere nearby, and conversely if $\mathbf{r}_{\mathrm{stab}}$ is large then the curvature of $\Sigma$ is small.

The strategy to prove Theorem \ref{maina} is the following. We cover $\Sigma$ with a finite number of open  geodesic balls $\{B_i\}_{i=1}^m$ of $(M,g)$ of the form $B(x, \mathbf{r}_{\mathrm{stab}}(x))$. If $\bar{r}$, $\frac{1}{\lambda}$ were chosen small enough, by well-known curvature bounds \cite{Sc,SchoenSimon}, each $B_i\cap \Sigma$ is a collection of almost flat $n$-disks (after rescaling). The number of such disks in a given $B_i$ is essentially bounded by the area of $\Sigma$ thanks to the monotonicity formula. Moreover we can guarantee using a Besicovitch covering argument that the overlap between the balls $B_i$ is bounded by a dimensional constant. Then applying the standard topological lemma recalled above to the cover of $\Sigma$ given by the connected components of $\Sigma \cap B_i$ ($i \in \{1,...,m\}$), the total Betti number of $\Sigma$ satisfies
\begin{equation} \label{apppli}
\sum_{i=0}^n b^i(\Sigma) \leq C m \area(\Sigma)
\end{equation}
where $C$ is a constant depending only on $(M,g)$ and $m$ is the number of balls $B_i$. Suppose for the sake of discussion that the radii of those balls $B_i$ are all strictly less than $\bar{r}$. In order to relate that inequality to the Morse index of $\Sigma$ we would like to construct another collection of balls $\{\tilde{B}_j\}_{j=1}^{\tilde{m}}$, each of the form $B(x, \mathbf{r}_{\mathrm{stab}}(x))$ with $\mathbf{r}_{\mathrm{stab}}(x)< \bar{r}$, such that
$$\forall i\neq j \in\{1,...,\tilde{m}\}, \quad \dist_g(\lambda \tilde{B}_i, \lambda \tilde{B}_j) >0,$$
$$ m \leq C_A\tilde{m},$$
where $C_A>0$ is a constant depending on $(M,g)$ and the area bound $A$. Indeed if such balls $\tilde{B}_j$ were found, since $\Sigma$ is unstable in each $s\lambda \tilde{B}_i$ whenever $s>1$, we would deduce
$$ m \leq  C_A\tilde{m} \leq C_A \mathbf{f}(\Sigma) \leq C_A \ind(\Sigma)$$
which, combined with (\ref{apppli}) would conclude the proof at least when $n\neq 2$.

To construct those balls $\tilde{B}_j$ we can try to extract a subfamily of balls in $\{B_i\}_{i=1}^m$ whose elements are far enough from each other. The difficulty lies in the fact that those balls $B_i $ have wildly different sizes in general. For simplicity, let us assume here that the radii $\rad(B_i)$ of $B_i$ are of the form $2^{-k_i}$ for some positive integer $k_i$. By extracting a subfamily of $\{B_i\}_{i=1}^m$, we can also easily assume that for each fixed $k$ the balls $3\lambda B_i$ which have radius $3\lambda 2^{-k}$ are at positive distance from one another. Now suppose that we can find an integer $J=J(M,g,A)$ such that $\{B_i\}_{i=1}^m$ satisfies the following \emph{``Property [J]''} (see Subsection \ref{combinatorial} for the rigorous definition): for each ${i_0} \in \{1,...,m\}$,
\begin{itemize}
\item either there are at most $J$ other balls $B_{i'}$  of  radii  at most $\rad(B_{i_0})$, such that
\begin{equation} \label{onepiece}
3\lambda B_{i'} \cap \lambda B_{i_0} \neq \varnothing,
\end{equation}
\item or there are $j_1,j_2 \in \{1,..., m\}$ such that $\rad(B_{j_1})=\rad(B_{j_2})$ and
$$3\lambda B_{j_1} \cap 3\lambda B_{j_2}  = \varnothing, \quad  3\lambda B_{j_1} \cup 3\lambda B_{j_2} \subset 2\lambda B_{i_0},$$
and there are at most $J$ other balls $B_{i'}$  of radii at least $\rad(B_{j_1})$, such that
\begin{equation} \label{twopiece}
3\lambda B_{i'}  \cap (3\lambda B_{j_1} \cup 3\lambda B_{j_2}) \neq \varnothing.
\end{equation}
\end{itemize}
Then we consider the balls $B_{i_0}$ of largest radius $2^{-1}$. Given such $B_{i_0}$, if the first possibility above occurs we keep $B_{i_0}$ but discard from our family $\{B_i\}_{i=1}^m$ the other balls $B_{i'}$ of radii at most $\rad(B_{i_0})$, satisfying (\ref{onepiece}); if the second possibility above occurs, we discard $B_{i_0}$ and all the balls $B_{i'}$ of radii at least $\rad(B_{j_1})$ satisfying (\ref{twopiece}), 
except the two balls $B_{j_1}, B_{j_2}$ themselves. By a top-down inductive replacement, we continue this procedure for each scale $2^{-k}$. After finitely many steps, the procedure ends and we are left with a subfamily of balls $B_{i_1},...,B_{i_L}$ such that 
$$\forall k\neq k' \in \{1,...,L\}, \quad \dist_g(\lambda B_{i_k}, \lambda B_{i_{k'}}) >0,$$
so that in particular
$$L \leq \mathbf{f}(\Sigma)\leq \ind(\Sigma).$$
On the other hand, we will prove in Proposition \ref{doublybis} that by construction of the subfamily, there is a constant $C_A$ depending on $(M,g)$ and $A$ but  independent of $\Sigma$ such that 
\begin{equation} \label{soccl}
m \leq C_A L.
\end{equation}
Roughly speaking, Proposition \ref{doublybis} is proved based on a \emph{``tree argument''}: alongside the inductive procedure we can construct a tree whose vertices are some balls of $\{B_i\}_{i=1}^m$, such that if $B_{i_0}$ is a vertex, it has no children if the first possibility of Property [J] occurs, and it has two children vertices if the second possibility occurs (corresponding to $B_{j_1}$ and $B_{j_2}$).  The leaves of this tree correspond to the balls $B_{i_1},...,B_{i_L}$ we were left with at the end of the inductive procedure. The number $L$ of leaves is comparable to the total number of vertices, which in turn is comparable to $m$. This concludes the proof, assuming Property P[J] is true. 

There are several issues with this outline. First, in the second alternative of Property P[J], in general we cannot find the balls $B_{j_1}$, $B_{j_2}$ in the original family $\{B_i\}_{i=1}^m$. It means that at the end of the inductive construction, we are left with new balls not necessarily contained in $\{B_i\}_{i=1}^m$. The second and more serious difficulty is that, even with that more flexible definition of Property P[J], \emph{it cannot be always true}. To see why, let us consider the following toy case. Suppose that inside $B(p,r)\backslash B(p,\epsilon)$, the minimal hypersurface $\Sigma $ is  smoothly close to a minimal hyperplane with integer multiplicity going through $p$ or more generally a smooth minimal cone with tip $p$, where $\epsilon>0$ is tiny, and $\Sigma$ has very large second fundamental form somewhere inside $B(p,\epsilon)$. Then it could be that there are an arbitrarily  large number of balls $\{b_1,...,b_{R}\}$ in the covering $\{B_i\}_{i=1}^m$ contained in $B(p,r)$ with radii roughly proportional to their distances to the tip $p$ but for any $i,j\in\{1,...,R\}$, 
$$\lambda b_i\cap \lambda b_j \neq \varnothing.$$ In that case we would not be able to find two disjoint balls of the form $\lambda B_i$ inside $B(p,r)$ and Property P[J] fails for any $J$ fixed a priori. In other words, there is possibly an issue whenever $\Sigma$ looks like a smooth cone, with large second fundamental form near the tip. We will see in Theorem \ref{noconcentration} and Proposition \ref{triplybis} using the monotonicity formula that \emph{conversely this is essentially the only situation where Property P[J] cannot hold}. Now the idea is to construct an open set of $M$ called \emph{``almost conical region $\mathfrak{C}$''} such that $\mathfrak{C} \cap \Sigma$ contains all the pieces of $\Sigma$ where it looks like a minimal cone with large second fundamental form near its tip. We can then remove that region from $\Sigma$, prove Property P[J] for $\Sigma \backslash \mathfrak{C}$ instead of $\Sigma$ and apply the outline of proof. This shows the desired control for the total Betti number of $\Sigma \backslash \mathfrak{C}$. Since $\mathfrak{C} \cap \Sigma$ is a union of cone-like regions diffeomorphic to some $X\times (0,1)$, with topology controlled by $(M,g)$ and $A$, the total Betti number of $\Sigma$ is equally well bounded, see Subsection \ref{splash}.
The almost conical region $\mathfrak{C}$ is constructed and analyzed in Subsection \ref{almostflat2}. This ends the sketch of proof for Theorem \ref{maina} when $n\neq 2$.

 \vspace{1em}

In dimension $n+1=3$, we can get better estimates for the constant $C_A$ in (\ref{soccl}) by showing that \emph{it depends linearly on $A$}. It comes from the fact that there are curvature estimates for stable surfaces independent of the area due to Schoen \cite{Sc}. We will also use the removable singularity theorem of Meeks-P\'{e}rez-Ros for minimal laminations \cite{MPR}. Their result applies to minimal surfaces without area bound, but with second fondamental form concentrating at a point, and roughly speaking it states that near this point but not too close, the minimal surface is almost flat after rescaling, in particular it is a union of disks and annuli. The region $\mathfrak{C}$ in dimension $3$ will be constructed so that $ \mathfrak{C}\subset \Sigma$ contains the union of these almost flat (after rescaling) regions in $\Sigma$. As in higher dimensions, if we start with a cover $\{B_i\}$ of $\Sigma\backslash \mathfrak{C}$ instead of $\Sigma$, then we avoid the issue of concentration of curvature at one point and we can finish the proof, see Section \ref{sectiongenus}. Compared to the case $3\leq n+1\leq 7$,  \cite{MPR}  replaces the monotonicity formula, but the proofs of Property P[J] and of the counting argument are essentially the same as in higher dimensions. 

Even when there is no area bound assumption on the minimal surface $\Sigma\subset (M,g)$, given a threshold number $N_0$, it is possible in dimension $3$ to define a natural \emph{two-piece decomposition}
$$\Sigma = \Sigma_{\leq N_0} \sqcup \Sigma_{>N_0}$$
where $\Sigma_{\leq N_0}$ is the part of $\Sigma$ where the rescaled area on scale of the stability radius $\mathbf{r}_{\mathrm{stab}}$ is at most $N_0$ (the non-sheeted part). The previous combinatorial arguments in fact automatically apply to $\Sigma_{\leq N_0}$, whose area and genus can be estimated in terms of the Morse index, see Section \ref{reamrks}.

 \vspace{1em}

 \vspace{1em}

\subsection{Outline of proof of Theorem \ref{mainb}} 

The proof is divided into two cases: when $n+1=8$ and when $n+1\geq 9$. 

In the former case, the proof uses the \emph{same} combinatorial argument as for Theorem \ref{maina}, we start with a cover of the (finite) singular set of $\Sigma$ by balls $\{B_i\}_{i=1}^K$ with bounded overlap and we try to find $K'$ disjoint balls of the form $\lambda B_i$ where $K'$ is larger than a definite fraction of $K$. Again we will need to first cut off an almost conical region $\mathfrak{C}$ from $\Sigma$.

In the case $n+1\geq 9$, the proof is in fact much easier due to the scaling properties of the size of the singular set, which now has positive dimension $n-7$. We do \emph{not} need to remove from $\Sigma$ an almost conical region. The proof is then based on the tree argument described earlier combined with the discrete H\"{o}lder's inequality. The reader might find it helpful to first understand the proof in dimensions $n+1\geq 9$ (Subsection \ref{plus grand que 9}) before going through the details of the rest of the paper.

 \vspace{1em}

\subsection{List of constants}
For the reader's convenience we list, for each section, the main constants or notations, the dependence between them and where they  are chosen.

\subsubsection{In Section \ref{sectionbetti}}
Let $(M,g)$ be a closed $(n+1)$-manifold with $n+1\leq 7$, $A>0$ and $\Sigma$ a closed embedded minimal hypersurface with $n$-volume at most $A$.

\begin{itemize}
\item $\mathbf{f}(\Sigma)$ is the folding number, it depends on $\Sigma \subset (M,g)$ and is defined in Subsection \ref{fofo},
\item $\bar{r}$, $\lambda$, $\varepsilon_{\mathrm{SS}}$ depend on $(M,g)$ and the area bound $A$; they are chosen in Subsections \ref{stabrad} and \ref{almostflat2}, in particular Lemma \ref{almostflatT},
\item $\mathbf{r}_{\mathrm{stab}}$ is the stability radius, it depends on $\bar{r}$, $\lambda$ and $\Sigma\subset (M,g)$, it is defined in Subsection \ref{stabrad},
\item $\bar{A}$ depends on $\bar{r}$, $A$,
 and is chosen in \ref{choiceofabar},
\item $\beta_0$ depends on $\bar{r}$, $\lambda$, $\bar{A}$, it is chosen in \ref{choiceofbeta0}, and is used to define the family of cones $\mathcal{G}_{\beta_0}$,
\item $\delta$, $K> 2^{1000}$ depend on $(M,g)$, $A$, $\bar{r}$, $\lambda$, $\varepsilon_{\mathrm{SS}}$, they are chosen in Lemma \ref{almostflatT} and they are used to define $(\delta,K)$-telescopes in \ref{defoftelescopes1},
\item $\mathfrak{C}$ is the almost conical region, which depends on $(M,g)$, $A$, $\bar{r}$, $\lambda$, $\varepsilon_{\mathrm{SS}}$, $\delta$, $K$ and is defined in Definition \ref{tuna1},
\item $\bar{K}:= 30 \lambda K^2$,
\item $\beta_1$, $\bar{R}$ depend on $\bar{A}$, $\lambda$, $\delta$, $K$ and are chosen in Theorem \ref{noconcentration},
\item $N(\Sigma)$ is the sheeting number, it depends on $\Sigma$, $\bar{r}$, $\lambda$ and is defined in Definition \ref{sheet},
\item $\mathcal{F}_k$ is the set of all geodesic balls $B(p,r)$ in $M$ such that $r\in [2^{-(k+1)},2^{-k})$.

\end{itemize}
\subsubsection{In Section \ref{sectiongenus}}

Let $(M,g)$ be a closed $3$-manifold and $\Sigma$ a closed embedded minimal surface.

\begin{itemize}
\item $\mathbf{f}(\Sigma)$ is the folding number, it depends on $\Sigma \subset (M,g)$ and is defined in Subsection \ref{fofo},
\item $\bar{r}$, $\lambda$, $\varepsilon_{\mathrm{S}}$ depend on $(M,g)$ only and are chosen in Subsection \ref{luyy}, in particular Lemma \ref{almostflat},
\item $\mathbf{r}_{\mathrm{stab}}$ is the stability radius, it depends on $\bar{r}$, $\lambda$ and $\Sigma\subset (M,g)$, it is defined in Subsection \ref{stabrad},
\item $\epsilon$, $K> 2^{1000}$ depend on $(M,g)$,  $\bar{r}$, $\lambda$, $\varepsilon_{\mathrm{S}}$, they are chosen in Lemma \ref{almostflat} and they are used to define $(\epsilon,K)$-telescopes in \ref{defoftelescopes2},
\item in dimension $3$, $\mathfrak{C}$ is the almost flat region, which depends on $(M,g)$, $\bar{r}$, $\lambda$, $\varepsilon_{\mathrm{S}}$, $\epsilon$, $K$ and is defined in Definition \ref{2020june},
\item $\bar{K}> 30 \lambda K^2$, $\beta_1$, $\bar{R}$ depend on $\lambda$, $\delta$, $K$ and are chosen in Theorem \ref{mpr},
\item $N(\Sigma)$ is the sheeting number, it depends on $\Sigma$, $\bar{r}$, $\lambda$ and is defined in Definition \ref{sheet},
\item $\mathcal{F}_k$ is the set of all geodesic balls $B(p,r)$ in $M$ such that $r\in [2^{-(k+1)},2^{-k})$.

\end{itemize}

\subsubsection{In Section \ref{undostres}}
Let $(M,g)$ be a closed $(n+1)$-manifold with $n+1\geq 8$, $A>0$ and $\Sigma$ a closed embedded minimal surface.

\begin{itemize}
\item $\mathbf{f}(\Sigma)$ is the folding number, it depends on $\Sigma \subset (M,g)$ and is defined in Subsection \ref{fofo},
\item $\lambda=2$, $\bar{r}$ depends on $(M,g)$ and is chosen in Subsection \ref{almostflat2}, in particular Lemma \ref{almostflatT},
\item $\mathbf{r}_{\mathrm{stab}}$ is the stability radius, it depends on $\bar{r}$,  and $\Sigma\subset (M,g)$, it is defined in Subsection \ref{stabrad},
\item $\bar{A}$ depends on $\bar{r}$, $A$,
 and is chosen in \ref{choiceofabar},
\item $\beta_0$ depends on $\bar{r}$, $\lambda$, $\bar{A}$, it is chosen in Lemma \ref{aroundtheworld}, and is used to define the family of cones $\mathcal{G}_{\beta_0}$,
\item $\delta$, $K> 2^{1000}$ depend on $(M,g)$, $A$, $\bar{r}$, they are chosen in Lemma \ref{almostflatT} and they are used to define $(\delta,K)$-telescopes in \ref{defoftelescopes1},
\item  $\mathfrak{C}$ is the almost conical region, which depends on $(M,g)$, $A$, $\bar{r}$, $\varepsilon_{\mathrm{SS}}$, $\epsilon$, $K$ and is defined in Definition \ref{tuna1},
\item $\bar{K}:= 60 K^2$,
\item $\beta_1$, $\bar{R}$ depend on $\bar{A}$, $\delta$, $K$ and are chosen in Theorem \ref{noconcentration8},
\item $N(\Sigma)$ is the sheeting number, it depends on $\Sigma$, $\bar{r}$ and is defined in Definition \ref{sheet},
\item $\mathcal{F}_k$ is the set of all geodesic balls $B(p,r)$ in $M$ such that $r\in [2^{-(k+1)},2^{-k})$.

\end{itemize}

\section{Total Betti number and Morse index in dimensions $3$ to $7$} \label{sectionbetti}

In this section, we prove Theorem \ref{maina} in dimensions $4$ to $7$. The arguments used work equally well for dimension $3$, but we will see in the next section that better estimates depending on the area bound $A$ can be obtained in that case.

\vspace{1em}

\subsection{The stability radius} \label{stabrad}
Let $(M,g)$ be a closed manifold of dimension $3\leq n+1\leq 7$.  Let $\bar{r}>0$, $\lambda> 1000$, $0<\varepsilon_{\mathrm{SS}}\ll 1$. We can assume $3\lambda \bar{r} <\injrad_M$.

We introduce a useful radius function associated to a closed immersed minimal hypersurface $\Sigma$ which is smooth (as all minimal hypersurfaces considered in this section). 

\begin{definition}
For all $x\in M$, set
$$\mathbf{r}_{\mathrm{stab}}(x):=\sup\{r\leq\bar{r} ; B(x,\lambda r)\cap \Sigma \text{ is stable}\}.$$
\end{definition}
In the above definition, $B(x,\lambda r)\cap \Sigma$ is said to be stable if empty. Sometimes, it will be convenient to allow $\bar{r}=\infty$ in this definition so that $\mathbf{r}_{\mathrm{stab}}$ can take values in $(0,\infty]$ (it will be mentioned whenever such a convention is made).

\begin{lemme} \label{continui}
Given $(M,g)$ and $\Sigma$ as above, the stability radius $\mathbf{r}_{\mathrm{stab}}:M\to (0,\bar{r}]$ is a $\frac{1}{\lambda}$-Lipschitz function; in particular it is a continuous function.
\end{lemme}

Let $A>0$. The following is a consequence of the curvature bounds for stable minimal hypersurfaces due to Schoen-Simon \cite[Section 6, Corollary 1]{SchoenSimon}: if we suppose that $\lambda$ is large enough and $\lambda \bar{r}$ small enough (both depending on $(M,g)$, $A$ and $\varepsilon_{\mathrm{SS}}$) then for any closed minimal hypersurface $\Sigma$ of $n$-volume at most $A$, for any $x\in M$ and any $y\in \Sigma\cap B(x,2\mathbf{r}_{\mathrm{stab}}(x))$,
\begin{equation} \label{ssbound}
\mathbf{r}_{\mathrm{stab}}(x)|\mathbf{A}(y)|< \varepsilon_{\mathrm{SS}},
\end{equation}
where $|\mathbf{A}(y)|$ denotes the norm of the second fundamental form of $\Sigma$ at $y$.

We choose $\varepsilon_{\mathrm{SS}}$ small enough, and rescale again the metric so that for any embedded hypersurface $\Sigma$ and $x\in M$, if for some $r\leq \bar{r}$, for all $y\in \Sigma\cap B(x,2r)$, we have
$$r|\mathbf{A}(y)|\leq \varepsilon_{\mathrm{SS}},$$ then $r$ is less that the convexity radius of $\Sigma$ with the metric induced by $g$, and $\Sigma\cap B(x,r)$ (and also $\Sigma\cap B(x,2r)$) is a union of $n$-dimensional disks which are convex with respect to the induced metric on $\Sigma$. This is true because the second fundamental form of $\partial B(x,r)$ with respect to the metric $r^{-2}g$ becomes arbitrarily close to that of the unit round hypersphere as $g$ converges to the Euclidean metric, while with respect to $r^{-2}g$ the second fundamental form of the hypersurface satisfying $r|\mathbf{A}(y)|\leq \varepsilon_{\mathrm{SS}}$ is at most $\varepsilon_{\mathrm{SS}}$. 
In particular, for any points $x_1,...,x_J\in M$, the connected components of 
$$\Sigma \cap \bigcap_{j=1}^J B(x_j,\mathbf{r}_{\mathrm{stab}}(x_j))$$ are open convex sets diffeomorphic to $n$-disks.

Another important point due to Lemma \ref{continui} is that,
given $y\in M$, the maximal number of disjoint balls of the form $ B(x,\mathbf{r}_{\mathrm{stab}}(x))$ intersecting $B(y,\mathbf{r}_{\mathrm{stab}}(y))$ non-trivially is bounded above by a constant 
\begin{equation}\label{muuu}
\text{ $\hat{\mu}= \hat{\mu}(M,g,\bar{r},\lambda)>0$ independent of $\Sigma$.}
\end{equation}

There is a reverse inequality that we will also often use, and which is based on the fact that a minimal disk in the Euclidean $1$-ball with sufficiently small second fundamental form everywhere is necessarily stable. By definition of the stability radius, if for a point $x$ we have
$\mathbf{r}_{\mathrm{stab}}(x)  < \bar{r}$ then the ball $B(x, 2 \lambda \mathbf{r}_{\mathrm{stab}}(x))$ is unstable so if $\bar{r}$ was chosen small enough, for an $\varepsilon'>0$ depending on $(M,g), \bar{r}$, $\lambda$: 
\begin{equation}\label{reverse1}
\text{there is $y'\in B(x, 2 \lambda \mathbf{r}_{\mathrm{stab}}(x))$ such that }\mathbf{r}_{\mathrm{stab}}(x)|\mathbf{A}(y')|>\varepsilon'. 
\end{equation}


By definition, clearly the maximal number of disjoint balls of the form $B(x,\lambda \mathbf{r}_{\mathrm{stab}}(x))$ with $x\in M$ and $\mathbf{r}_{\mathrm{stab}}(x)<\bar{r}$ at positive distance from one another is smaller than or equal to the folding number $\mathbf{f}(\Sigma)$ defined in the Subsection \ref{folding}. Defining $\mathbf{r}_{\mathrm{stab}}$ enables us to rigorously quantify the argument outlined in Section \ref{first}.

\vspace{1em}

\subsection{The almost conical region $\mathfrak{C}$} \label{almostflat2}
In this subsection we define the almost conical region $\mathfrak{C}$ and list some useful properties.

\subsubsection{Choice of the constant $\bar{A}$} \label{choiceofabar}

Given $A>0$, we can suppose (by rescaling $g$) that balls of radius $\bar{r}$ are sufficiently close to Euclidean so that by the monotonicity formula, for any ball $B(x,r)$ with $r\leq \bar{r}$, and any closed minimal hypersurface $\Sigma'$ of $n$-volume at most $A$, the $n$-volume of $\Sigma'\cap B(x,r)$ is at most $\bar{A}r^n/2$.

\subsubsection{The family of minimal cones $\mathcal{G}_{\beta_0}$} \label{choiceofbeta0}
Consider a complete Riemannian manifold $(M^{n+1},g)$ of dimension $3\leq n+1\leq 7$. If $\Sigma$ is a hypersurface, $p\in M$, $r>0$ denote by $ \Theta_g(p,r)$ the usual quantity
$$\frac{\text{$n$-volume of } B_g(p,r) \cap \Sigma}{\omega_n r^n}$$
where  the $n$-volume is computed with $g$ and $\omega_n$ is the Euclidean $n$-volume of the unit $n$-disk. We denote by $\Theta_{\text{Eucl}}(.,.)$ the analogue quantity for the Euclidean metric. For $s\leq t<\bar{r}$ and $p\in M$, let $A(p,s,t)$ be the open annulus $B(p,t)\backslash \overline{B}(p,s)\subset M$ where $\overline{B}$ is the closure of $B$.

Let $\bar{A}>0$ chosen as above (depending on the area bound $A$). Let $\Gamma$ be a minimal cone of $\mathbb{R}^{n+1}$
whose tip is $0\in\mathbb{R}^{n+1}$ and smoothly embedded outside $\{0\}$, such that $ \Theta_{\text{Eucl}}(0,1)\leq \bar{A}$. By the Frankel property and embeddedness, $\Gamma$ is connected. Let $|\mathbf{A}(.)|$ denote the norm of the second fundamental form of such a $\Gamma$ at a point. 
\vspace{1em}

In the following paragraph, we set $\bar{r}=\infty$ so that $\mathbf{r}_{\mathrm{stab}}$ can take values in $(0,\infty]$.
Let $\beta_0>0$ be large enough so that for $\Gamma$ as above, with
$$\max_{y\in \Gamma\cap \partial B_{\text{Eucl}}(0,1)} |\mathbf{A}(y)| \geq \beta_0,$$
there exist two balls $b_1,b_2\subset B_{\text{Eucl}}(0,2)\backslash B_{\text{Eucl}}(0,1/2)$ of the form 
$$b_1= B(x,\mathbf{r}_{\mathrm{stab}}(x)), \quad b_2= B( \frac{4x}{3},  \mathbf{r}_{\mathrm{stab}}(\frac{4x}{3}))$$
such that the following two conditions are satisfied:
$$4\lambda b_1\cap 4\lambda b_2 =\varnothing,$$ 
$$\mathbf{r}_{\mathrm{stab}}(x) < 2^{-1000}.$$
The constant $\beta_0$ depends on $\bar{A}$ and $\bar{r}$, $\lambda$, and its existence follows from the curvature bounds of Schoen-Simon: because of the density bound $\bar{A}$, if the curvature of $\Gamma$ is huge at some point $x$ of $\Gamma\cap \partial B_{\text{Eucl}}(0,1)$, then $\mathbf{r}_{\mathrm{stab}}(x)$ is tiny by Schoen-Simon, and we use $\mathbf{r}_{\mathrm{stab}}(\frac{4x}{3}) = \frac{4}{3}\mathbf{r}_{\mathrm{stab}}(x)$ to find $b_1,b_2$ as above.

We set
\begin{align*}
\mathcal{G}_{\beta_0} :=\{&\text{ minimal cones $\Gamma$ tipped at $0$ smooth outside of $0$} \\
& \text{ such that } \Theta_{\text{Eucl}}(0,1)\leq \bar{A} \text{ and }\max_{y\in \Gamma\cap \partial B_{\text{Eucl}}(0,1)} |\mathbf{A}(y)| \leq \beta_0\}.
 \end{align*}
By definition, cones of $\mathcal{G}_{\beta_0}$ have uniformly bounded second fundamental form and volume in the annulus $B_{\text{Eucl}}(0,1) \backslash B_{\text{Eucl}}(0,1/2)$. By the usual compactness theorems for minimal hypersurfaces with bounded second fundamental form and bounded volume, the set of diffeomorphism types of the cross-sections $\Gamma \cap\partial B_{\text{Eucl}}(0,1) $ for $\Gamma \in \mathcal{G}_{\beta_0}$ is finite. In particular, with respect to any field, the Betti numbers of such cross-sections is bounded by a number depending only on $\bar{A}$ and $\beta_0$. 
 

\vspace{1em}

\subsubsection{Definition of pointed $\delta$-conical annuli} Let $\delta>0$.
For any bounded subset $\Omega$ of $\mathbb{R}^{n+1}$, we fix in what follows a metric on the set of smooth metrics on the closure of $\Omega$ reflecting the $C^5$ topology.

Recall that $\bar{r}<\injrad_M$ is a radius chosen so that all the analysis will be done below scale $\bar{r}$. 
Let $\Sigma\subset (M,g)$ be a closed embedded minimal hypersurface whose $n$-volume is at most $A$. 
For $s\leq \frac{1}{2}t<\frac{1}{2}\bar{r}$ and $p\in M$, we will say that $\Sigma \cap A(p,s,t)$ is \textit{$\delta$-close to a cone} $\Gamma \in \mathcal{G}_{\beta_0}$ if the following is true: there exists $\Gamma \in \mathcal{G}_{\beta_0}$ so that for all $r\in[s,t/2]$ and for each component $\Sigma_{0,r}$ of $\Sigma \cap A(p,r,2r)$, there is a diffeomorphism $\Phi:(B(p,2r), g/{r}{^2})  \to (B_{\text{Eucl}}(0,2), g_{\text{Eucl}})$ with 
\begin{enumerate} [label=(\roman*)]
\item $(\Phi^{-1})^*(g/{r}{^2})$ is $\delta$-close to $g_{\text{Eucl}}$ in the $C^5$-topology,
\item $\Phi(A(p,r,2r)) =A_{\text{Eucl}}(0,1,2)$,
\item $\Phi(\Sigma_{0,r}) =\Gamma \cap A_{\text{Eucl}}(0,1,2)$.
\end{enumerate}
We will say that $A(p,s,t)$ is a \textit{pointed $\delta$-conical annulus} if 
$$\mathbf{r}_{\mathrm{stab}}(p) < s \leq  \frac{1}{2}t<\frac{1}{2}\bar{r}$$
and if $\Sigma \cap A(p,s,t)$ is $\delta$-close to a cone $\Gamma \in \mathcal{G}_{\beta_0}$. 

If $\delta$ is small enough and if $A(p,s,t)$ is a pointed $\delta$-conical annulus, by embeddedness of $\Sigma$ and the Frankel property, 
the components of $\Sigma \cap \partial B(p,t)$ are diffeomorphic to the cross-section of a same cone in $\mathcal{G}_{\beta_0}$, and 
$\Sigma \cap A(p,s,t)$ is diffeomorphic to $\big(\Sigma \cap \partial B(p,t)\big) \times (0,1)$. The reverse inequality (\ref{reverse1}) also implies that there is a point in $B(p,2\lambda s)$ with second fundamental form of norm at least $\frac{\varepsilon'}{s}$. In fact if $\delta = \delta(\bar{A},\beta_0)$ is small, by the properties of $\mathbf{r}_{\mathrm{stab}}$ and the fact that cross-sections of cones in $\mathcal{G}_{\beta_0}$ have uniformly bounded second fundamental form, for any pointed $\delta$-conical annulus $A(p,s,t)$ and any $x$ such that
$$B(x, \mathbf{r}_{\mathrm{stab}}(x)) \cap \partial B(p,t) \neq \varnothing,$$
the stability radius is comparable to $t$:
\begin{equation} \label{review1}
C_{\beta_0}t \leq \mathbf{r}_{\mathrm{stab}}(x) \leq (1-\frac{1}{\lambda})^{-1}(\frac{t}{\lambda} + s)
\end{equation}
for some $C_{\beta_0}>0$ depending only on $\beta_0$ 
but independent of $\Sigma$ or $t$. Here are some explanations. First by $\frac{1}{\lambda}$-Lipschitzness of $\mathbf{r}_{\mathrm{stab}}$, we have $\mathbf{r}_{\mathrm{stab}}(x) - \mathbf{r}_{\mathrm{stab}}(p) \leq \frac{1}{\lambda}\dist_g(x,p) \leq \frac{1}{\lambda}(\mathbf{r}_{\mathrm{stab}}(x) + t)$, which implies $\mathbf{r}_{\mathrm{stab}}(x) \leq (1-\frac{1}{\lambda})^{-1}(\frac{t}{\lambda} +\mathbf{r}_{\mathrm{stab}}(p))\leq (1-\frac{1}{\lambda})^{-1}(\frac{t}{\lambda} + s)$. Next, assume $\mathbf{r}_{\mathrm{stab}}(x) \leq \frac{t}{2}$, then if $\delta = \delta(\bar{A},\beta_0)$ is small, the curvature of $\Sigma\cap A(p,\frac{t}{2},t)$ rescaled by $t^{2}$ is uniformly bounded by a constant depending only on $\beta_0$, and so $\mathbf{r}_{\mathrm{stab}}(x)\geq C_{\beta_0}t $ for $C_{\beta_0}>0$ depending only on $\beta_0$.




\vspace{1em}

\subsubsection{Definition of $(\delta,K)$-telescopes} \label{defoftelescopes1}

Consider a parameter $K>1000 + \lambda$. Define
$$\mathcal{A}_{\delta}=\{A(p,s,2 s) ; A(p,\frac{s}{K},Ks) \text{ is a pointed $\delta$-conical annulus}\},$$
$$\mathcal{A}^{bis}_{\delta}=\{A(p,s,2 s) ; A(p,\frac{100s}{K},\frac{Ks}{100}) \text{ is a pointed $\delta$-conical annulus}\}$$
(note that $\mathcal{A}_{\delta}\subset \mathcal{A}^{bis}_{\delta}$). Those sets depend on $(M,g)$, $\Sigma$,  $\bar{A}$, $\lambda$, $\bar{r}$, $\delta$, $K$.
Because of the reverse inequality (\ref{reverse1}), roughly speaking if $A(p,s,2s)$ is in $\mathcal{A}_{\delta}$, then by rescaling it to have size $1$, the second fundamental form of $\Sigma \cap A(p,s,2s)$ is uniformly bounded while there is a point in $B(p,\frac{2\lambda s}{K})$ with large second fundamental form when $K$ is chosen large compared to $\lambda$.

Let $A(z_1,r_1,2r_1),..., A(z_m,r_m,2r_m)$ be annuli in $\mathcal{A}^{bis}_{\delta}$. We say that 
$$T:= \bigcup_{i=1}^m A(z_i,r_i,2r_i)$$ is a \textit{$(\delta,K)$-telescope} if for all $i=1,...m-1$, 
\begin{equation}\label{ininin}
\begin{split}
\overline{B}(z_i,r_i) & \subset B(z_{i+1},r_{i+1}) \\
\overline{B}(z_{i+1},r_{i+1}) & \subset B(z_i,2r_i)\\ 
\overline{B}(z_i,2r_i) & \subset B(z_{i+1},2r_{i+1})
\end{split}
\end{equation}
where $\overline{B}$ denotes the closure of $B$.

Let $T= \bigcup_{i=1}^m A(z_i,r_i,2r_i)$ be a $(\delta,K)$-telescope as above.
Note that $\partial T$ has two components $\partial_-T := \partial B(z_1,r_1)$ and $\partial_+T:= \partial B(z_m,2r_m)$

The next lemma states that if $K$ is large enough, any two annuli of $\mathcal{A}^{bis}_{\delta}$ which contain respective points close to each other on their scale have in fact very close centers on their scale and comparable radii.

\begin{lemme}\label{Fact}
 For any $K_1>0$, if $K = K(M,g,\bar{r},\lambda,A,\delta, K_1)$ is chosen large enough, if moreover $A(x,u,2u)$ and $A(y,v, 2v)$ are two annuli in $\mathcal{A}^{bis}_{\delta}$ such that
\begin{equation} \label{assumption1}
\dist_g(A(x,u,2u), A(y,v,2v))\leq \frac{u}{3},
\end{equation}
then 
$$\dist_g(x,y) \leq \frac{u}{K_1} \text{ and } \frac{u}{4} < v <4u.$$
\end{lemme}

\begin{proof}
In this statement, the distance between two annuli is understood in the usual sense of the infimum of the distance between two respective points in those annuli. Consider two annuli $A(x,u,2u), A(y,v,2v) \in \mathcal{A}^{bis}_{\delta}$ as above, and we can first assume that $v\leq u$. By definition, $A(x,\frac{100u}{K},\frac{Ku}{100})$ is a pointed $\delta$-conical annulus.
If $v\leq \frac{u}{4}$, then because of (\ref{assumption1}), we have
\begin{equation} \label{after2}
B(y,v) \subset B(x,3u)\backslash B(x,\frac{u}{6}). 
\end{equation}
By (\ref{review1}), 
$$ C_{\beta_0} \frac{u}{6} \leq  \mathbf{r}_{\mathrm{stab}}(y).$$
But by definition again, since $A(y,\frac{100v}{K},\frac{Kv}{100})$ is a pointed $\delta$-conical annulus, 
$$\mathbf{r}_{\mathrm{stab}}(y) < \frac{100v}{K},$$ 
so when $K$ is chosen large enough depending on $C_{\beta_0}$, we get a contradiction.
This implies that $\frac{u}{4}< v$ provided $K$ is large enough (once $\delta$ is fixed). By  similar arguments, 
we find that given $K_1$, when $K$ is large enough then $v <4u$ and $\dist_g(x,y) \leq \frac{u}{K_1}$. 
\end{proof}

The lemma below enables to combine a telescope with a certain annulus to get a larger telescope.
\begin{lemme}\label{tedddy}
If $K=K(M,g,\bar{r},\lambda,A,\delta)$ is chosen large enough, the following holds. Let $T = \bigcup_{i=1}^m A(z_i,r_i,2r_i)$ be a $(\delta,K)$-telescope satisfying (\ref{ininin}). For any annulus $A(p,s,2s)$ such that $A(p,\frac{75}{K} s, \frac{K}{75}s)$ is a pointed $\delta$-conical annulus, if 
$$\overline{B}(p,s) \subset B(z_1,r_1), \quad\overline{B}(z_1,r_1) \subset B(p,2s),  \quad\overline{B}(p,2s) \subset B(z_m, 2r_m),$$
$$\text{or}  \quad \overline{B}(z_1,r_1) \subset B(p,s),  \quad\overline{B}(p,s) \subset B(z_m,2r_m),  \quad\overline{B}(z_m,2r_m) \subset B(p,2s),$$
then $A(p,s,2s)\cup T$ is also a $(\delta,K)$-telescope.
\end{lemme}
\begin{proof}
We prove the lemma when $\overline{B}(p,s) \subset B(z_1,r_1), \overline{B}(z_1,r_1) \subset B(p,2s),  \overline{B}(p,2s) \subset B(z_m, 2r_m),$ the other case is similar. Let 
$$i_0 := \min\{i\in \{1,...,m\}; \quad \partial {B}(z_i,2r_i) \cap B(p,2s) =\varnothing\}.$$
By Lemma \ref{Fact}, for $K$ large enough, $\dist_g(p,z_{i_0}) \leq 10^{-10} s$.
By definition of $i_0$, necessarily $$\partial B(z_{i_0},r_{i_0}) \cap B(p,2s) \neq \varnothing.$$ Indeed if $i_0=1$ then this is true by hypothesis, and if $i_0>1$ then $\overline{B}(p,2s) \subset {B}(z_{i_0},2r_{i_0})$ and $\partial {B}(z_{i_0-1},2r_{i_0-1}) \cap B(p,2s) \neq\varnothing$, so we get the same conclusion.

If $\partial B(z_{i_0},r_{i_0})\subset B(p,2s)$ then $\overline{B}(z_{i_0},r_{i_0}) \subset B(p,2s)$ so 
$$A(p,s,2s) \cup \bigcup_{i=i_0}^m A(z_i,r_i,2r_i) = A(p,s,2s)\cup T$$
is a $(\delta,K)$-telescope. 

If $\partial B(z_{i_0},r_{i_0})\nsubseteq B(p,2s)$, then $\partial B(z_{i_0},r_{i_0}) \cap \partial B(p,2s) \neq \varnothing$.  Recall that for $K$ large enough, $\dist_g(p,z_{i_0}) \leq 10^{-10} s$.
 Moreover $A(p,1.1s,2.2s) \in \mathcal{A}^{bis}_{\delta}$ because $A(p,\frac{75}{K} s, \frac{K}{75}s)$ is a pointed $\delta$-conical annulus.
Thus for $K$ large enough,
$$A(p,s,2s) \cup A(p,1.1s,2.2s) \cup \bigcup_{i=i_0}^m A(z_i,r_i,2r_i) = A(p,s,2s)\cup T $$ 
is again a $(\delta,K)$-telescope, and the lemma is proved.  
\end{proof}

\vspace{1em}

\subsubsection{Construction of the almost conical region $\mathfrak{C}\subset M$ associated to $\Sigma$}

The distance between two closed subsets of $(M,g)$ is as usual the infimum of the distance between two points belonging respectively to these two sets.

\begin{lemme} \label{almostflatT}

Let $(M^{n+1},g)$ be a closed manifold of dimension $3\leq n+1\leq 7$ and $A>0$.
If the constants $\varepsilon_{\mathrm{SS}}, \frac{1}{\lambda}, \bar{r}$ are small enough, and if $\delta$, $\frac{1}{K}$ are chosen small enough depending on the previous constants then the following is true.

Let $\Sigma$ be a closed smoothly embedded minimal hypersurface of $n$-volume at most $A$ and let $\mathcal{A}_{\delta}$ be defined as above.  
Then there exist disjoint $(\delta,K)$-telescopes $T_1,...,T_L$  at positive distance from one another, and 
\begin{equation} \label{tubb}
\bigcup_{An\in \mathcal{A}_{\delta}} An \subset \bigcup_{i=1}^L T_i.
\end{equation}
Moreover
\begin{enumerate} [label=(\roman*)]
\item for any $i \in \{1,...,L\}$, the intersection $\Sigma \cap T_i$ is diffeomorphic to $(\partial_+T_i \cap \Sigma) \times (0,1)$ where $\partial_+T_i$ is one of the components of $\partial T_i$, and each connected component of $\partial_+T_i \cap \Sigma$ is diffeomorphic to the cross-section of a single cone $\Gamma_i \in \mathcal{G}_{\beta_0}$,
\item for any $i \in \{1,...,L\}$ and any points $x_1,..,x_J\in M$, the connected components of 
$$  \Sigma \cap T_i\cap \bigcap_{j=1}^J B(x_j,\mathbf{r}_{\mathrm{stab}}(x_j))$$ are diffeomorphic to $n$-disks, 
\item for any $i \in \{1,...,L\}$, $T_i$ intersects at most $\mu$ disjoint balls in 
$$\{B(x,\mathbf{r}_{\mathrm{stab}}(x)); x\in M\backslash T_i\}$$ where $\mu = \mu(M,g,\lambda, \bar{r},A)>0$,
\item for any $x\in M$, there are at most $2$ distinct indices $i,i'\in \{1,...,L\}$ such that
$B(x,\mathbf{r}_{\mathrm{stab}}(x)) \cap T_i \neq \varnothing$ and $B(x,\mathbf{r}_{\mathrm{stab}}(x)) \cap T_{i'} \neq \varnothing$.

\end{enumerate}

\end{lemme}

\begin{proof}
Take $K>1000$, $\delta$ to be fixed later, and consider an arbitrary parameter $K_1<K$.

First, note that $\Sigma$ being a fixed smooth hypersurface, $\mathbf{r}_{\mathrm{stab}}$ is bounded away from zero and
$$\hat{s}:=\inf\{s;\exists p, A(p,s,2 s) \in\mathcal{A}_{\delta}\} >0.$$
Any $A(p,s,2s) \in \mathcal{A}_{\delta}$ has its closure strictly contained in the union $A(p,0.9s,1.8s)\cup A(p,1.1s,2.2s)$, and one checks that $A(p,\frac{50}{K}0.9s,\frac{K}{50}0.9s)$ and $A(p,\frac{50}{K}1.1s,\frac{K}{50} 1.1s)$ are pointed $\delta$-conical annulus.
Consequently by a compactness argument, 
there are finitely many annuli $A(p_1, t_1,2 t_1),...,A(p_k, t_k,2 t_k) \in \mathcal{A}^{bis}_\delta$ such that if we set $\hat{An}_j :=A(p_j,{t_j},2t_j)$, 
$$\bigcup_{An\in \mathcal{A}_\delta} An \subset \bigcup_{j=1}^k \hat{An}_j$$
and for any $j\in \{1,...,k\}$, 
\begin{equation}\label{duredure}
A(p_j, \frac{50}{K} {t_j},\frac{K}{50}t_j) \text{ is a pointed $\delta$-conical annulus}.
\end{equation}
In particular all $\hat{An}_j$ belong to $\mathcal{A}^{bis}_\delta$.
To prove the first part of the lemma (\ref{tubb}) it is sufficient to show that there are finitely many $(\delta,K)$-telescopes at positive distance from each other (later called strictly disjoint), whose union contains $\bigcup_{j=1}^k \hat{An}_j$.

We would like to apply Lemma \ref{Fact} to construct the desired $(\delta,K)$-telescopes in (\ref{tubb}). 
Note that if $K$ is taken large enough then all the annuli we are considering   in this proof have an induced metric which is arbitrarily close to being flat (since given $A(y,v,2v) \in \mathcal{A}^{bis}_{\delta}$, $\frac{Kv}{100} < \bar{r}$).
We construct a finite number of strictly disjoint $(\delta,K)$-telescopes, whose union contains $\bigcup_{i=1}^k \hat{An}_i$ by the following inductive process. In the first step, we set $S^1_1:=\hat{An}_1$. 
Next, suppose that at step $j\in \{1,...,k-1\}$ we have a finite number of strictly disjoint $(\delta,K)$-telescopes $\{S^j_q\}_{q=1}^{l_j}$ whose union contains $\bigcup_{i=1}^j \hat{An}_i$. Each $S^j_q$ ($q\in \{1,...,l_j\}$) is by definition a union of elements in $\mathcal{A}^{bis}_\delta$:
\begin{equation}\label{biloop}
S^j_q = \bigcup_{i=1}^{m_{j,q}} A(x_{j,q,i},s_{j,q,i},2s_{j,q,i})\quad  \text{ (the annuli being naturally ordered)}.
\end{equation}
If $\hat{An}_{j+1}$ does not intersect $\bigcup_{q=1}^{l_j} S^j_q$, we can simply set $S^{j+1}_q := S^j_q$ for $q\in \{1,...,l_j\}$ and $S^{j+1}_{j+1}:= \hat{An}_{j+1}$. Otherwise $\hat{An}_{j+1}\cap \bigcup_{q=1}^{l_j} S^j_q\neq \varnothing$ and we want to modify some of the $S^j_q$ to get a new finite collection of strictly disjoint $(\delta,K)$-telescopes $S^{j+1}_1,...,S^{j+1}_{l_{j+1}}$ whose union contains $\bigcup_{i=1}^{j+1} \hat{An}_i$. 

Recall that $\hat{An}_{j+1} = A(p_{j+1},t_{j+1},2t_{j+1}) \in \mathcal{A}^{bis}_\delta$.
If $K=K(M,g,\bar{r},\lambda,A,\delta)$ is sufficiently large, the annulus $A(p_{j+1},0.9t_{j+1},2.2t_{j+1})$ can intersect at most $3$ disjoint $(\delta,K)$-telescopes of the form $S^j_q$. This is verified by a contradiction argument together with Lemma \ref{Fact}: assume we have a sequence of annuli $A_m:=A(x_m,0.9s_m,2.2s_m)$, such that for each $m$, $A_m$ intersects $4$ disjoint annuli $R_{m,1} := A(y_{m,1},r_{m,1},2r_{m,1}),...,R_{m,4} := A(y_{m,4},r_{m,4},2r_{m,4})$, and such that for all $\alpha\in \{1,2,3,4\}$,
$$\lim_{m\to \infty} s_m=0,\quad \lim_{m\to \infty}\frac{1}{s_m}\dist_g(x_m,y_{m,\alpha}) =0, \quad \frac{s_m}{4}< r_{m,\alpha} <4s_m.$$
Rescaling the metric by a factor $s_m^{-2}$, and taking a converging subsequence, we obtain a Euclidean annulus $A_\infty=A_{\text{Eucl}}(0,0.9,2.2) \subset \mathbb{R}^{n+1}$ whose closure intersects the closures of $4$ Euclidean annuli centered at $0$ with disjoint interiors, of the form $A_{\text{Eucl}}(0,r',2r')$, which is impossible. This proves that $A(p_{j+1},0.9t_{j+1},2.2t_{j+1})$ can intersect at most $3$ disjoint $S^j_q$ provided $K$ is large enough. Let $I_j \subset  \{1,...,l_j\}$ denote the set of (at most $3$) indices $q$ for which 
$$ A(p_{j+1},0.9t_{j+1},2.2t_{j+1}) \cap S^j_q \neq \varnothing.$$

Next, because of (\ref{duredure}), for any $\theta_1\in (0.9,1)$ and $\theta_2\in(1,1.1)$,
$$A(p_{j+1}, \frac{75}{K}\theta_1 t_{j+1},\frac{K}{75}\theta_1 t_{j+1}), A(p_{j+1}, \frac{75}{K}\theta_2 t_{j+1},\frac{K}{75}\theta_2 t_{j+1}) \text{ are pointed $\delta$-conical annulus}$$
so in particular
$$A(p_{j+1}, \theta_1 t_{j+1}, 2\theta_1 t_{j+1}), A(p_{j+1}, \theta_2 t_{j+1}, 2\theta_2 t_{j+1})\in \mathcal{A}^{bis}_\delta,$$
and moreover
$$\hat{An}_{j+1} \subset A(p_{j+1}, \theta_1 t_{j+1}, 2\theta_1 t_{j+1}) \cup A(p_{j+1}, \theta_2 t_{j+1}, 2\theta_2 t_{j+1}).$$
We claim the following:
\begin{equation}\label{theta1theta2}
\begin{split}
& \text{if $K=K(M,g,\bar{r},\lambda,A,\delta)$ is large enough,}\\
& \text{by choosing $\theta_1\in (0.9,1),\theta_2\in(1,1.1)$ appropriately,}\\
 & A(p_{j+1}, \theta_1 t_{j+1}, 2\theta_1 t_{j+1}) \cup A(p_{j+1}, \theta_2 t_{j+1}, 2\theta_2 t_{j+1}) \cup \bigcup_{q\in I_j} S^j_q
\\
& \text{is a strictly disjoint union of $(\delta,K)$-telescopes.} 
\end{split}
\end{equation}
Since the subset in (\ref{theta1theta2}) is at positive distance from $\bigcup_{q\in \{1,...,l_j\}\setminus I_j}S^j_q$, that would be enough to finish the inductive step, and conclude the construction of the strictly disjoint $(\delta,K)$-telescopes as in  (\ref{tubb}). Before checking (\ref{theta1theta2}), let us assume for simplicity that $I_j \subset  \{1,...,l_j\}$ is reduced to $1$ element, since the cases where it has $2$ or $3$ elements are treated with almost the same argument. Let us also assume without loss of generality that $I_j =\{1\}$.
If (\ref{theta1theta2}) could not be satisfied however large $K$ is, we can as previously argue by contradiction: 
suppose that for a sequence $K_\alpha\to \infty$ and a sequence of closed embedded minimal hypersurface $\{\Sigma_\alpha\}$, we can find an annulus and a $(\delta,K_\alpha)$-telescope of the type $A(p_{j+1},0.9t_{j+1},2.2t_{j+1})$, $S^j_1$ (not indexed by $\alpha$ to simplify notations) which do not satisfy (\ref{theta1theta2}).
Then by rescaling the ambient metric by $t_{j+1}^{-2}$, extracting a subsequence, taking a limit and using Lemma \ref{Fact}, we obtain a Euclidean annulus $A_{\text{Eucl}}(0,0.9,2.2)$ whose closure intersects the closure of another Euclidean annulus $A_{\text{Eucl}}(0,r_1,r_2)$ for some $r_1\in[0,2.2]$, $r_2\in [0.9,\infty]$. But we can always find $\theta_1\in (0.9,1),\theta_2\in(1,1.1)$  so that 
$$\{\theta_1,2\theta_1 ,\theta_2, 2\theta_2\} \cap \{r_1,r_2\} =\varnothing.$$
By Lemma \ref{tedddy}, this means that for large $\alpha$, 
$$A(p_{j+1}, \theta_1 t_{j+1}, 2\theta_1 t_{j+1}) \cup A(p_{j+1}, \theta_2 t_{j+1}, 2\theta_2 t_{j+1})\cup S^j_1$$ is either one $(\delta,K_\alpha)$-telescope or two strictly disjoint $(\delta,K_\alpha)$-telescopes.
 That contradicts our assumption that (\ref{theta1theta2}) could not be true for large $K$, which finishes the proof of (\ref{theta1theta2}) and (\ref{tubb}).

\vspace{1em}
It remains to show Properties $(i)$, $(ii)$, $(iii)$ and $(iv)$ in the statement. Whenever $\delta>0$ is small enough, if $s\leq \frac{1}{2}t< \frac{1}{2} \bar{r}$ and  if $\Sigma \cap A(p,s,t)$ is $\delta$-close to a cone $\Gamma\in \mathcal{G}_{\beta_0}$, then for any $s'\in (s,t)$ the boundary $\partial B(p,s')$ intersects $\Sigma$ transversally, with an angle $\frac{1}{1000}$-close to $\frac{\pi}{2}$. 
Let  $T_1,...,T_L$ be the disjoint $(\delta,K)$-telescopes constructed above so that (\ref{tubb}) is true. 
Suppose that 
$$T_1 = \bigcup_{i=1}^{m} A(z_i,r_i,2r_i)$$
where the annuli are naturally ordered, meaning that (\ref{ininin}) is satisfied. 

By definition of $\delta$-close to a cone, $\Sigma \cap A(z_1,r_1,2r_1)$ is diffeomorphic to $(\Sigma \cap \partial B(z_1,2r_1)) \times (0,1)$. To show $(i)$ we argue by induction: we suppose that for $j\in \{1,...,m-1\}$, $\Sigma \cap \bigcup_{i=1}^{j} A(z_i,r_i,2r_i)$ is diffeomorphic to $(\Sigma \cap \partial B(z_j,2r_j)) \times (0,1)$, and we want to check that
$$\Sigma \cap \bigcup_{i=1}^{j+1} A(z_i,r_i,2r_i) \text{  is diffeomorphic to } (\Sigma \cap \partial B(z_{j+1},2r_{j+1})) \times (0,1).$$
Note that, by choice of $\delta$, the restriction of the  function $\dist_g(z_{j+1},.)$ on $\Sigma$ is smooth, with gradient of norm close to $1$ inside $B(z_{j+1},2r_{j+1})\backslash B(z_{j},2r_{j})$. Every gradient line of  $\dist_g(z_{j+1},.)\big|_\Sigma$ starting at a point of $\Sigma\cap \partial B(z_{j+1},2r_{j+1})$ arrives after positive time at a point of $\Sigma\cap \partial B(z_{j},2r_{j})$, and vice versa any point of $\Sigma\cap \partial B(z_{j},2r_{j})$ can be thus attained, besides time of arrival is a smooth positive function on $\Sigma\cap \partial B(z_{j+1},2r_{j+1})$. Using Lemma \ref{Fact}, when $K$ is large enough, by rescaling the gradient flow of $\dist_g(z_{j+1},.)\big|_\Sigma$ we get a diffeomophism from $(\Sigma\cap \partial B(z_{j+1},2r_{j+1})) \times (0,1)$ to $\Sigma \cap (B(z_{j+1},2r_{j+1})\backslash B(z_{j},2r_{j}))$, which is enough to conclude the induction. Property $(i)$ also requires that each component of $\partial_+T_i \cap \Sigma$ is diffeomorphic to the cross-section of a same cone $\Gamma_i \in \mathcal{G}_{\beta_0}$ but this is clear from the above argument.

Consider a non-empty family $\{ B(x_j,\mathbf{r}_{\mathrm{stab}}(x_j))\}_{j=1}^J$. Recall that if $\varepsilon_\mathrm{SS}$ is small enough and the rescaled metric $\bar{r}^{-2}g$ is close enough to the flat metric in $1$-balls, the curvature bound (\ref{ssbound}) implies that any connected component of $\Sigma \cap \bigcap_{j=1}^J B(x_j,\mathbf{r}_{\mathrm{stab}}(x_j))$ is a convex open $n$-disk inside $\Sigma$ with the induced metric. 
If $\delta$, $\frac{1}{K}$, $\frac{1}{\lambda}$ are small, then the intersection between $\partial T_i$ and $\Sigma$ is transversal, with angle close to $\frac{\pi}{2}$.
Hence the second inequality in (\ref{review1}) implies that when $\delta$, $\frac{1}{K}$, $\frac{1}{\lambda}$ are small enough, for any telescope $T_i$,  any connected component of $\Sigma \cap T_i \cap \bigcap_{j=1}^J B(x_j,\mathbf{r}_{\mathrm{stab}}(x_j))$ is an open $n$-disk inside $\Sigma$. In particular Property $(ii)$ holds.
Property $(iii)$ follows from the first inequality in (\ref{review1}). Property $(iv)$ follows from the second inequality in (\ref{review1}),  Lemma \ref{Fact} and the fact that the telescopes $T_i$ are disjoint.

\end{proof}

We can now define the almost conical region (which depends on $\Sigma$).

\begin{definition} \label{tuna1}

Given a closed embedded minimal hypersurface $\Sigma$, we define the almost conical region as
$$\mathfrak{C} := 
\bigcup_{i=1}^L T_i,$$
where $ T_i$ are the $(\delta,K)$-telescopes given by Lemma \ref{almostflatT}. 
\end{definition}

Of course from the proof of Lemma \ref{almostflatT}, we see that the choice of these telescopes is not unique.
From now on we fix $\bar{r}$, $\lambda$, $A,\bar{A}$ as above, $\delta$ and $K>2^{1000}$ as in Lemma \ref{almostflatT}. 


\vspace{1em}

\subsection{No curvature concentration far from $\mathfrak{C}$} \label{noconcentr}

Recall that $\bar{r}$, $\lambda>2$, $\mathbf{r}_{\mathrm{stab}}$, $\bar{A}$, $\delta$, $K>2^{1000}$ were introduced in Subsection \ref{stabrad}. 
We also already fixed $\bar{A}>0$, $\beta_0>0$ and $\delta, K$, we defined $\mathcal{G}_{\beta_0}$ and being ``$\delta$-close to a cone $\Gamma \in \mathcal{G}_{\beta_0}$ ''  in the previous subsection.

We denote by $B_{\text{Eucl}}(.,.)$ (resp. $B(.,.)$) a geodesic open ball for the Euclidean metric (resp. the metric $g$).  
The following theorem holds for dimensions $3\leq n+1 \leq 7$. Roughly speaking it is a quantified version of the following statement: given a uniform $n$-volume bound, for a minimal hypersurface approximated by a cone in the varifold sense, if the curvature of the hypersurface is concentrated around a single point then the minimal hypersurface is in fact smoothly approximated by a smooth minimal cone on a large region. 
The assumptions of Theorem \ref{noconcentration} are written in a way to be more easily checked in the next Corollary \ref{sotechnical}, but are far from being optimal. In the statement of Theorem \ref{noconcentration}, we set $\bar{r}=\infty$ so that $\mathbf{r}_{\mathrm{stab}}$ can take values in $(0,\infty]$.


\begin{theo} \label{noconcentration}
Let $3\leq n+1\leq 7$.
Set $\bar{K} := 30\lambda K^2$. 

There exist 
$\beta_1>1$, $\mu>0$, and $\bar{R}>1000$ depending only on $\lambda, \bar{A}, \delta, {K}$ such that the following is true. Let $g$ be a metric $\mu$-close to the Euclidean metric in the $C^5$-topology on $B_{\text{Eucl}}(0,\bar{K}) \subset \mathbb{R}^{n+1}$, and let 
$$(\Sigma, \partial \Sigma) \subset (B_{\text{Eucl}}(0,\bar{K}), \partial B_{\text{Eucl}}(0,\bar{K}))$$ be a compact embedded minimal hypersurface with respect to $g$, such that 
\begin{itemize}
\item the $n$-volume of $\Sigma\cap B(0,\bar{K})$ is at most $\bar{A}\bar{K}^n/2$, 
\item $\mathbf{r}_{\mathrm{stab}}(0) < {\bar{K}}^{-1}$,
\item there is $y' \in B(0,7\lambda)$
with $\mathbf{r}_{\mathrm{stab}}(y')=1$, 
\item $\Theta_g(0,10\lambda K) -  \Theta_g(0,\frac{1}{3K}) \leq \beta_1^{-1}.$
\end{itemize}
Then 
\begin{enumerate} [label=(\roman*)]
\item either there are two balls $b_1=B(z',\mathbf{r}_{\mathrm{stab}}(z')), b_2=B(z'',\mathbf{r}_{\mathrm{stab}}(z''))$ such that 
$$z',z''\in B(0, \frac{\bar{K}}{2}), \quad 3\lambda b_1\cap 3\lambda b_2 =\varnothing,$$
$$\mathbf{r}_{\mathrm{stab}}(z'), \mathbf{r}_{\mathrm{stab}}(z'') \in [2^{-(R+1)}, 2^{-{R}}) \quad \text{for some $R\in [1000,\bar{R}]$},$$
\item or $\Sigma \cap A\big(0,\frac{1}{2{K}}, 7\lambda {K}\big)$  is ${\delta}$-close to a cone $\Gamma \in \mathcal{G}_{\beta_0}$ in the sense of Subsection \ref{almostflat2}. 

\end{enumerate}

\end{theo}

\begin{proof}

Let $\bar{K}:= 30\lambda K^2$ and consider a large constants $\bar{R}$. We will give an explicit bound later in the proof, see (\ref{impos}). 

We argue by contradiction and assume that the theorem is not true for $\bar{R}$: let $g_k$ be a sequence of metrics as in the statement converging to the Euclidean metric, and let $\Sigma_k$ be a sequence of minimal hypersurfaces (with respect to $g_k$) as in the statement, except that the third bullet is replaced by 
$$\lim_{k\to \infty} \Theta_{g_k}(0,10\lambda K) -  \Theta_{g_k}(0,\frac{1}{3K})=0,$$
and assume that neither $(i)$ nor $(ii)$ occurs. For clarity, denote by $B_k(.,.)$ 
a geodesic ball for the metric $g_k$. 
The assumption $\lambda \mathbf{r}_{\mathrm{stab}}(0)\leq \frac{1}{30{K^2}} \leq 1$ for each $k$, means that each $\Sigma_k$ intersects $B_k(0,1)$.
As noted in Subsection \ref{stabrad}, for $g_k$ and $\Sigma_k$, in any ball of the form $B_k(x,\mathbf{r}_{\mathrm{stab}}(x))$, the hypersurface $\Sigma_k$ satisfies curvature bounds due to Schoen-Simon of the form
\begin{equation} \label{schoensimon1}
|\mathbf{A}(y)| \mathbf{r}_{\mathrm{stab}}(x) \leq \varepsilon_{\mathrm{SS}},
\end{equation}
where $y\in B_k(x,\mathbf{r}_{\mathrm{stab}}(x))$ and $\varepsilon_{\mathrm{SS}}$ was chosen in Subsection \ref{stabrad}. By definition of $\mathbf{r}_{\mathrm{stab}}$, by the fact that $\mathbf{r}_{\mathrm{stab}}(y')=1$ and $\mathbf{r}_{\mathrm{stab}}(0)<\bar{K}^{-1}$, we have 
$$B_k(y',1)\subset B_k(0, 10\lambda K)\backslash B_k(0,\frac{1}{3{K}}).$$
Note that $\frac{\bar{K}}{2} > 10\lambda K$.  Let us check that provided $\bar{R}$ is well chosen, $(i)$ not happening means the following\newline


\textbf{Fact:}
If $\bar{R}$ is large enough depending only on $\lambda, K$, and if $(i)$ does not occur, then in $B_k(0,10\lambda {K})\backslash B_k(0,\frac{1}{3{K}})$, the stability radius $\mathbf{r}_{\mathrm{stab}}(.)$ of $\Sigma_k$ is  bounded below by $2^{-\bar{R}}$  (thus the second fundamental form is also uniformly bounded independently of $k$).
\newline

To check the Fact, we start by observing that given two points $a,b $ with 
$$B_k(a,\mathbf{r}_{\mathrm{stab}}(a)), B_k(b,\mathbf{r}_{\mathrm{stab}}(b)) \subset B_k(0, 10\lambda K),$$ for any number $s$ between $\mathbf{r}_{\mathrm{stab}}(a)$ and $\mathbf{r}_{\mathrm{stab}}(b)$, if the $s$-neighborhood of a continuous path $\gamma_{a,b}$ joining $a$ to $b$ is contained in $B(0,\bar{K})$, then there is a point $c\in \gamma_{a,b}$, such that $\mathbf{r}_{\mathrm{stab}}(c)=s$. This follows from the continuity of $\mathbf{r}_{\mathrm{stab}}$, see Lemma \ref{continui}. Moreover, when $k$ is large and $g_k$ close enough the Euclidean one, 
for any point $z\in B_k(0, 10\lambda K) \backslash B_k(0,\frac{1}{3{K}})$ there are two points $u,v\in B_k(y',1)$ and paths $\gamma_{0,u},\gamma_{z,v}$ so that 
$$\gamma_{z,v}\subset B_k(0, 10\lambda K) \backslash B_k(0,\frac{1}{3{K}})$$
 and the $\frac{1}{10K}$-neighborhoods (with respect to $g_k$) of $\gamma_{0,u}$ and $\gamma_{z,v}$ are disjoint, contained in $B_k(0,\bar{K})$. We impose the real number $\bar{R}$ to satisfy
\begin{equation}\label{impos}
\bar{R} = \log_2(30 \lambda K )+1000,
\end{equation}
so that in particular 
$$6\lambda 2^{-\bar{R}}<\frac{1}{10K} \text{ and }\bar{R}>1000.$$
 Now any point $u\in B_k(y',1)$ has $\mathbf{r}_{\mathrm{stab}}(u) \geq \frac{\lambda-1}{\lambda} \geq \frac{1}{2}$ by definition of $\mathbf{r}_{\mathrm{stab}}$. So because
$$\mathbf{r}_{\mathrm{stab}}(0)\leq  \frac{1}{\bar{K}} =  \frac{1}{30\lambda K^2}< 2^{-\bar{R}}$$ 
and due to the previous observations, for any point $z\in  B_k(0, 10\lambda K ) \backslash B_k(0,\frac{1}{3{K}})$ with $\mathbf{r}_{\mathrm{stab}}(z)<2^{-\bar{R}} $, there are two points $z',z''$ satisfying 
$$z',z''\in B_k(0,  10\lambda K  ) \subset B_k(0,  \frac{\bar{K}}{2}  )  ,$$
$$B_k(z',3\lambda \mathbf{r}_{\mathrm{stab}}(z'))\cap B_k(z'',3\lambda \mathbf{r}_{\mathrm{stab}}(z''))=\varnothing, $$
$$\text{ and } \mathbf{r}_{\mathrm{stab}}(z')= \mathbf{r}_{\mathrm{stab}}(z'') \in [2^{-(\bar{R} +1)}, 2^{-\bar{R}}).$$
More precisely, to construct $z',z''$ we do this: we find $u,v\in B_k(y',1)$, $z'\in\gamma_{0,u}$ and $z''\in\gamma_{z,v}$ such that $\mathbf{r}_{\mathrm{stab}}(z')=\mathbf{r}_{\mathrm{stab}}(z'') \in [2^{-(\bar{R} +1)}, 2^{-\bar{R}})$ and by construction $\dist_g(z',z'') >\frac{1}{10K}> 6\lambda2^{-\bar{R}}$.
But since we are assuming that $(i)$ does not happen, for any point $z\in  B_k(0, 10\lambda K ) \backslash B_k(0,\frac{1}{3{K}})$ we must have $\mathbf{r}_{\mathrm{stab}}(z)\geq 2^{-\bar{R}} $ and the Fact is proved.

Thanks to this Fact and the $n$-volume bound (first bullet in the assumptions), one can take a converging subsequence of $\Sigma_k$ in $B_k(0,10\lambda{K})\backslash B_{\text{Eucl}}(0,\frac{1}{3{K}})$ so that the smooth limit is a compact connected minimal hypersurface 
$$\Sigma_\infty\subset B_{\text{Eucl}}(0,10\lambda{K}) \backslash B_{\text{Eucl}}(0,\frac{1}{3{K}}).$$
Moreover, by assumption we have in the limit:
$$\Theta_{\text{Eucl}}(0,1) \leq \bar{A},$$
$$\Theta_{\text{Eucl}}(0,10\lambda K )-\Theta_{\text{Eucl}}(0,\frac{1}{3 K})=0.$$ This implies by the monotonicity formula that there is a smooth cone $\Gamma$ such that
$$\Sigma_\infty = \Gamma\cap \text{ in }\big(B_{\text{Eucl}}(0,10\lambda{K}) \backslash B_{\text{Eucl}}(0,\frac{1}{3{K}})\big).$$
For large $k$, by embeddedness, 
$$\Sigma_k\cap \big(B_k(0,7\lambda{K}) \backslash B_k(0,\frac{1}{2{K}})\big)$$
 is a union of graphs over $\Gamma$ in 
$$B_k(0,7\lambda{K}) \backslash B_k(0,\frac{1}{2{K}}).$$

The cone $\Gamma$ belongs to the family  $\mathcal{G}_{\beta_0}$ (see Subsection \ref{almostflat2} for its definition). Indeed, by our construction, any point of $\Gamma \backslash  B_{\text{Eucl}}(0,1/4)$ has $\mathbf{r}_{\mathrm{stab}} \geq 2^{-\bar{R}}$, so if $\Gamma$ were not a cone in $\mathcal{G}_{\beta_0}$ then there would be 	a point $x\in  B_{\text{Eucl}}(0,2)$ such that 
$$\mathbf{r}_{\mathrm{stab}}(x)= 2^{-(R+1)}  \text{ for some } R\in [1001,\bar{R}-1],$$
$$B(y,4\lambda \mathbf{r}_{\mathrm{stab}}(\frac{4x}{3})) \cap B(x,4\lambda \mathbf{r}_{\mathrm{stab}}(x)) = \varnothing.$$
This is not possible since we assumed that $(i)$ in the statement does not hold for $k$ large. To conclude, we indeed have $\Gamma \in\mathcal{G}_{\beta_0}$ so $(ii)$ actually holds for $k$ large enough which contradicts our assumption that the theorem was not true.

\end{proof}

The next technical lemma is a corollary of the previous theorem. Let $(M^{n+1},g)$, $3\leq n+1\leq 7$, $\bar{r}$, $\lambda>2$ be as in Subsection \ref{stabrad}, $A$, $\bar{A}$, $\delta$, $K>2^{1000}$ as in Subsection \ref{almostflat2}, let  $\bar{K}=30\lambda K^2$. The constants $\beta_1$, $\mu$, $\bar{R}>1000$ are given by Theorem \ref{noconcentration}. Define $\mathfrak{C}$ as in Subsection \ref{almostflat2}. We can assume that for any ball of radius $r$ less than $\bar{r}$ the metric $\frac{1}{{r}^2}g$ is $\mu$-close to a flat metric in the $C^5$-topology. Recall that given a ball $b$ and a positive number $t$, $tb$ denotes the ball with same center as $b$ and radius $t$ times larger.  We did not try to optimize the constants.

\begin{coro} \label{sotechnical}
Let $\Sigma\subset M^{n+1}$ be a closed smoothly embedded minimal hypersurface of $n$-volume at most $A$.
Consider $x_0,x_1,x_2\in M$ such that
\begin{itemize}
\item $\mathbf{r}_{\mathrm{stab}}(x_0) < {\bar{r}}$, $\mathbf{r}_{\mathrm{stab}}(x_1)\leq \mathbf{r}_{\mathrm{stab}}(x_0)\bar{K}^{-1}$, $\mathbf{r}_{\mathrm{stab}}(x_2) \leq \mathbf{r}_{\mathrm{stab}}(x_1)\bar{K}^{-1}$,
\item \begin{align*}
& B(x_0,1.1\lambda \mathbf{r}_{\mathrm{stab}}(x_0)) \cap B(x_2,\mathbf{r}_{\mathrm{stab}}(x_2)) \neq \varnothing, \\
& B(x_1,6\lambda \mathbf{r}_{\mathrm{stab}}(x_1) ) \cap B(x_2,3\lambda \mathbf{r}_{\mathrm{stab}}(x_2)) \neq \varnothing,
\end{align*}
\item $\Theta_g(x_2, 10\lambda K \mathbf{r}_{\mathrm{stab}}(x_1)) - \Theta_g(x_2, \frac{1}{3K} \mathbf{r}_{\mathrm{stab}}(x_1)) \leq \beta_1^{-1}$,
\item $x_1$ is not contained in the almost conical region $\mathfrak{C}$.
\end{itemize}
Then there are two balls $b_1,b_2$ of the form $B(y,\mathbf{r}_{\mathrm{stab}}(y))$ such that 
\begin{enumerate} [label=(\roman*)]
\item $3\lambda b_1\cap 3\lambda b_2\neq \varnothing$, 
\item the radii of $b_1,b_2$ belong to 
$$[2^{-({R}+1)}  \mathbf{r}_{\mathrm{stab}}(x_1), 2^{-R} \mathbf{r}_{\mathrm{stab}}(x_1)) \text{ for some  $R\in [1000, \bar{R}]$},$$
\item $3\lambda b_1\cup 3\lambda b_2 \subset B(x_0,1.7\lambda \mathbf{r}_{\mathrm{stab}}(x_0))$. 
\end{enumerate}
\end{coro}

\begin{proof}

Recall that $\bar{K}:=30\lambda K^2$. We apply Theorem \ref{noconcentration} by rescaling and replacing $\mathbf{r}_{\mathrm{stab}}(x_1)$ (resp. $x_1$, $x_2$) with $1$ (resp. $y'$, $0$). We need to check that:
\begin{itemize}
\item $\bar{K}\mathbf{r}_{\mathrm{stab}}(x_1) <\bar{r}$, this comes from the fact that $\bar{K}\mathbf{r}_{\mathrm{stab}}(x_1)\leq \frac{\bar{K} \mathbf{r}_{\mathrm{stab}}(x_0)}{\bar{K}} \leq \mathbf{r}_{\mathrm{stab}}(x_0)<\bar{r}$;
\item the $n$-volume of $\Sigma\cap B(x_2,\bar{K}\mathbf{r}_{\mathrm{stab}}(x_1))$ is at most $\bar{A}(\bar{K}\mathbf{r}_{\mathrm{stab}}(x_1))^n/2$ but this follows from the area bound, the definition of $\bar{A}$ (see Subsection \ref{almostflat2}) and the previous point;
\item $\mathbf{r}_{\mathrm{stab}}(x_2) < \frac{\mathbf{r}_{\mathrm{stab}}(x_1)}{\bar{K}}$ but this is true by assumption of the lemma;
\item $\dist_g(x_1,x_2) < 7\lambda \mathbf{r}_{\mathrm{stab}}(x_1)$, this is because one of the assumptions give $\dist_g(x_1,x_2)\leq 6\lambda \mathbf{r}_{\mathrm{stab}}(x_1) + 3\lambda \mathbf{r}_{\mathrm{stab}}(x_2)\leq (6\lambda +\frac{3\lambda}{\bar{K}})\mathbf{r}_{\mathrm{stab}}(x_1)< 7\lambda \mathbf{r}_{\mathrm{stab}}(x_1)$.
\end{itemize}
Thus we can indeed apply Theorem \ref{noconcentration}. Suppose for a moment that $(ii)$ of Theorem \ref{noconcentration} occurs and that
$$\Sigma\cap A\big(x_2,\frac{\mathbf{r}_{\mathrm{stab}}(x_1)}{2K}, 7\lambda K\mathbf{r}_{\mathrm{stab}}(x_1)\big)$$ is $\delta$-close to a cone $\Gamma\subset \mathcal{G}_{\beta_0}$. As 
$$\mathbf{r}_{\mathrm{stab}}(x_2)<\frac{\mathbf{r}_{\mathrm{stab}}(x_1)}{\bar{K}} < \frac{\mathbf{r}_{\mathrm{stab}}(x_1)}{2K} < \frac{1}{2} 7\lambda K\mathbf{r}_{\mathrm{stab}}(x_1) <  \frac{1}{2} \bar{K}\mathbf{r}_{\mathrm{stab}}(x_1) < \frac{1}{2} \bar{r}$$ it follows
that 
 $ A\big(x_2,\frac{\mathbf{r}_{\mathrm{stab}}(x_1)}{2K}, 7\lambda K\mathbf{r}_{\mathrm{stab}}(x_1)\big)$ is a pointed $\delta$-conical annulus and by definition
$$A\big(x_2,\frac{\mathbf{r}_{\mathrm{stab}}(x_1)}{2},7\lambda \mathbf{r}_{\mathrm{stab}}(x_1)\big) \subset \mathfrak{C}.$$
 However, since $B(x_1,6\lambda \mathbf{r}_{\mathrm{stab}}(x_1) )$ intersects $B(x_2,3\lambda \mathbf{r}_{\mathrm{stab}}(x_2))$ and $\mathbf{r}_{\mathrm{stab}}(x_2)<\frac{\mathbf{r}_{\mathrm{stab}}(x_1)}{\bar{K}}$, 
$$\dist_g(x_1,x_2) \leq 6\lambda \mathbf{r}_{\mathrm{stab}}(x_1) +3\lambda \mathbf{r}_{\mathrm{stab}}(x_2)\leq (6\lambda +\frac{3\lambda}{\bar{K}})\mathbf{r}_{\mathrm{stab}}(x_1) \leq 7\lambda \mathbf{r}_{\mathrm{stab}}(x_1)$$ 
$$ \text{ and } \dist_g(x_1,x_2)\geq \mathbf{r}_{\mathrm{stab}}(x_1) - \lambda \mathbf{r}_{\mathrm{stab}}(x_2) \geq  \mathbf{r}_{\mathrm{stab}}(x_1)(1-\frac{\lambda}{\bar{K}})\geq \mathbf{r}_{\mathrm{stab}}(x_1)/2 .$$
 We are assuming that $x_1$ is not in $\mathfrak{C}$, so we conclude that actually $(i)$ of Theorem \ref{noconcentration} occurs here and we obtain two balls $b_1,b_2$ with radii belonging to 
 $$[2^{-({R}+1)}\mathbf{r}_{\mathrm{stab}}(x_1),2^{-{R}}\mathbf{r}_{\mathrm{stab}}(x_1)) \text{ for some integer $R\in [1000, \bar{R}]$},$$ whose centers are in $B(x_2, \frac{\bar{K}}{2} \mathbf{r}_{\mathrm{stab}}(x_1))$. To finish the proof of our lemma, we have to finally check that $$3\lambda b_1\cup 3\lambda b_2\subset B(x_0, 1.7\lambda \mathbf{r}_{\mathrm{stab}}(x_0)).$$ For that, we write 
\begin{align*}
\dist_g(x_0, \text{ any} & \text{ point of $3\lambda b_1$ or $3\lambda b_2$}) \\ & \leq 1.1\lambda \mathbf{r}_{\mathrm{stab}}(x_0) +\mathbf{r}_{\mathrm{stab}}(x_2) +\frac{\bar{K}}{2}\mathbf{r}_{\mathrm{stab}}(x_1) + 3\lambda  2^{-1000}\mathbf{r}_{\mathrm{stab}}(x_1) \\
 & \leq 1.1\lambda \mathbf{r}_{\mathrm{stab}}(x_0) + (\frac{1}{\bar{K}}  + \frac{\bar{K}}{2} + 3\lambda 2^{-1000}) \mathbf{r}_{\mathrm{stab}}(x_1) \\
 &  \leq 1.1\lambda \mathbf{r}_{\mathrm{stab}}(x_0)+ \frac{1}{\bar{K}}(\frac{1}{\bar{K}}  + \frac{\bar{K}}{2} + 3\lambda) \mathbf{r}_{\mathrm{stab}}(x_0) \\
 & \leq 1.1\lambda \mathbf{r}_{\mathrm{stab}}(x_0) + 0.6 \lambda \mathbf{r}_{\mathrm{stab}}(x_0) = 1.7 \lambda \mathbf{r}_{\mathrm{stab}}(x_0).
 \end{align*}
\end{proof}

\begin{remarque} \label{pourquo} By the proof of Theorem \ref{noconcentration}, in Corollary \ref{sotechnical} instead of $(ii)$, it is not hard to make sure that if $\mathbf{r}_{\mathrm{stab}}(x_1) \in [2^{-(k+1)},2^{-k}) $ for some integer $k$ then the radii of $b_1,b_2$ are in fact in $[2^{-(k+R+1)}, 2^{-(k+R)})$ for some integer $R\in [1000,\bar{R}]$.
\end{remarque}

\vspace{1em}

\subsection{Proof of Theorem \ref{maina} in dimensions $4$ to $7$} \label{splash}
Let $\Sigma\subset (M^{n+1},g)$ be a closed embedded minimal hypersurface, where $2\leq n\leq 6$, whose $n$-volume is bounded by $A>0$.
Let $F$ be a field. The Betti numbers $b^k(.)$ of a set are the dimensions over $F$ of the cohomology groups $H^{k}(.,F)$.
Denote by $\sharp$ the number of elements of a set.

Recall that the constants $\bar{r}$, $\lambda$, $\bar{A}$, $\delta$, $K$, the functions $\mathbf{f}(\Sigma)$, $\mathbf{r}_{\mathrm{stab}}$, the almost conical region $\mathfrak{C}$ were defined in Subsections \ref{folding}, \ref{stabrad}, \ref{almostflat2}, \ref{noconcentr}. 
Rescaling the metric by a large constant, we can of course assume 
$$\bar{r} = \frac{1}{2}.$$ 


We first look for a special cover of the closed subset $\Sigma\backslash \mathfrak{C}$. By Besicovitch's covering lemma, we can extract from $\{B(x,\mathbf{r}_{\mathrm{stab}}(x)); x\in \Sigma\backslash \mathfrak{C}\}$ some subfamilies of balls 
$$\mathcal{B}^{(1)},...,\mathcal{B}^{(c_1)}\subset \{B(x,\mathbf{r}_{\mathrm{stab}}(x)); x\in \Sigma\backslash \mathfrak{C}\}$$
such that each $\mathcal{B}^{(i)}$ consists of disjoint balls and 
\begin{equation}\label{goodcover}
\Sigma \backslash \mathfrak{C}  \subset \bigcup_{i=1}^{c_1}\bigcup_{b\in \mathcal{B}^{(i)}} b,
\end{equation}
where $c_1$ is an integer depending only on $(M,g)$.

Define the \emph{overlap} of a family of open sets $\{U_i\}$ in a manifold to be the maximal number of open sets of $\{U_i\}$ intersecting non-trivially a single element of $\{U_i\}$, namely
\begin{align*}
\text{overlap}(\{U_i\}) := \max\{P; \quad & \text{there are different $i_0,i_1,...,i_P$ such that} \\
& \text{ $U_{i_0} \cap U_{i_k} \neq \varnothing$ for all $k\in \{1,...,P\}$}\}.
\end{align*}

By (\ref{goodcover}) and (\ref{muuu}),
we obtain a family $\{b\}$ of open balls of the form $B(x,\mathbf{r}_{\mathrm{stab}}(x))$ covering the closed subset $\Sigma\backslash \mathfrak{C}$ with overlap at most $\mu'$ for some $\mu'=\mu'(M,g,\bar{r},\lambda)>0$. Recall that $\mathfrak{C}$ is a disjoint union of $(\delta,K)$-telescopes $T_1,...,T_L$. Because of Lemma \ref{almostflatT} $(iii) (iv)$, by adding $T_1,...,T_L$ to the family $\{b\}$, we obtain a larger family of open sets $\{b'\}$ with overlap still bounded by $\mu'$ (which may be taken larger). Recall that
$$\sharp \{b\} = \sum_{i=1}^{c_1} \sharp \mathcal{B}^{(i)}.$$
By Lemma \ref{almostflatT} $(iv)$, the number of $(\delta,K)$-telescopes $T_i$ is bounded by  $ 2\sum_{i=1}^{c_1} \sharp \mathcal{B}^{(i)} $.

By the curvature bound (\ref{ssbound}) and our choice of $\varepsilon_{\mathrm{SS}}$ in Subsection \ref{stabrad}, for all $x\in M$, the intersection $\Sigma\cap B(x, \mathbf{r}_{\mathrm{stab}}(x))$  is a union of convex $n$-disks. In particular, any finite intersection of such disks is contractible. Recall by Lemma \ref{almostflatT} $(i)$ that for each $(\delta,K)$-telescopes $T_i$, $\Sigma \cap T_i$ is diffeomorphic to $(\partial_+T_i \cap \Sigma) \times (0,1)$ where $\partial_+ T_i$ is one of the two components of $\partial T_i = \partial_+ T_i \cup \partial_- T_i$. Each component of $\partial_+T_i \cap \Sigma$ is diffeomorphic to the cross-section of a cone in $\mathcal{G}_{\beta_0}$ so the sum of the Betti numbers of  a connected component of $T_i \cap \Sigma$ is bounded from above by a number depending only on $(M,g), \lambda, \bar{r}, A$. Moreover by Lemma \ref{almostflatT} $(ii)$, intersections of the form
$$\Sigma \cap T_i \cap \bigcap_{j=1}^J B(x_j,\mathbf{r}_{\mathrm{stab}}(x_j))$$
are unions of $n$-disks.

\begin{definition} \label{sheet}
Let $N(\Sigma)$ be the number defined by 
$$N(\Sigma):=\max\{N; \exists p\in M, \Sigma\cap B(p,\mathbf{r}_{\mathrm{stab}}(p)) \text{ has $N$ components}\}.$$
We call this number the (stable) sheeting number.
\end{definition}

By the monotonicity formula, since $\Sigma$ has area bounded above by $A>0$, there is a positive constant $c_2$ only depending on $(M,g)$, $\lambda, \bar{r}, $ and $A$ such that the sheeting number satisfies 
\begin{equation} \label{gloo}
N(\Sigma)  \leq c_2.
\end{equation}
Similarly each $\Sigma \cap T_i$ admits a number of components bounded by $c_2$ by the monotonicity formula and properties of $\mathfrak{C}$.

From those remarks, the connected components of the intersections $\Sigma \cap b'$ (where $\{b'\}$ is as above) form a cover of $\Sigma$ by open $n$-dimensional subsets $\{U_j\}$ contained in $\Sigma$ with overlap at most $\mu'$, such that any $U_j$ has Betti numbers bounded independently of $\Sigma$, such that
$$\sharp\{U_j\} \leq c_3 \sum_{i=1}^{c_1} \sharp \mathcal{B}^{(i)} $$
for a constant $c_3$ only depending on $(M,g)$, $\lambda$, $\bar{r}$, $A$,
and such that any finite intersection of more than one element $U_{j_1}\cap...\cap U_{j_l}$ is contractible or empty.  
We can suppose that $\sharp \mathcal{B}^{(1)}$ is at least as large as any other $\sharp \mathcal{B}^{(i)}$.
Now we apply the topological result Lemma \ref{lemmetopo} in the Appendix to estimate the Betti numbers of $\Sigma$ in terms of the number of open sets $U_j \subset \Sigma$: there is $c_4$ depending on $(M,g), \lambda, \bar{r}, A$ but not on $\Sigma$ such that

\begin{equation} \label{sheetinequality}
\sum_{i=0}^n b^i(\Sigma) \leq   c_4 \sharp \mathcal{B}^{(1)}.
\end{equation}

In view of this inequality, the goal of the end of our proof is to show that there is $c_5$ depending only on $(M,g)$, the area bound $A$, and $\lambda, \bar{r},K, \delta, \bar{A}$ but independent of $\Sigma$ so that
\begin{equation}\label{goalbis}
\sharp \mathcal{B}^{(1)} \leq c_5 (\mathbf{f}(\Sigma)+1),\end{equation}
since the folding number satisfies
$\mathbf{f}(\Sigma)\leq \ind(\Sigma)$ (see Subsection \ref{fofo}). Inequality (\ref{goalbis}) is enough to conclude the proof of Theorem \ref{maina} and we postpone its proof to the next Subsection.

\vspace{1em}

\subsection{The combinatorial argument} \label{combinatorial}

In this subsection, we present the main combinatorial argument which enables to prove Inequality (\ref{goalbis}) (see Corollary \ref{C:counting} and its application).

In what follows, we are assuming that $\bar{r} = \frac{1}{2}$. For any positive integer $k$, let $\mathcal{F}_k $ be the set of all geodesic balls $B(p,r)$ in $M$ of radius $r\in[2^{-(k+1)},2^{-k})$.

For each $k\geq 1$, let $$\mathcal{B}'_k\subset \mathcal{F}_k$$ be any family of balls of the form $B(x,\mathbf{r}_{\mathrm{stab}}(x))$ which are \textit{strongly disjoint}, in the sense that
for any integer $k$, and any two different balls $b,b'\in \mathcal{B}'_k$, 
\begin{equation}\label{ravel}
3\lambda b\cap  3\lambda b'=\varnothing.
\end{equation}

\vspace{1em}

\subsubsection{ Definition of basis}
Recall that, as outlined in Section  \ref{first}, we are about to apply an induction construction to $\bigcup_{k\geq 1} \mathcal{B}'_k$. At each step we want to discard some balls and add new ones in the family, by applying the so-called Property P[J] (see below).  It is useful to introduce a notion of basis for $\bigcup_{k\geq 1} \mathcal{B}'_k$; it is roughly defined as the subfamily of balls being locally the largest. Each step of the inductive construction will then be applied to a basis.

 If $\bigcup_{k\geq 1}\mathcal{B}'_k$ satisfies (\ref{ravel}), we define a \emph{basis} $\mathbf{B}$ of $\bigcup_{k\geq 1}\mathcal{B}'_k$ as follows. We proceed by induction, first we impose $A_1:=\mathcal{B}'_1$ to be included in $\mathbf{B}$. Then we choose a maximal subfamily $A_2$ of $\mathcal{B}'_2$ such that for any $b\in \mathcal{B}'_1$, $b'\in A_2$, we have $3\lambda b \cap 3\lambda b' =\varnothing$. Next we choose similarly a maximal subfamily $A_3\subset \mathcal{B}'_3$ such that for any $b\in  \mathcal{B}'_1\cup \mathcal{B}'_2$, $b'\in A_3$, we have $3\lambda b \cap3 \lambda b' =\varnothing$, etc. The union  $\mathbf{B} :=\bigcup_{k\geq 1}A_k$ forms what we call a basis of $\bigcup_{k\geq 1}\mathcal{B}'_k$. In other words, for any $i$, $\mathbf{B}\cap \mathcal{B}_i$ is a maximal subset of $\mathcal{B}_i$ such that for any $j<i$:
$$\forall b \in \mathbf{B}\cap \mathcal{B}_i, \forall b'\in \mathcal{B}_j, 3\lambda b\cap 3\lambda b' = \varnothing.$$

\vspace{1em}

\subsubsection{Definition of Property P[J]}

It is useful to introduce the following property for $(\bigcup_{k\geq 1}\mathcal{B}'_k,\mathbf{B})$.  \newline

       
\vspace{1em}

\textbf{Property P[J]:} Let $J$ be a  positive integer, let $\bigcup_{k\geq 1}\mathcal{B}'_k$ be a family satisfying (\ref{ravel}) and let $\mathbf{B}$ be a basis of $\bigcup_{k\geq 1}\mathcal{B}'_k$. 

We say that $(\bigcup_{k\geq 1}\mathcal{B}'_k,\mathbf{B})$ satisfies Property {P[J]} if for any $k\geq1$, any $B\in \mathbf{B} \cap \mathcal{B}'_k$, one of the following occurs:
\begin{enumerate}
\item either the size of the set 
$$\{\hat{b}\in \mathcal{B}'_{k+u} ; u\geq0, 3\lambda \hat{b} \cap \lambda B \neq \varnothing\}$$ is bounded above by $J$,
\item or there are $v=v(B)>1000$, and two balls $b_B := B(x,\mathbf{r}_{\mathrm{stab}}(x))\in \mathcal{F}_{k+v}$, $b'_B:=B(x',\mathbf{r}_{\mathrm{stab}}(x'))\in \mathcal{F}_{k+v}$  (where $x,x'\in M$) with 
$$3\lambda b_B \cap 3\lambda b'_B =\varnothing\text{ and }3\lambda b_B \cup 3\lambda b'_B \subset 2\lambda B.$$ Moreover 
the size of $$\{\hat{b}\in \mathcal{B}'_{k+u} ; 0\leq u\leq v, 3\lambda \hat{b} \cap (3\lambda b_B\cup 3\lambda b'_B) \neq \varnothing\}$$ is bounded above by $J$. \newline

\end{enumerate}

\vspace{1em}

\subsubsection{Property P[J] outside of $\mathfrak{C}$}
We check that for any family $\bigcup_{\geq1}\mathcal{B}'_k$ satisfying (\ref{ravel}) and with basis $\mathbf{B}$, if the center of any ball in $\bigcup_{\geq1}\mathcal{B}'_k\backslash \mathbf{B}$ avoids the almost conical region $\mathfrak{C}$, then $(\bigcup_{\geq1}\mathcal{B}'_k, \mathbf{B})$ has Property P[J] for an integer $J$ that does not depend on the minimal hypersurface $\Sigma$.

\begin{prop} \label{triplybis}
Let $3\leq n+1\leq 7$. There is a positive integer $J$ depending only on $(M^{n+1},g)$, on $A>0$ and on the constants  $\bar{r}$, $\lambda$, $\bar{A}$, $\delta$, $K$, $\bar{K}$, such that the following is true. 
Let $\Sigma\subset (M,g)$ be a closed embedded minimal hypersurface of $n$-volume at most $A$. For each $k\geq 1$, let $\mathcal{B}'_k\subset \mathcal{F}_k$ be a family of balls of the form $B(x,\mathbf{r}_{\mathrm{stab}}(x))$ satisfying (\ref{ravel}),
and let $\mathbf{B}$ be a basis of $\bigcup_{k\geq1}\mathcal{B}'_k$. Assume additionally that any ball $b=B(x,\mathbf{r}_{\mathrm{stab}}(x))\in \bigcup_{k\geq1}\mathcal{B}'_k \backslash \mathbf{B}$ is centered at $x\in M\backslash \mathfrak{C}$.

Then $(\bigcup_{\geq1}\mathcal{B}'_k,\mathbf{B})$ satisfies Property P[J].
\end{prop}

\begin{proof} 
Recall the definitions of the constants $\bar{r} =  \frac{1}{2}$, $\bar{A}$, $\lambda$, $\delta$,  $K>2^{1000}$, $\bar{K} = 30\lambda K^2$ (Subsections  \ref{stabrad}, \ref{almostflat2}, \ref{noconcentr}).
Before choosing the integer $m$, we make the following elementary remark: for any point $p\in (M,g)$, at scale smaller than $\bar{r}$, the volume ratio $\Theta_g(x,r)$ is not exactly nondecreasing with respect to the radius $r$ in general, but it is almost non-increasing by the monotonicity formula: without loss of generality we can rescale the metric $g$ with a large enough constant depending only on $(M,g)$, and the quantity $(1+cr)\Theta_g(x,r)$ is nondecreasing at scale less than $\bar{r}$ for a universal constant $c$. 
Now fix $m$ a large integer depending on $\bar{A}$, so that for any $p\in M$, and any radii $r_1,...,r_{m-1},r_{m} \in (0,\bar{r})$ with $r_i> \bar{K} r_{i+1}$ for $i=1,...,m-1$, there exists $j\in\{1,...,m-1\}$ such that
$$\Theta_g(p,10\lambda K r_j) -\Theta_g(p,\frac{1}{3K}r_j) \leq 1/\beta_1.$$
where $\beta_1$ is given in Theorem \ref{noconcentration}. The existence of such $m$ follows from almost monotonicity and the pigeon hole principle.

\vspace{1em}

Let $\Sigma$, $\bigcup_{\geq1}\mathcal{B}'_k$, $\mathbf{B}$ be as in the statement. Let $B\in \mathbf{B}\cap \mathcal{B}'_k$ for some $k\geq 1$. We denote the radius of a ball by $\rad$. The following simple facts are similar. \newline

\textbf{Fact 1:} For any $1\leq k'\leq k''$, for any ball $X \in \mathcal{F}_{k''}$, there is an upper bound $\Upsilon_1>0$, depending only on $(M,g)$ and the constants, for the size of any family of balls $\{ \hat{b}_i\}\subset \mathcal{B}'_{k'}$ such that
$$\forall i,\quad 3\lambda \hat{b}_i \cap  3\lambda X\neq \varnothing.$$ \newline

\textbf{Fact 2:} There is an upper bound $\Upsilon_2>0$, depending only on $(M,g)$ and the constants, for the size of any family of balls $\{\hat{b}_i\}$ such that 
$$\forall i, \quad \hat{b}_i  \in \mathcal{B}'_{k+u}, 0\leq u \leq \log_2(2\bar{K}), \quad  3\lambda \hat{b}_i \cap  \lambda B\neq \varnothing.$$ 
\newline




Set $$J:=2002\Upsilon_1+ 2m(\log_2(2\bar{K})+1)\Upsilon_1 +2\bar{R}\Upsilon_1 +\Upsilon_2,$$
where $\bar{R}>1000$ is given in Theorem \ref{noconcentration}.

We can suppose that there is an integer $v_1> \log_2(2\bar{K})$ and a ball $\bar{b} \in \mathcal{B}'_{k+v_1}$ with $3\lambda\bar{b} \cap \lambda B\neq \varnothing$, since otherwise Case (1) of Property P[J] is true by Fact 2. We can assume $v_1$ to be the smallest integer satisfying these properties.
\newline

\textbf{Case A.} First suppose that the number of $u> 0$ such that there is a $\hat{b}\in \mathcal{B}'_{k+u}$ with $3\lambda\hat{b}\subset (2\lambda -2^{-1000})B$ is larger than $ m(\log_2(2\bar{K})+1)$. Here $(2\lambda -2^{-1000})B$ is the ball sharing the same center with $B$ but with radius $2\lambda-2^{-1000}$ times larger.

Write for convenience $k=u_0$. We can find integers $u_1,...,u_m$ with 
\begin{itemize}
\item $u_{i} > u_{i-1}+\log_2(2\bar{K})$ for $i=1,...,m$,
\item there is $\hat{b}_i \in \mathcal{B}'_{k+u_i}$ with $3\lambda \hat{b}_i \subset  (2\lambda-2^{-1000})B$,
\item for $u\in(u_{i-1}+\log_2(2\bar{K}) , u_i)$, there is no $\hat{b} \in \mathcal{B}'_{k+u}$ with $3\lambda \hat{b} \subset  (2\lambda-2^{-1000})B$. 
\end{itemize}
Note that, since for $u\geq \log_2(2 \bar{K})$, any $\hat{b}\in \mathcal{B}'_{k+u}$ with $3\lambda\hat{b}\cap \lambda B\neq \varnothing$ is such that $3\lambda\hat{b}\subset(2\lambda-2^{-1000})B$,
 we have $v_1\geq u_1$. Note also that
$$\rad(\hat{b}_1) 
< \frac{\rad(B)}{\bar{K}} \text{ and }$$
$$\rad(\hat{b}_{i}) < \frac{\rad(\hat{b}_{i-1})}{\bar{K}}  \quad \text{ for all } i=2,...,m.$$ \newline


Case A.1.\newline 
Suppose that for some $j\in\{1,...,m\}$,
$$\hat{b}_j \cap 1.1\lambda B=\varnothing.$$
Let $x$ (resp. $y$) be the center of $\hat{b}_j$ (resp. of the ball $\bar{b} \in \mathcal{B}'_{k+v_1}$ defined earlier). By definition of $\bar{b}$, the distance between the center of $B$ and the center $y$ of $\bar{b}$ is at most $\lambda \rad(B) + \frac{3\lambda}{\bar{K}} \rad(B)$. As for $\hat{b}_j$, the distance between the center of $B$ and the center $x$ of $\hat{b}_j$ is at least $1.1\lambda\rad(B)$. Thus the distance between $x$ and $y$ is at least $(0.1 - \frac{3}{\bar{K}})\lambda \rad(B) >6\lambda 2^{-1000} \rad(B)$ and there exists a continuous path $\gamma_x:[0,1]\to (2\lambda-2^{-1000})B$ (resp. $\gamma_y:[0,1]\to (2\lambda-2^{-1000})B$) such that $\gamma_x(0) = x$ (resp. $\gamma_y(0) = y$), the endpoints $\gamma_x(1)$, $\gamma_y(1)$ are in $B$ and $\dist_g(\gamma_x,\gamma_y) > 6\lambda 2^{-1000} \rad(B)$. 
Observe that by definition 
$$\mathbf{r}_{\mathrm{stab}}(\gamma_x(1))> \frac{\lambda-1}{\lambda} \rad(B) > 2^{-1000}\rad(B) $$
$$ \text{ and similarly } \mathbf{r}_{\mathrm{stab}}(\gamma_y(1))>2^{-1000}\rad(B),$$
while
$$\mathbf{r}_{\mathrm{stab}}(\gamma_x(0)) =\mathbf{r}_{\mathrm{stab}}(x) = \rad(\hat{b}_j) < \frac{\rad(B)}{\bar{K}} < 2^{-1000}\rad(B)$$
$$\text{ and similarly } \mathbf{r}_{\mathrm{stab}}(\gamma_y(0))<2^{-1000}\rad(B).$$
Then by Lemma \ref{continui} we can get two balls $b_B=B(x',\mathbf{r}_{\mathrm{stab}}(x')),b'_B=B(y',\mathbf{r}_{\mathrm{stab}}(y'))\subset 2\lambda B$ (with $x'\in \gamma_x([0,1]), y'\in \gamma_y([0,1])$) such that $b_B,b'_B\in \mathcal{F}_{k+1000}$ (in fact their radii $\mathbf{r}_{\mathrm{stab}}(x), \mathbf{r}_{\mathrm{stab}}(x')$ can be taken to be equal to $2^{-1000}\rad(B)\in [2^{-(k+1001)}, 2^{-(k+1000)})$), 
and
$$3\lambda b_B\cap 3\lambda b'_B =  \varnothing, \quad 3\lambda b_B\cup 3\lambda b'_B \subset 2\lambda B.$$
Moreover, the size of $$\{\hat{b} \in \mathcal{B}'_{k+u} ; 0\leq u\leq 1000, 3\lambda \hat{b} \cap(3\lambda b_B\cup 3\lambda b'_B) \neq \varnothing\}$$
is bounded above by $$1001.(2\Upsilon_1)$$ (see Fact 1)
so Case (2) of Property P[J] is satisfied.\newline

Case A.2.\newline 
Suppose that for all $i\in\{1,...,m\}$, 
$$\hat{b}_i \cap 1.1\lambda B\neq \varnothing.$$
Assume also that there is $j\in\{1,...,m-1\}$ such that 
$$6\lambda\hat{b}_j\cap 3\lambda\hat{b}_m = \varnothing.$$
We use an argument similar to the one in the previous paragraph. 
Let $x$ (resp. $y$) 
be the center of $\hat{b}_m$ (resp. $\hat{b}_j$). 
There exists a continuous path $\gamma_x:[0,1]\to (1.1\lambda+2^{-1000})B$ such that $\gamma_x(0) = x$, the endpoint $\gamma_x(1)$ is in $B$ and $\dist_g(\gamma_x, y) > 6\lambda \rad(\hat{b}_j) $. 
We have
$$\mathbf{r}_{\mathrm{stab}}(\gamma_x(1))> \frac{\lambda-1}{\lambda} \rad(B) > \rad(\hat{b}_j)$$
while 
$$\mathbf{r}_{\mathrm{stab}}(\gamma_x(0))<\rad(\hat{b}_j).$$
Then by Lemma \ref{continui} we can get a ball $b_B=B(x',\mathbf{r}_{\mathrm{stab}}(x')) \subset (1.1\lambda + 2^{-999})B$ (with $x'\in \gamma_x([0,1])$) such that 
$$\mathbf{r}_{\mathrm{stab}}(x') = \mathbf{r}_{\mathrm{stab}}(y) = \rad(\hat{b}_j) \in [2^{-(k+u_j+1)},2^{-(k+u_j)})$$
and if we set $b'_B:=\hat{b}_j$,
$$3\lambda b_B\cap 3\lambda b'_B = \varnothing, \quad 3\lambda b_B\cup 3\lambda b'_B \subset 2\lambda B.$$
Moreover, since $b_B,b'_B \subset (1.1\lambda + 2^{-999})B$, if a ball $Y$ has radius less than $2^{-(k+1000)}$ and $3\lambda Y\cap (3\lambda b_B \cup 3\lambda b'_B) \neq \varnothing$, then $3\lambda Y\subset (2\lambda-2^{-1000})B$.
Thus by Fact 1 and how we defined the $u_i$'s, the size of 
$$\{\hat{b} \in \mathcal{B}'_{k+u} ; 0\leq u\leq u_j, 3\lambda \hat{b} \cap(3\lambda b_B\cup 3\lambda b'_B) \neq \varnothing\}$$
is bounded above for instance by $$1001.(2\Upsilon_1) + m(\log_2(2\bar{K})+1).(2\Upsilon_1)$$
so Case (2) of Property P[J] is satisfied.\newline

Case A.3.\newline 
Suppose again that for all $i\in\{1,...,m\}$, 
$$\hat{b}_i \cap 1.1\lambda B\neq \varnothing$$
but this time assume additionally that for all $i \in \{1,...,m\}$, 
$$6\lambda\hat{b}_i\cap 3\lambda\hat{b}_m\neq\varnothing.$$
Let $y_i$ be the center of $\hat{b}_i$. For $i<i'$:
$$ \mathbf{r}_{\mathrm{stab}}(y_i)> \bar{K}\mathbf{r}_{\mathrm{stab}}(y_{i'}) .$$
By the choice of $m$, 
there exists $j\in\{1,...m-1\}$, 
$$\Theta_g(y_m,10\lambda K \mathbf{r}_{\mathrm{stab}}(y_j)) -\Theta_g(y_m,\frac{1}{3K}\mathbf{r}_{\mathrm{stab}}(y_j)) \leq 1/\beta_1.$$
Since $B$ is an element of the basis $\mathbf{B}$ and $\hat{b}_j \cap 1.1\lambda B\neq \varnothing$, $\hat{b}_j$ is not in this basis so it is centered outside of the almost conical region $\mathfrak{C}$ by assumption. We can thus apply Corollary \ref{sotechnical} and Remark \ref{pourquo} with $x_0$ being the center of $B$ and $x_1=y_j$, $x_2=y_m$.
There are two balls $b_1,b_2$ of the form $B(x,\mathbf{r}_{\mathrm{stab}}(x))$ with 
$$3\lambda b_1\cap 3\lambda b_2 =\varnothing, \quad 3\lambda b_1 \cup 3\lambda b_2\subset 1.7\lambda B,$$
such that $b_1,b_2\in \mathcal{F}_{k+u_j+R}$ for some integer $R\in [ 1000, \bar{R}]$ (here $\bar{R}$ is given in Theorem \ref{noconcentration}).
Note that if a ball $Y$ has radius less than $2^{-(k+1000)} $ and $3\lambda Y\cap (3\lambda b_1 \cup 3\lambda b_2) \neq \varnothing$, then $3\lambda Y\subset (2\lambda-2^{-1000})B$. By Fact 1 and how we defined the $u_i$'s, the size of 
$$\{\hat{b} \in \mathcal{B}'_{k+u} ; 0\leq u\leq u_j+R, 3\lambda \hat{b} \cap(3\lambda b_1\cup 3\lambda b_2) \neq \varnothing\}$$
is bounded above for instance by 
$$1001 .(2\Upsilon_1)  + m(\log_2(2\bar{K})+1).(2\Upsilon_1)  +\bar{R}.(2\Upsilon_1).$$
So Case (2) of Property P[J] is satisfied.\newline

\textbf{Case B.} Secondly, suppose that the number of indices $u>0$ such that there is a $\hat{b}\in \mathcal{B}'_{k+u}$ with $3\lambda\hat{b}\subset (2\lambda-2^{-1000})B$ is bounded above by $ m(\log_2(2\bar{K})+1)$. Let $u_1 < u_2<... < u_L$ be these indices, in particular $L\leq m(\log_2(2\bar{K})+1)$. \newline

Case B.1.\newline 
Suppose that for some $j\in\{1,...,L\}$, there is $\hat{b}_j\in \mathcal{B}'_{k+u_j}$ with
$$u_j\geq 1000,\quad \hat{b}_j \cap 1.1\lambda B=\varnothing.$$
Then by the arguments of Case A.1. in this proof, the second case of Property P[J] is satisfied.\newline

Case B.2.\newline 
Suppose that for all $j\in\{1,...,L\}$ with $u_j\geq 1000$, and for all $\hat{b}_j\in \mathcal{B}'_{k+u_j}$ with $\hat{b}_j\subset (2\lambda-2^{-1000})B$, one has
$$\hat{b}_j \cap 1.1\lambda B\neq \varnothing.$$
If for a certain $i\in \{1,...,L\}$ with $u_i\geq 1000$, there are two different balls $b_1, b_2\in \mathcal{B}'_{k+u_i}$ with 
$$3\lambda b_1\cap 3\lambda b_2= \varnothing, \quad b_1 \cap 1.1 \lambda B \neq \varnothing, \quad b_2 \cap 1.1\lambda B\neq \varnothing,$$
then we can set
$$b_B = b_1, b'_B = b_2. $$
Then clearly $3\lambda b_B\cup 3\lambda b'_B\subset 1.5 \lambda B\subset 2 \lambda B$. Again, note that if a ball $Y$ has radius less than $2^{-(k+1000)} $ and $3\lambda Y\cap (3\lambda b_1 \cup 3\lambda b_2) \neq \varnothing$, then $3\lambda Y\subset (2\lambda-2^{-1000})B$.
By Fact 1, the size of 
$$\{\hat{b} \in \mathcal{B}'_{k+u} ; 0\leq u\leq u_i, 3\lambda \hat{b} \cap(3\lambda b_B\cup 3\lambda b'_B) \neq \varnothing\}$$
is bounded by 
$$1001 .(2\Upsilon_1)+ m(\log_2(2\bar{K})+1) .(2\Upsilon_1).$$ Case (2) of Property P[J] is then satisfied.\newline

Case B.3.\newline 
Suppose again that for all $j\in\{1,...,L\}$ with $u_j\geq 1000$, and for all $\hat{b}_j\in \mathcal{B}'_{k+u_j}$ with $\hat{b}_j\subset (2\lambda-2^{-1000})B$, one has
$$\hat{b}_j \cap 1.1\lambda B\neq \varnothing.$$
The last case left is when for all $j\in \{1,...,L\}$ with $u_j\geq 1000$, there is at most one ball  $b_1\in \mathcal{B}'_{k+u_j}$ with 
$$\quad b_1 \cap 1.1 \lambda B \neq \varnothing.$$
Then by Fact 2, the size of 
$$\{ \hat{b} \in \mathcal{B}'_{k+u}; u\geq 0, 3\lambda \hat{b}\cap \lambda B \neq \varnothing\}$$
is bounded above by 
$$\Upsilon_2+m(\log_2(2\bar{K})+1)$$ and Case (1) of Property P[J] is then satisfied.

\end{proof}

\vspace{1em}

\subsubsection{The counting argument}

We prove the following proposition, which relies on a counting argument. We are still assuming $\bar{r} = \frac{1}{2}$.

\begin{prop} \label{doublybis} 
Let $\Sigma\subset (M,g)$ be a closed embedded minimal hypersurface of $n$-volume at most $A$. Let  $\bigcup_{k\geq 1}\mathcal{B}'_k$ be a family of balls of the form $B(x,\mathbf{r}_{\mathrm{stab}}(x))$ centered at $x\in M\backslash \mathfrak{C}$, satisfying (\ref{ravel}), and let $\mathbf{B}$ be a basis of  $\bigcup_{k\geq 1}\mathcal{B}'_k$ so that  $(\bigcup_{k\geq 1}\mathcal{B}'_k, \mathbf{B})$ satisfies Property P[J] for an integer $J>0$. Then there exists an integer $\mathcal{K}^{(0)}$ with  
\begin{equation} \label{double}
\mathcal{K}^{(0)}\geq \frac{1}{2J} \sum_{k\geq 1}\sharp\mathcal{B}'_k +\frac{1}{2} \sharp \mathbf{B},
\end{equation}
and a family of balls $\{b_i\}_{i=1}^{\mathcal{K}^{(0)}}$ so that:
\begin{itemize}
\item $\forall j\in\{1,...,{\mathcal{K}^{(0)}}\}$,  $b_j$ is of the form $B(x,\mathbf{r}_{\mathrm{stab}}(x))$ for some $x\in M$ and $\mathbf{r}_{\mathrm{stab}}(x)<\bar{r}$, and $\lambda b_j $ is included in a ball $2\lambda B$ where $B\in \bigcup_{k\geq 1}\mathcal{B}'_k$,
\item 
$\forall i\neq j\in\{1,...,{\mathcal{K}^{(0)}}\}$,  $\dist(\lambda b_i , \lambda b_j) >0.$
\end{itemize}
\end{prop}

\begin{proof}

To find $\{b_i\}_{i=1}^{\mathcal{K}^{(0)}}$ and prove (\ref{double}), we use a top-down inductive argument on the smallest integer $k_1$ such that $\mathcal{B}'_{k_1}\neq \varnothing$. Necessarily, there is an integer $k_0$ such that for any point $x\in M$, $\mathbf{r}_{\mathrm{stab}}(x) \geq 2^{-({k_0}+1)}$ (the minimal hypersurface $\Sigma$ is fixed in this proof). In particular, $k_1\leq k_0$.


 
 
If $k_1=k_0$, i.e. if $\mathcal{B}'_{k_0}$ is the only non-empty $\mathcal{B}'_k$,  then the basis $\mathbf{B}$ is $\mathcal{B}'_{k_0}$ itself. Clearly one can take $\{b_i\}_{i=1}^{\mathcal{K}^{(0)}}$ to be $ \mathcal{B}'_{k_0}$, and
$$\mathcal{K}^{(0)} = \sharp\mathcal{B}'_{k_0} \geq \frac{1}{2J} \sum_{k\geq 1}\sharp \mathcal{B}'_{k}+ \frac{1}{2} \sharp \mathbf{B}.$$

Suppose that the proposition is true for any family $(\bigcup_{k\geq 1} \mathcal{B}'_k,\mathbf{B})$ satisfying the assumptions of the proposition, such that $k_1\in [ k_0-D,k_0]$, where $D$ is a nonnegative integer.  
Let $\bigcup_{k\geq 1} \mathcal{B}'_k$ be a family satisfying the assumptions of the proposition, with basis $\mathbf{B}$, such that $k_1=k_0-(D+1)$.
Let $\mathbf{B}'$ (resp. $\mathbf{B}''$ ) be the balls in $\mathbf{B}$ satisfying Case (1) (resp. satisfying Case (2) but not Case (1)) of Property P[J]. Write $L:=\sharp \mathbf{B}$, $L':=\sharp \mathbf{B}'$, $L'':=\sharp \mathbf{B}''$, so that $L=L'+L''$. For a ball $B\in \mathbf{B}''$, there are two balls $b_B, b'_B\in \mathcal{F}_{k+v}$ as in Case (2) of Property P[J]. 

Set 
$$\mathcal{V}=\{b_B; \quad B\in \mathbf{B}''\} \cup \{b'_B; \quad B\in \mathbf{B}''\},$$
$$\mathcal{W}' := \{A ; \quad \exists B\in \mathbf{B}'\cap \mathcal{F}_k, \exists u\geq 0, A \in \mathcal{B}'_{k+u}, 3\lambda A \cap  \lambda B  \neq \varnothing\},$$
\begin{align*}
\mathcal{W}'' := \{ & A; \quad \exists B\in \mathbf{B}''\cap \mathcal{F}_k \text{ with corresponding } b_B, b'_B\in \mathcal{F}_{k+v},\\
 & \exists u\in \{0,1,...,v\} , A \in \mathcal{B}'_{k+u}, 3\lambda A \cap (3 \lambda b_B \cup 3\lambda b'_B)  \neq \varnothing\}.
\end{align*}
Note that $\mathbf{B}\subset  \mathcal{W}''\cup \mathcal{W}'$ and that Property P[J] implies that $\mathcal{W}'$ (resp. $\mathcal{W}''$) contains at most $JL'$ (resp. $JL''$) 
elements. Define also
$$\mathcal{Z}:= \mathcal{V}\cup 
\bigcup_{k\geq 1}\mathcal{B}'_k \backslash (\mathcal{W}'\cup \mathcal{W}'')\quad \text{and } \mathcal{Z}_k:=\mathcal{Z}\cap \mathcal{F}_k.$$ 
 Let $\mathbf{B}_{\mathcal{Z}}$ be any basis of $\mathcal{Z}=\bigcup_{k\geq 1}\mathcal{Z}_k$. One can check that $\mathcal{V}\subset \mathbf{B}_{\mathcal{Z}}$ by definition of basis, since the elements of $\{3\lambda b; b\in \mathcal{V}\}$ are disjoint and since $\mathcal{Z}\cap \mathcal{W}'' = \varnothing$. Thus, the fact that $(\bigcup_{k\geq 1}\mathcal{Z}_k,\mathbf{B}_{\mathcal{Z}})$ satisfies property P[J] follows from Proposition \ref{triplybis} and our assumptions on $\bigcup_{k\geq 1}\mathcal{B}'_k$. 
Now the key point is that the smallest integer $k'_1$ such that $\mathcal{Z}_{k'_1}\neq \varnothing$ has to be strictly larger than $k_1=k_0-(D+1)$, since $ \mathcal{B}'_{k_1}\subset \mathbf{B}\subset  \mathcal{W}''\cup \mathcal{W}'$. We can thus apply the induction assumption to $\mathcal{Z}$.

 By construction, $\mathbf{B}_{\mathcal{Z}}$ has at least $2L''$ balls. Next, let $\{b^{(1)}_i\}_{i=1}^{\mathcal{K}^{(1)}}$ be a family of balls as in the statement but associated to $\bigcup_{k\geq 1}\mathcal{Z}_k$ by our induction assumption. By our hypothesis, 
$$\bigcup_{i=1}^{\mathcal{K}^{(1)}}\lambda b^{(1)}_i \subset \bigcup_{B\in \bigcup_{k\geq 1}\mathcal{B}'_k\backslash  \mathcal{W}'} 2\lambda B  \cup \bigcup_{B\in \mathcal{V}} 2\lambda B .$$
Notice that $\bigcup_{B\in \mathcal{V}} 2\lambda B\subset \bigcup_{B\in \bigcup_{k\geq 1}\mathcal{B}'_k\backslash  \mathcal{W}'} 2\lambda B$, because of Property P[J] and $\mathbf{B}'' \cap \mathcal{W}' =\varnothing$. So we have
$$\bigcup_{i=1}^{\mathcal{K}^{(1)}}\lambda b^{(1)}_i  \subset \bigcup_{B\in \bigcup_{k\geq 1}\mathcal{B}'_k\backslash  \mathcal{W}'} 2\lambda B \subset \bigcup_{B\in \bigcup_{k\geq 1}\mathcal{B}'_k} 2\lambda B,$$
$$\dist(\bigcup_{i=1}^{\mathcal{K}^{(1)}}\lambda b^{(1)}_i , \bigcup_{B\in \mathbf{B}'} \lambda B) >0.$$
Consequently, one sees that $\{b^{(1)}_i\}_{i=1}^{\mathcal{K}^{(1)}}\cup \mathbf{B}'$ satisfies the two bullets in the statement of the proposition and one can take the family $\{b_i\}_{i=1}^{\mathcal{K}^{(0)}}$ to be  $\{b^{(1)}_i\}_{i=1}^{\mathcal{K}^{(1)}}\cup \mathbf{B}'$.
Hence by induction
\begin{align*}
\mathcal{K}^{(0)} & = \mathcal{K}^{(1)} + L'\\
& \geq \frac{1}{2J}\sharp \mathcal{Z}+ \frac{1}{2} 2L'' + L' \\
& \geq \frac{1}{2J}(\sharp \bigcup_{k\geq 1}\mathcal{B}'_k  - JL'-JL'') +L'' + L' \\
& \geq  \frac{1}{2J} \sum_{k\geq 1}\sharp \mathcal{B}'_k - \frac{1}{2}L +L\\
& \geq \frac{1}{2J}\sum_{k\geq 1}\sharp \mathcal{B}'_k + \frac{1}{2}L 
\end{align*}
and (\ref{double}) is proved for families $(\bigcup_{k\geq 1} \mathcal{B}'_k,\mathbf{B})$ satisfying the assumptions of the proposition, such that $k_1 =  k_0-(D+1)$. By induction the proof is finished.
\end{proof}

Combining Proposition \ref{triplybis} and Proposition \ref{doublybis} yields the following key corollary:

\begin{coro} \label{C:counting}
Consider $(M^{n+1},g)$, $3\leq n+1\leq 7$, and the constants  $A$, $\bar{r}= \frac{1}{2}$, $\lambda$, $\bar{A}$, $\delta$, $K$ as above. There is a constant $\mathbf{c}\in(0,1)$ depending on all the previous data such that the following holds. Let $\Sigma\subset M$ be any closed embedded minimal hypersurface  of $n$-volume at most $A$, and let $\mathfrak{C}\subset M$ be the associated almost conical region. 
Let $\{\bar{b}_j\}_{j=1}^Q$ be a family of $Q$ geodesic balls in $M$ so that
\begin{itemize}
\item each $\bar{b}_j$ is of the form $B(x, \mathbf{r}_{\mathrm{stab}}(x))$, with $\mathbf{r}_{\mathrm{stab}}(x) <\bar{r}$,
\item the balls $\bar{b}_j$ are pairwise disjoint,
\item each $\bar{b}_j$ is centered at a point of $M\backslash \mathfrak{C}$.
\end{itemize}
Then there exists an integer $\mathcal{K}^{(0)} \geq \mathbf{c} Q$ and a family of geodesic balls $\{{B}'_j\}_{j=1}^{\mathcal{K}^{(0)}}$ in $M$ so that
\begin{itemize}
\item each ${B}'_j$ is of the form $B(x,\lambda \mathbf{r}_{\mathrm{stab}}(x))$, with $\mathbf{r}_{\mathrm{stab}}(x) <\bar{r}$,
\item for any two different indices $j_1$ and $j_2$, we have $\dist({B}'_{j_1},  {B}'_{j_2}) >0$.
\end{itemize} 
In particular, the index and folding number satisfy 
$$\ind(\Sigma)\geq \mathbf{f}(\Sigma) \geq \mathbf{c}Q.$$
\end{coro}

\begin{proof}
Note that given $\{B'_j\}_{j=1}^{\mathcal{K}^{(0)}}$ as in the statement of the corollary, the conclusion for the folding number $\mathbf{f}(\Sigma) $ simply follows from the fact that if $B'=B(x,\lambda \mathbf{r}_{\mathrm{stab}}(x))$ with $\mathbf{r}_{\mathrm{stab}}(x)<\bar{r}$ then any open set of $\Sigma$ containing the closure of $\Sigma\cap B'$ is unstable, and so $\mathbf{f}(\Sigma) \geq \mathcal{K}^{(0)}$.

Let $\{\bar{b}_j\}_{j=1}^Q$ as in the statement.  We are assuming that $\bar{r} = 1/2$, and so the balls in $\{\bar{b}_j\}_{j=1}^Q$ have radii less than $1/2$.
To prove the corollary, first note that for a fixed $k$, given $b\in \mathcal{F}_k$, the maximal number $Q$ of disjoint balls  $\{b_1,...,b_Q\} \subset  \mathcal{F}_k$ such that 
$$\forall i\in \{1,...,Q\},\quad 3\lambda b_i \cap 3\lambda b \neq \varnothing$$
is uniformly bounded independently of $k$. Thus we can find, for each $k\geq 1$, a subfamily $\mathcal{B}'_k\subset \mathcal{F}_k\cap\{\bar{b}_j\}_{j=1}^Q$ such that 
\begin{equation}
\begin{split} \label{(23)}
\forall b,b'\in \mathcal{B}'_k,\quad 3\lambda b\cap 3\lambda b' =\varnothing,
\\
\text{and} \quad \sharp(\mathcal{F}_k\cap\{\bar{b}_j\}_{j=1}^Q) \leq {c_6} \sharp \mathcal{B}'_k ,
\end{split}
\end{equation}
where $c_6\in (0,1)$ depends only on $(M,g)$ and $\lambda$, and $\sharp$ is the cardinality. In other words, $ \bigcup_{k\geq 1}\mathcal{B}'_k$ satisfies (\ref{ravel}) and
 $$Q = \sharp\{\bar{b}_j\}_{j=1}^Q \leq c_5 \sum_{k\geq1}\sharp \mathcal{B}'_k .$$
Then conclude by applying Proposition \ref{triplybis} and Proposition \ref{doublybis} to the family $\bigcup_{k\geq 1}\mathcal{B}'_k$.

\end{proof}

 \vspace{1em}

 We finally apply this corollary to show (\ref{goalbis}).
 
\begin{proof}[Proof of Inequality (\ref{goalbis})]
We are assuming that $\bar{r}= \frac{1}{2}$. For any positive integer $k$, let $\mathcal{F}_k $ be the set of all geodesic balls $B(p,r)$ in $M$ of radius $r\in[2^{-(k+1)},2^{-k})$. Let $\mathcal{B}^{(1)}$ be the family of balls constructed in Subsection \ref{splash}. 
Note that the maximal number of disjoint balls in $\mathcal{F}_0$ is bounded above by a constant depending only on $(M,g)$. In particular, the size of $\mathcal{B}^{(1)} \cap \mathcal{F}_0$ is uniformly bounded  independently of the minimal hypersurface $\Sigma$.  Then applying Corollary \ref{C:counting} to the family  $\mathcal{B}^{(1)} \backslash \mathcal{F}_0$, the inequality (\ref{goalbis}) is proved and so is Theorem \ref{maina}.
\end{proof}

\section{Genus, Morse index and area in dimension $3$} \label{sectiongenus}

The goal of this section is to explain that in dimension $3$, the factor $C_A$ in Theorem \ref{maina} can be chosen to satisfy
$$C_A \leq C \area(\Sigma)$$
where $C$ is a constant depending on $(M,g)$ but \textit{not} on the area bound $A$ for the minimal surface $\Sigma$.
The proofs are similar to the case of higher dimensions, and we will point out the necessary modifications compared to Section \ref{sectionbetti}. The main difference is that a curvature estimate for stable minimal surfaces independent of the area is available, unlike in higher dimensions (\cite{Sc}). Moreover stable minimal surfaces or minimal surfaces with curvature bounds are much better understood in dimension $3$. For instance we will make use of the removable singularity theorem of Meeks-P\'{erez}-Ros (see Theorem \ref{mpr}) instead of the monotonicity formula.

\subsection{The almost flat region $\mathfrak{C}$} \label{almostflat1} \label{luyy}

We define a region which will play a similar role as the almost conical region defined in Subsection \ref{almostflat2}. We will still denote it by $\mathfrak{C}$ for simplicity. Let $\varepsilon_{\mathrm{S}}\in (0,1)$. Given a closed minimal surface $\Sigma$, let $\mathbf{r}_{\mathrm{stab}}$ be the stability radius associated to $\Sigma$ as defined in Subsection \ref{stabrad}. One essential difference with Section \ref{sectionbetti} is that there are curvature bounds independent of the area due to Schoen \cite{Sc}. If we suppose $\lambda>2$ large enough and $\lambda \bar{r}$ small enough (depending on $(M,g)$ and $\varepsilon_{\mathrm{S}}$, but not on the area of the minimal surface), for any closed embedded minimal surface $\Sigma$, for any $x\in M$ and any $y\in \Sigma\cap B(x,2\mathbf{r}_{\mathrm{stab}}(x))$,
\begin{equation} \label{sbound}
\mathbf{r}_{\mathrm{stab}}(x) |\mathbf{A}(y)|< \varepsilon_{\mathrm{S}}.
\end{equation}
As before, by choosing $\varepsilon_{\mathrm{S}}, \frac{1}{\lambda}, \bar{r}$ small enough (independently of $\Sigma$), for any finite family of points $x_1,...,x_J \in M$, the intersection 
\begin{equation} \label{huy}
\Sigma \cap \bigcap_{j=1}^J B(x_j,\mathbf{r}_{\mathrm{stab}}(x_j))
\end{equation}
is a union of open sets diffeomorphic to $2$-disks which are convex for the induced metric on $\Sigma$. 

Again, another important point due to the fact that $\mathbf{r}_{\mathrm{stab}}$ is $\frac{1}{\lambda}$-Lipschitz (Lemma \ref{continui}) is that, 
given $y\in M$, the maximal number of disjoint balls of the form $ B(x,\mathbf{r}_{\mathrm{stab}}(x))$ intersecting $B(y,\mathbf{r}_{\mathrm{stab}}(y))$ non-trivially is bounded above by a constant 
\begin{equation}\label{muuudeux}
\text{$\hat{\mu}=\hat{\mu}(M,g,\bar{r},\lambda)>0$.}
\end{equation}

There is a reverse inequality like (\ref{reverse1}): for a point $x$ if $\mathbf{r}_{\mathrm{stab}}(x) < \bar{r}$ then for an $\varepsilon'$
 depending on $(M,g)$, $\bar{r}$ but independent of $\Sigma$ 
\begin{equation} \label{reverse2}
\text{ there is $y'\in B(2\lambda\mathbf{r}_{\mathrm{stab}}(x))$ such that $\mathbf{r}_{\mathrm{stab}}(x) |\mathbf{A}(y')|>\varepsilon'$}
\end{equation}

\vspace{1em}

\subsubsection{Definition of pointed $\epsilon$-flat annuli}

Let $\epsilon>0$. Let $\Sigma\subset M$ be a closed embedded minimal surface. For $2s<t<\bar{r}$ and $p\in M$, we will say that $\Sigma\cap A(p,s,t)$ is \textit{$\epsilon$-flat} if: for all $r\in[s,t/2]$ and for any point $x\in \Sigma\cap A(p,r,2r)$, 
\begin{equation} \label{flllllat}
r |\mathbf{A}(x)| \leq \epsilon.
\end{equation}
We will say that $ A(p,s,t)$ is a \textit{pointed $\epsilon$-flat annulus} if 
$$\mathbf{r}_{\mathrm{stab}}(p)<s\leq \frac{1}{2} t < \frac{1}{2} \bar{r}$$
and if $\Sigma \cap A(p,s,t)$ is $\epsilon$-flat.

If $\epsilon, \bar{r}$ are small enough independently of $\Sigma$, and if $A(p,s,t)$ is a pointed $\epsilon$-flat annulus, the components of $\Sigma\cap A(p,s,t)$ are annuli and disks. In general $\Sigma\cap A(p,s,t)$ is composed of many almost flat sheets which do not have to intersect $\partial A(p,s,t)$ at an angle close to $\frac{\pi}{2}$. The reverse inequality (\ref{reverse2}) also implies that there is a point in $B(p,2\lambda s)$ with second fundamental form of norm at least $\frac{\varepsilon'}{s}$. As in (\ref{review1}), if $\epsilon$ is small, by the properties of $\mathbf{r}_{\mathrm{stab}}$, for any pointed $\epsilon$-flat annulus $A(p,s,t)$ and any $x$ such that
$$B(x, \mathbf{r}_{\mathrm{stab}}(x)) \cap \partial B(p,t) \neq \varnothing,$$
the stability radius is comparable to $t$:
\begin{equation} \label{review2}
C t \leq \mathbf{r}_{\mathrm{stab}}(x) \leq (1-\frac{1}{\lambda})^{-1}(\frac{t}{\lambda} +s)
\end{equation}
for some $C>0$ depending on $\lambda, \bar{r}$ but independent of $\Sigma$ or $t$ (we emphasize that here $C$ does not depend on an area bound for $\Sigma$).

\vspace{1em}

\subsubsection{Definition of $(\epsilon,K)$-telescopes} \label{defoftelescopes2}

Consider a parameter $K>1000+\lambda$. Define
$$\bar{\mathcal{A}}_{\epsilon}:=\{ A(p,s,2s); A(p,\frac{s}{K},Ks) \text{ is a pointed $\epsilon$-flat annulus}\},$$
$$\bar{\mathcal{A}}^{bis}_{\epsilon}:=\{ A(p,s,2s); A(p,\frac{100s}{K},\frac{Ks}{100}) \text{ is a pointed $\epsilon$-flat annulus}\}.$$
Those sets depend on $(M,g)$, $\Sigma$, $\lambda$, $\bar{r}$, $\epsilon$, $K$.


Let $A(z_1,r_1,2r_1),..., A(z_m,r_m,2r_m)$ be annuli in $\bar{\mathcal{A}}^{bis}_{\epsilon}$ such that their boundaries $\partial A(z_i,r_i,2r_i)$ intersect $\Sigma$ transversely. We say that 
$$T:= \bigcup_{i=1}^m A(z_i,r_i,2r_i)$$ is an \textit{$(\epsilon,K)$-telescope} if for all $i=1,...m-1$, 
\begin{equation}\label{inininbis}
\begin{split}
\overline{B}(z_i,r_i) & \subset B(z_{i+1},r_{i+1}) \\
\overline{B}(z_{i+1},r_{i+1}) & \subset B(z_i,2r_i)\\ 
\overline{B}(z_i,2r_i) & \subset B(z_{i+1},2r_{i+1})
\end{split}
\end{equation}
where $\overline{B}$ denotes the closure of $B$.
Note that if $T$ is a telescope $\partial T$ has two  components $\partial_-T:= \partial B(z_1,r_1)$ and $\partial_+T:= \partial B(z_m,2r_m)$. Moreover from the definitions, $T$ is strictly included in a ball of radius $\bar{r}$.

\subsubsection{Construction of the almost flat region $\mathfrak{C}\subset M$ associated to $\Sigma$} \label{defoftelescopes3}

The following lemma is very similar to Lemma \ref{almostflatT}:

\begin{lemme} \label{almostflat}

Let $(M^{3},g)$ be a closed $3$-manifold.
If the constants $\varepsilon_{S}, \frac{1}{\lambda}, \bar{r}$ are small enough and if $\frac{1}{K}, \epsilon$ are chosen small enough depending on the previous constants, then the following is true.

Let $\Sigma$ be a closed smoothly embedded minimal surface and let $\bar{\mathcal{A}}_{\epsilon}$ be defined as above.  
Then there exist disjoint $(\epsilon,K)$-telescopes $T_1,...,T_L$  at positive distance from one another, and 
\begin{equation} \label{tubb'}
\bigcup_{An\in \bar{\mathcal{A}}_{\epsilon}} An \subset \bigcup_{i=1}^L T_i.
\end{equation}
Moreover
\begin{enumerate} [label=(\roman*)]
\item for any $i \in \{1,...,L\}$, the intersection $\Sigma \cap T_i$ is diffeomorphic to a union of $2$-dimensional disks and annuli, 
\item if $\mathcal{R}\subset\Sigma$ is a compact subsurface of $\Sigma$ with $\mathbf{b}(\mathcal{R})$ boundary components such that $\Sigma\setminus \mathcal{R} \subset \bigcup_{i=1}^L T_i$, 
then $\genus(\Sigma) \leq \genus(\mathcal{R}) + \frac{\mathbf{b}(\mathcal{R})}{2}$,
\item  for any $i \in \{1,...,L\}$, $T_i$ intersects at most $\mu$ disjoint balls in 
$$\{B(x,\mathbf{r}_{\mathrm{stab}}(x)); x\in M\backslash T_i\}$$ where $\mu= \mu(M,g,\bar{r},\lambda)>0$,
\item for any $x\in M$, there are at most $2$ distinct indices $i,i'\in \{1,...,L\}$ such that
$B(x,\mathbf{r}_{\mathrm{stab}}(x)) \cap T_i \neq \varnothing$ and $B(x,\mathbf{r}_{\mathrm{stab}}(x)) \cap T_{i'} \neq \varnothing$.
\end{enumerate}

\end{lemme}

\begin{proof} The existence of the $(\epsilon,K)$-telescopes is shown as in  the proof of Lemma \ref{almostflatT} since essentially only the combinatorial properties of annuli in $\bar{\mathcal{A}}^{bis}_{\epsilon}$ or $\bar{\mathcal{A}}_{\epsilon}$ are needed.

As for $(i)$,  given an annulus $A(p,t,2t)\in \bar{\mathcal{A}}^{bis}_{\epsilon}$, $\Sigma \cap  A(p,t/2,4t)$ satisfies the curvature bound (\ref{flllllat}) and $2t \leq \bar{r}$ by definition of the annulus. Rescale the metric so that $A(p,t,2t)$ has inner radius $1$ then any connected component of $\Sigma \cap A(p,t,2t)$ is arbitrarily close to being flat when $\epsilon$ is small enough. But if  $\bar{r}$ was chosen small enough the rescaled metric is very close to the flat metric, and the second fundamental forms of $\partial B(p,t)$ and $\partial B(p,2t)$ are uniformly close to being that of the unit sphere and the sphere of radius $2$ in the Euclidean $3$-space. By the implicit function theorem, $\Sigma \cap A(p,t,2t)$ is a union of $2$-disks and $2$-annuli whenever $\bar{r}$ and $\epsilon$ are small enough. This is enough to deduce $(i)$: the components of $\Sigma\cap T_i$ intersecting the lower boundary component $\partial_-T_i$ are annuli while the others are disks.

For $(ii)$, the genus of the compact surface $\mathcal{R}$ is defined by the genus of the closed surface obtained by gluing disks along the boundary circles of $\mathcal{R}$. Property $(ii)$ is a corollary of $(i)$.

The proofs of $(iii)$, $(iv)$ are as in Lemma \ref{almostflatT}. 

\end{proof}

We can now define the almost flat region (which depends on $\Sigma$).

\begin{definition} \label{2020june}

Given a closed embedded minimal surface $\Sigma$ in the closed $3$-manifold $(M,g)$, we define the almost flat region as
$$\mathfrak{C} := 
\bigcup_{i=1}^L T_i,$$
where $ T_i$ are the $(\epsilon,K)$-telescopes given by Lemma \ref{almostflat}. 
\end{definition}

The choice of the telescopes in the definition of $\mathfrak{C}$ is not unique. From now on we fix $\bar{r}$, $\lambda$ as above, $\epsilon$ and $K>2^{1000}$ as in Lemma \ref{almostflat}.

\vspace{1em}

\subsection{Quantitative removable singularity theorem}  \label{mmpprr}
We present the removable singularity theorem of Meeks-P\'{e}rez-Ros \cite{MPR} in a more quantitative manner than the original statement. This will play the role of Theorem \ref{noconcentration} in dimension $3$. Recall that $\bar{r}$, $\lambda>2$, $\mathbf{r}_{\mathrm{stab}}$, $K$, $\epsilon$ were introduced before. In the statement of Theorem \ref{mpr}, 
we set $\bar{r}=\infty$ so that $\mathbf{r}_{\mathrm{stab}}$ can take values in $(0,\infty]$.

\begin{theo}[\cite{MPR}] \label{mpr}
There exist $\bar{K}\geq 30\lambda K^2$, $\mu>0$, and $\bar{R}>1000$ such that the following is true. Let $g$ be a metric $\mu$-close to the Euclidean metric in the $C^5$-topology on $B_{\text{Eucl}}(0,\bar{K})\subset \mathbb{R}^3$, and let 
$$(\Sigma, \partial \Sigma) \subset (B_{\text{Eucl}}(0,\bar{K}), \partial B_{\text{Eucl}}(0,\bar{K}))$$ be a compact embedded minimal surface with respect to $g$, such that
\begin{itemize}
\item $\mathbf{r}_{\mathrm{stab}}(0) < \bar{K}^{-1}$,
\item there is $y'\in B(0,7\lambda)$ with $\mathbf{r}_{\mathrm{stab}}(y')=1$.
\end{itemize}

Then 
\begin{enumerate} [label=(\roman*)]
\item either there are two balls $b_1=B(z',\mathbf{r}_{\mathrm{stab}}(z')), b_2=B(z'',\mathbf{r}_{\mathrm{stab}}(z''))$ such that 
$$z',z''\in B(0,\frac{\bar{K}}{2}), \quad  3\lambda b_1\cap 3\lambda b_2 =\varnothing,$$ 
$$\mathbf{r}_{\mathrm{stab}}(z') , \mathbf{r}_{\mathrm{stab}}(z'') \in [2^{-({R}+1)},2^{-R}) \text{  for some $R\in [1000,\bar{R}]$},$$
\item or $\Sigma \cap A\big(0,\frac{1}{2{K}}, 7\lambda {K}\big)$  is ${\epsilon}$-flat in the sense of Subsection \ref{almostflat1}.

\end{enumerate}

\end{theo}

\begin{proof}
The essential ingredient is the result of Meeks-P\'{e}rez-Ros \cite{MPR}, which has the role played by the monotonicity formula in the proof of Theorem \ref{noconcentration}.

If Theorem \ref{mpr} is supposed to be not true, we want to derive a contradiction. For all $k$ large so that $2^k\geq 30\lambda K^2$, let $g_k$ be a sequence of metrics on $B_{\text{Eucl}}(0,2^k)$ converging to the Euclidean metric on compact sets, and let $\Sigma_k$ be a sequence of smooth embedded compact minimal surfaces as in the statement, with possibly unbounded area, except that the first bullet is replaced by
$$\mathbf{r}_{\mathrm{stab}}(0) < \frac{1}{2^k}.$$
For clarity denote by $B_k(.,.)$ (resp. $B_g(.,.)$) a geodesic ball for the metric $g_k$ (resp. for a metric $g$). 
Let 
$$\hat{r}_l:=\frac{1}{2^{l+1}},$$ this number satisfies the following. Given an integer $l$, for any metric $g$ close enough to the Euclidean one,  for any point $z\in B_g(0, 2^l)\backslash B_g(0,\frac{1}{2^l})$ there exist two points $u,v\in B_k(y',1) $ and two continuous paths $\gamma_{0,u}$ from $0$ to $u$, $\gamma_{z,v}$ from $z$ to $v$ so that
$$\gamma_{z,v}\subset B_g(0,2^l) \backslash B_g(0,\frac{1}{2^l})$$
and the $\hat{r}_l$-neighborhoods (with respect to $g$) of $\gamma_{0,u}$ and $\gamma_{z,v}$ are disjoint, contained in $B_g(0,2^l+1)$. 
Let $\bar{R}_l$ be defined as follows: 
$$\bar{R}_l = \log_2(\frac{3\lambda}{\hat{r}_l})+1000 \quad \text{(so that in particular $6\lambda 2^{-\bar{R}_l}<\hat{r}_l$ and $\bar{R}_l>1000$)}.$$
One can check that for two positive integers $l_1,l_2$, $\hat{r}_{l_1} = 2^{l_2-l_1}\hat{r}_{l_2}$ and $\bar{R}_{l_1} = \bar{R}_{l_2} - l_2+ l_1$.
Similarly to the proof of Theorem \ref{noconcentration}, one can prove that if $(i)$ does not occur to the sequences $g_k$ and $\Sigma_k$ with  $\bar{R} = \bar{R}_k$, then for any given $l$, for $k$ large enough:
\begin{equation} \label{kart}
\forall z\in B_k(0,2^{l})\backslash B_k(0,\frac{1}{2^l}), \quad \mathbf{r}_{\mathrm{stab}}(z) \geq 2^{-\bar{R}_l} = 2^{-l+1}.2^{-\bar{R}_1} .
\end{equation}

Recall that Schoen's $3$-dimensional curvature bounds \cite{Sc} (which are independent of the area) imply the following: for $k$ large, for $g_k$ and $\Sigma_k$, in any ball of the form $B(x,\mathbf{r}_{\mathrm{stab}}(x))$, one has
$$|\mathbf{A}(y)| \mathbf{r}_{\mathrm{stab}}(x) \leq \varepsilon_{\mathrm{S}}$$
(see beginning of Subsection \ref{almostflat1}). 
Using this curvature bound and (\ref{kart}), we conclude that if $(i)$ does not hold for $g_k$ and $\Sigma_k$ with $\bar{R}=\bar{R}_k$ then there is a constant $C$ independent of $k,l$ so that for any given $l>0$, for all $k$ large enough and for all $y\in B_k(0,2^{l})\backslash B_k(0,\frac{1}{2^l})$
\begin{equation} \label{bounddd}
|\mathbf{A}(y)| \leq C(\dist_{g_k}(0,y)+\frac{1}{\dist_{g_k}(0,y)}).
\end{equation}

As a consequence, $\Sigma_k$ subsequently converges smoothly outside of $0$ to a lamination $\mathcal{L}$ of $(\mathbb{R}^3\backslash \{0\},g_{\text{Eucl}})$. Since $\mathcal{L}$ still satisfies the curvature bound (\ref{bounddd}), by \cite{MPR}, $0$ is a removable singularity of $\mathcal{L}$. Since for $g_k$ and $\Sigma_k$, $\mathbf{r}_{\mathrm{stab}}(0) < 2^{-k}$, by monotonicity, $\mathcal{L}$ has a leaf $L$ passing through $0$. This leaf $L$ cannot be isolated and of multiplicity one, since by Allard's regularity theorem, the convergence would be smooth but it would contradict $\mathbf{r}_{\mathrm{stab}}(0) < 2^{-k}$ for $\Sigma_k$. Thus $L$ is either isolated of multiplicity greater than one, or non-isolated. By Corollaries 3.4 and 3.5 of \cite{MPR}, the closure $\bar{L}$ is a flat plane. By the half-space theorem for bounded curvature minimal surfaces of Xavier \cite{Xavier}, any other leaf of $\mathcal{L}$ is flat. But it means that for $k$ large, $\Sigma_k \cap B_k(0,7\lambda K) \backslash B_k(0,\frac{1}{2K})$  was ${\epsilon}$-flat, a contradiction.

\end{proof}

Let $(M^3,g)$, $\bar{r}$, $\lambda>2$, $\epsilon$, $K$ be as in the previous subsection, $\mu$, $\bar{R}>1000$, $\bar{K}\geq 30\lambda K^2$ as in Theorem \ref{mpr}. We assume that for any ball of radius $r$ less than $\bar{r}$ the metric $\frac{1}{r^2} g$ is $\mu$-close to a flat metric in the $C^5$-topology. The next lemma corresponds to Corollary \ref{sotechnical} and has the same proof, except that one replaces Theorem \ref{noconcentration} by Theorem \ref{mpr}.

\begin{coro}\label{quitetechnical}
Let $\Sigma\subset M^3$ be a closed smoothly embedded minimal surface. Consider $x_0,x_1,x_2\in M$ such that
\begin{itemize}
\item $\mathbf{r}_{\mathrm{stab}}(x_0) < {\bar{r}}$, $\mathbf{r}_{\mathrm{stab}}(x_1)\leq {\mathbf{r}_{\mathrm{stab}}(x_0)}{\bar{K}}^{-1}$, $\mathbf{r}_{\mathrm{stab}}(x_2) \leq {\mathbf{r}_{\mathrm{stab}}(x_1)}{\bar{K}}^{-1}$,
\item \begin{align*}
& B(x_0,1.1\lambda \mathbf{r}_{\mathrm{stab}}(x_0)) \cap B(x_2,\mathbf{r}_{\mathrm{stab}}(x_2)) \neq \varnothing, \\
& B(x_1,6\lambda \mathbf{r}_{\mathrm{stab}}(x_1) ) \cap B(x_2,3\lambda \mathbf{r}_{\mathrm{stab}}(x_2)) \neq \varnothing,
\end{align*}
\item $x_1$ is not contained in the almost conical region $\mathfrak{C}$.
\end{itemize}
Then there are two balls $b_1,b_2$ of the form $B(y,\mathbf{r}_{\mathrm{stab}}(y))$ such that 
\begin{enumerate} [label=(\roman*)]
\item $3\lambda b_1\cap 3\lambda b_2\neq \varnothing$, 
\item the radii of $b_1,b_2$  belong to $$[2^{-({R}+1)} \mathbf{r}_{\mathrm{stab}}(x_1),2^{-{R}} \mathbf{r}_{\mathrm{stab}}(x_1)) \text{ for some $R \in [1000, \bar{R}]$},$$ 
\item $3\lambda b_1\cup 3\lambda b_2 \subset B(x_0,1.7\lambda \mathbf{r}_{\mathrm{stab}}(x_0))$. 
\end{enumerate}
\end{coro}

\vspace{1em}

\subsection{Proof of Theorem \ref{maina} in dimension $3$} \label{mainapf}

The arguments are essentially the same as those in Subsection \ref{splash}. Let us go through it briefly and insist on the differences. Let $\Sigma \subset (M^3,g)$ be a closed embedded smooth minimal surface.
We will use the constants $\bar{r}= \frac{1}{2}$, $\lambda$, $\epsilon$, $K$, the functions $\mathbf{f}$, $\mathbf{r}_{\mathrm{stab}}$, and the almost flat region $\mathfrak{C}$ defined previously in this section.

By the discussion surrounding (\ref{huy}), a finite intersection of subsets of the form $B(x,\mathbf{r}_{\mathrm{stab}}(x)) \cap \Sigma$ is a union of convex $2$-disks.
Recall that we previously defined the sheeting number as follows:
\begin{definition}
The sheeting number of $\Sigma$ is the number $N(\Sigma)$ defined by 
$$N(\Sigma):=\max\{N; \exists p\in M, \Sigma\cap B(p,\mathbf{r}_{\mathrm{stab}}(p)) \text{ has $N$ components}\}.$$
\end{definition}
By monotonicity formula, for a constant $c_2$ depending on $(M,g)$, $\lambda, \bar{r}$ but not  on $\Sigma$ or its area,
\begin{equation} \label{il est tard}
N(\Sigma) \leq c_2 \area(\Sigma).
\end{equation}


As in Subsection \ref{splash}, starting with Besicovitch's covering lemma, we find families of balls
$$\mathcal{B}^{(1)},...,\mathcal{B}^{(c_1)} \subset \{B(x,\mathbf{r}_{\mathrm{stab}}(x));x\in \Sigma \setminus \mathfrak{C}\},$$
such that each $\mathcal{B}^{(i)}$ consists of disjoint balls, and 
$$\Sigma\setminus \mathfrak{C} \subset \bigcup_{i=1}^{c_1} \bigcup_{b\in \mathcal{B}^{(i)}} b,$$
where $c_1$ is an integer depending only on $(M,g)$. We can assume that $\sharp \mathcal{B}^{(1)}$ is at least as large as any other $\sharp \mathcal{B}^{(i)}$. 

By slightly enlarging the radius of each ball $b \in\mathcal{B}^{(1)}\cup ...\cup\mathcal{B}^{(c_1)}$ if necessary, we get a new ball $\tilde{b}$ such that
$$\mathcal{R} := \Sigma \cap \bigcup_{i=1}^{c_1} \bigcup_{b\in \mathcal{B}^{(i)}} \tilde{b}$$
is a compact surface with piecewise smooth boundary satisfying
$$\Sigma \cap \bigcup_{i=1}^{c_1} \bigcup_{b\in \mathcal{B}^{(i)}} b \subset \mathcal{R} \subset \Sigma \quad (\text{in particular $\Sigma\setminus\mathcal{R} \subset \mathfrak{C}$}).$$
The genus and number of boundary components of $\mathcal{R}$ are linearly controlled by its total Betti number over $\mathbb{R}$. Thus after  applying the topological result Lemma \ref{lemmetopo}  in the Appendix, we find
$$\genus(\mathcal{R}) +\frac{\mathbf{b}(\mathcal{R})}{2} \leq c_4 N(\Sigma) \sharp \mathcal{B}^{(1)}$$
where  $\mathbf{b}(\mathcal{R})$ denotes the number of boundary components of $\mathcal{R}$, and $c_4$ depends only on $(M,g)$, $\lambda$, $\bar{r}$ (independent of the area of $\Sigma$). By Lemma \ref{almostflat} $(ii)$ and (\ref{il est tard}), 
\begin{equation} \label{glad}
\begin{split}
\genus(\Sigma) &\leq c_4 N(\Sigma) \sharp \mathcal{B}^{(1)} \\
\genus(\Sigma) &  \leq c_4c_2 \area(\Sigma ) \sharp \mathcal{B}^{(1)}.
\end{split}
\end{equation}

It remains to show the following inequality 
\begin{equation} \label{nujabes}
 \sharp \mathcal{B}^{(1)} \leq c_5 (\mathbf{f}(\Sigma)+1)
\end{equation}
for a constant $c_5$ depending only on $(M,g)$. Its proof is postponed to the next subsection.

\vspace{1em}

\subsection{The combinatorial argument in dimension $3$} \label{olpo}

The following version of Proposition \ref{triplybis} is true in dimension $3$ without the dependency on the area of $\Sigma$. We restate it because its proof is a bit different. Recall that Property P[J] was introduced in Subsection \ref{combinatorial}.

In what follows, we are assuming that $\bar{r}= \frac{1}{2}$. For any positive integer $k$, let $\mathcal{F}_k $ be the set of all geodesic balls $B(p,r)$ in $M$ of radius $r\in[2^{-(k+1)},2^{-k})$. Let $\mathcal{B}'_k$, the notion of basis, and Property [J] be defined as in Subsection \ref{combinatorial}.

\begin{prop} \label{triply}
There is a positive integer $J$ depending only on $(M^3,g)$ and on the constants $\bar{r}, \lambda, \epsilon, K$, such that the following is true. 
Let $\Sigma\subset (M,g)$ be a closed embedded minimal surface. For each $k\geq 1$, let $\mathcal{B}'_k\subset \mathcal{F}_k$ be a family of balls of the form $B(x,\mathbf{r}_{\mathrm{stab}}(x))$ such that for any two different balls $b,b'\in \mathcal{B}'_k$, we have
$$3\lambda b\cap  3\lambda b'=\varnothing,$$
and let $\mathbf{B}$ be a basis of $\bigcup_{k\geq1}\mathcal{B}'_k$. Assume additionally that any ball $b=B(x,\mathbf{r}_{\mathrm{stab}}(x))\in \bigcup_{k\geq1}\mathcal{B}'_k \backslash \mathbf{B}$ is centered at $x\in M\backslash \mathfrak{C}$.

Then $(\bigcup_{\geq1}\mathcal{B}'_k,\mathbf{B})$ satisfies Property P[J].
\end{prop}

\begin{proof} 
The proof slightly differs from that of Proposition \ref{triplybis} so it deserves some explanations.

Recall the definitions of $\bar{r}= \frac{1}{2}$, $\bar{A}$, $\lambda$, $\delta$, $\bar{K}\geq 30\lambda K^2$ (Subsections \ref{stabrad}, \ref{almostflat2}, \ref{mmpprr}).

Let $\Sigma$, $\bigcup_{\geq1}\mathcal{B}'_k$, $\mathbf{B}$ be as in the statement. Let $B\in \mathbf{B}\cap \mathcal{B}'_k$ for some $k\geq 1$. We denote the radius of a ball by $\rad$. Facts 1 and 2 of the proof of Proposition \ref{triplybis} still hold true. \newline

Set $$J:= 2002 \Upsilon_1 + 4(\log_2(2\bar{K})+1)\Upsilon_1 +2\bar{R}\Upsilon_1 + \Upsilon_2 ,$$
where $\bar{R}$ is given in Theorem \ref{noconcentration}, and $\Upsilon_1,\Upsilon_2$ are as in the proof of Proposition \ref{triplybis}.

We can suppose that there is an integer $v_1> \log_2(2\bar{K})$ and a ball $\bar{b} \in \mathcal{B}'_{k+v_1}$ with $3\lambda\bar{b} \cap \lambda B\neq \varnothing$, since otherwise Property P[J] is true by Fact 2. We can assume $v_1$ to be the smallest integer satisfying these properties.
\newline

\textbf{Case A.} First suppose that the number of $u> 0$ such that there is a $\hat{b}\in \mathcal{B}'_{k+u}$ with $3\lambda\hat{b}\subset (2\lambda -2^{-1000})B$ is larger than $ 2(\log_2(2\bar{K})+1)$.  Instead of introducing $m$ balls as in the proof of Proposition \ref{triplybis}, we just need two of them.

Write for convenience $k=u_0$. We can find two integers $u_1,u_2$ with 
\begin{itemize}
\item $u_{i} \geq u_{i-1}+\log_2(2\bar{K})$ for $i=1,2$,
\item there is $\hat{b}_i \in \mathcal{B}'_{k+u_i}$ with $3\lambda \hat{b}_i \subset  (2\lambda-2^{-1000})B$,
\item for $u\in(u_{i-1}+\log_2(2\bar{K}) , u_i)$, there is no $\hat{b} \in \mathcal{B}'_{k+u}$ with $3\lambda \hat{b} \subset  (2\lambda-2^{-1000})B$. 
\end{itemize}
Since for $u\geq \log_2(2\bar{K})$, any $\hat{b}\in \mathcal{B}'_{k+u}$ with $3\lambda\hat{b}\cap \lambda B\neq \varnothing$ is such that $3\lambda\hat{b}\subset(2\lambda-2^{-1000})B$, we have $v_1\geq u_1$. Note also that
$$\rad(\hat{b}_1) < \frac{\rad(B)}{ \bar{K}} \text{ and }$$
$$\rad(\hat{b}_2) < \frac{\rad(\hat{b}_1)}{\bar{K}} .$$ \newline


Case A.1.\newline 
Suppose that for some $j\in\{1,2\}$,
$$\hat{b}_j \cap 1.1\lambda B=\varnothing.$$
Then Case (2) of Property P[J] is satisfied (same proof as Case A.1. in the proof of Proposition \ref{triplybis}).\newline

Case A.2.\newline 
Suppose that, 
$$\hat{b}_1 \cap 1.1\lambda B\neq \varnothing \text{ and }\hat{b}_2 \cap 1.1\lambda B\neq \varnothing.$$
Assume also that 
$$6\lambda\hat{b}_1\cap 3\lambda\hat{b}_2 = \varnothing.$$
Then Case (2) of Property P[J] is satisfied (same proof as Case A.2. in the proof of Proposition \ref{triplybis}).\newline

Case A.3.\newline 
Suppose again that
$$\hat{b}_1 \cap 1.1\lambda B\neq \varnothing \text{ and }\hat{b}_2 \cap 1.1\lambda B\neq \varnothing.$$
but this time assume additionally that  
$$6\lambda\hat{b}_1\cap 3\lambda\hat{b}_2\neq\varnothing.$$
Since $B$ is an element of the basis $\mathbf{B}$ and $\hat{b}_1 \cap 1.1\lambda B\neq \varnothing$, $\hat{b}_1$ is not in this basis and is centered outside of the almost flat region $\mathfrak{C}$. We can thus apply Corollary \ref{quitetechnical} with $x_0$, $x_1$, $x_2$ being the centers of $B$, $\hat{b}_1$, $\hat{b}_2$: there are two balls $b_1,b_2$ of the form $B(x,\mathbf{r}_{\mathrm{stab}}(x))$ with 
$$3\lambda b_1\cap 3\lambda b_2 =\varnothing, \quad 3\lambda b_1 \cup 3\lambda b_2\subset 1.7\lambda B,$$
such that $b_1,b_2\in \mathcal{F}_{k+u_j+R}$ for some $R\in [1000,\bar{R}]$.
Note that if a ball $Y$ has radius less than $2^{-(k+1000)} $ and $3\lambda Y\cap (3\lambda b_1 \cup 3\lambda b_2) \neq \varnothing$, then $3\lambda Y\subset (2\lambda-2^{-1000})B$. So the size of
$$\{\hat{b}\in \mathcal{B}'_{k+u} ; 0\leq u \leq u_j+R, 3\lambda \hat{b} \cap (3\lambda b_1\cup 3\lambda b_2)\neq \varnothing\}$$
is bounded above by 
$$1001.(2\Upsilon_1) + 2(\log_2(2\bar{K})+1).(2\Upsilon_1)+\bar{R}.(2\Upsilon_1).$$ 
Consequently Case (2) of Property P[J] is satisfied.\newline

\textbf{Case B.} Secondly, suppose that the number of indices $u>0$ such that there is a $\hat{b}\in \mathcal{B}'_{k+u}$ with $3\lambda\hat{b}\subset (2\lambda-2^{-1000})B$ is bounded above by $ 2(\log_2(2\bar{K})+1)$. 
Then Case (1) or (2) of Property P[J] is satisfied (same proof as Case B. in the proof of Proposition \ref{triplybis}).

\end{proof}

Combining Proposition \ref{triply} and Proposition \ref{doublybis} yields the following corollary similar to Corollary \ref{C:counting} except that the constant now does not depend on the area. The proof is completely similar to that of Corollary \ref{C:counting}.

\begin{coro} \label{C:counting3D}
Let $(M^{3},g)$ and the constants $\bar{r}$, $\lambda$, $\epsilon$, $K$ be as above. There is a constant $\mathbf{c}\in(0,1)$ depending on all the previous data such that the following holds. Consider any closed embedded minimal surface $\Sigma\subset M$, and the associated almost flat region $\mathfrak{C}\subset M$. 
Let $\{\bar{b}_j\}_{j=1}^Q$ be a family of $Q$ geodesic balls in $M$ so that
\begin{itemize}
\item each $\bar{b}_j$ is of the form $B(x, \mathbf{r}_{\mathrm{stab}}(x))$, with $\mathbf{r}_{\mathrm{stab}}(x) <\bar{r}$,
\item the balls $\bar{b}_j$ are pairwise disjoint,
\item each $\bar{b}_j$ is centered at a point of $M\backslash \mathfrak{C}$.
\end{itemize}
Then there exists an integer $\mathcal{K}^{(0)} \geq  \mathbf{c}Q$ and a family of geodesic balls $\{{B}'_j\}_{j=1}^{\mathcal{K}^{(0)}}$ in $M$ so that
\begin{itemize}
\item each ${B}'_j$ is of the form $B(x,\lambda \mathbf{r}_{\mathrm{stab}}(x))$, with $\mathbf{r}_{\mathrm{stab}}(x) <\bar{r}$,
\item for any two different indices $j_1$ and $j_2$, we have $\dist({B}'_{j_1},  {B}'_{j_2}) >0$.
\end{itemize} 
In particular, the index and folding number satisfy $$\ind(\Sigma)\geq \mathbf{f}(\Sigma) \geq \mathbf{c}Q.$$
\end{coro}

We apply this corollary to show (\ref{nujabes}).

\begin{proof}[Proof of Inequality (\ref{nujabes})]
Let $(M^{3},g)$ and $\Sigma$ be as in the statement of Theorem \ref{maina}, and suppose that $\bar{r}=1/2$ by rescaling the metric. Let $\mathcal{B}^{(1)}$ be the family of balls considered at the beginning of this subsection. The size of $\mathcal{B}^{(1)} \cap \mathcal{F}_0$ is uniformly bounded  independently of the minimal surface $\Sigma$.  Then we apply Corollary \ref{C:counting3D} to the family  $\mathcal{B}^{(1)} \backslash \mathcal{F}_0$: inequality (\ref{nujabes}) follows and so Theorem \ref{maina} is proved in dimension 3.

\end{proof}

\section{A two-piece decomposition for minimal surfaces in dimension $3$} \label{reamrks}

In this section, let $(M^3,g)$ be a closed Riemannian $3$-manifold. We define a two-piece decomposition for embedded minimal surfaces and refine the results of Section \ref{sectiongenus}.

\subsection{The sheeted/non-sheeted decomposition}
Let $(M^3,g)$ be a closed Riemannian $3$-manifold. A consequence of the previous sections is a ``two-piece decomposition'' for minimal surfaces analogous to the usual thick-thin decomposition for manifolds with bounded sectional curvature. Let $\bar{r}$, $\lambda$ be large constants depending on $(M,g)$, fixed as in the previous section, then the radius $\mathbf{r}_{\mathrm{stab}}$ is well defined (see Subsection \ref{stabrad}). Let $N_0\geq1$ be a positive integer and let $\Sigma\subset (M,g)$ be an embedded minimal surface. Then we define the non-sheeted part $\Sigma_{\leq N_0}$ and the sheeted part $\Sigma_{>N_0}$ as follows:
$$\Sigma_{\leq N_0} := \{q\in \Sigma; \Sigma\cap B(q,\mathbf{r}_{\mathrm{stab}}(q)) \text{ has at most $N_0$ components} \},$$
$$\Sigma_{> N_0} := \{q\in \Sigma; \Sigma\cap B(q,\mathbf{r}_{\mathrm{stab}}(q)) \text{ has more than $N_0$ components} \}.$$
Notice that 
$$\Sigma = \Sigma_{\leq N_0} \sqcup \Sigma_{> N_0}, $$
and $\Sigma_{\leq N_0}$ (resp. $\Sigma_{> N_0}$) is a closed (resp. open) subset of $\Sigma$. We will see in the next subsections that both the genus and the area of the non-sheeted part $\Sigma_{\leq N_0}$ can be controlled by the Morse index of $\Sigma$.
The sheeting number $N(\Sigma)$ defined in Subsection \ref{mainapf} is simply the smallest integer $N$ so that $\Sigma = \Sigma_{\leq N}$. By Schoen's curvature estimates for stable minimal surfaces, the sheeted part $\Sigma_{> N_0}$ looks like a large number of almost flat disks around any point $p$ of $\Sigma_{> N_0}$, on the scale of stability $\mathbf{r}_{\mathrm{stab}}(p)$. That decomposition can be compared to the different one in \cite[Lemma 2.15]{C&Mboundedgenus2} for minimal surfaces with bounded second fundamental form.

The definition of the decomposition relies on the stability radius, which is not canonically defined ($\bar{r}$, $\lambda$ can be taken arbitrarily large), however the qualitative properties of the decomposition do not depend on that choice. The results of this section are to be understood to hold whenever the constants $\lambda, \bar{r}, \epsilon, \frac{1}{K}$ of last section are small enough.


\subsection{Genus of $\Sigma_{\leq N_0}$}
A closed subset $X$ of a $2$-dimensional surface $\Sigma$ will be said to have genus at most $\gamma$ if there is a compact $2$-dimensional submanifold with boundary of $\Sigma$ with genus at most $\gamma$ and containing $X$. The genus of a compact surface is the genus of the closed surface obtained by gluing disks along the boundary circles. A planar surface is a surface that can be embedded in the plane; those surfaces have genus $0$.

The non-sheeted part $\Sigma_{\leq N_0}$ has genus bounded linearly by the Morse index as follows:

\begin{theo} \label{rolala}
There is a constant $C$ depending only on $(M,g)$ such that for any integer $N_0>0$, for any closed embedded minimal surface $\Sigma\subset (M,g)$,
$$\genus(\Sigma_{\leq N_0}) \leq CN_0(\ind(\Sigma)+1).$$
\end{theo}

\begin{proof}
The proof essentially follows from the arguments of Sections \ref{sectionbetti}, \ref{sectiongenus}. Let $\mathfrak{C} = \bigcup_{i=1}^LT_i$ denote the almost flat region associated to $\Sigma$ composed of the $(\epsilon,K)$-telescopes $T_i$, as defined in Subsection \ref{almostflat1}.

The boundary of the closed subset $\Sigma_{\leq N_0}$ may not have nice regularity properties,
and we want to find a compact subsurface of $\Sigma$ containing $\Sigma_{\leq N_0}$, with genus at most $CN_0(\ind(\Sigma)+1)$.

Set 
$$\mathcal{Z} := \Sigma_{\leq N_0} \setminus \mathfrak{C}.$$
As in Subsections \ref{mainapf} and \ref{olpo}, using Besicovitch's covering lemma and Corollary \ref{C:counting3D}, the closed set $\mathcal{Z}$ can be covered by $Q'$ intrinsic smooth convex open $2$-disks $D_1,...,D_{Q'}$ inside $\Sigma$ such that each $D_q$ is 
a connected component of some $\Sigma \cap B(x,\mathbf{r}_{\mathrm{stab}}(x))$ where $x\in \mathcal{Z}$, and:
\begin{itemize}
\item any finite intersection of some disks in $\{D_1,...,D_{Q'}\}$ is diffeomorphic to a disk or is empty,

\item the overlap, i.e. the maximal number of disks $D_q$ ($q\in \{1,...,Q'\}$) intersecting non-trivially a given element $D_{q'}$, is bounded by $C=C(M,g)$ (this comes from Besicovitch's covering lemma and (\ref{muuudeux})),
\item given a telescope $T_i$, a disk $D_q$ can intersect at most $1$ connected component of $\Sigma\cap T_i$ (this comes from \ref{sbound}), (\ref{flllllat}) and (\ref{review2})),
\item the number of telescopes intersecting a given disk $D_q$ is at most $2$ (this comes from Besicovitch's covering lemma and Lemma \ref{almostflat} $(iv)$),
\item given a telescope $T_i$, the number of disks $D_q$ intersecting a fixed connected component of $\Sigma\cap T_i$ is bounded by $C=C(M,g)$ (this comes from Besicovitch's covering lemma, Lemma \ref{almostflat} $(iii)$ and \ref{sbound}), (\ref{flllllat}), (\ref{review2})),
\item $$Q'\leq CN_0(\ind(\Sigma)+1),$$ where $C=C(M,g)$ (this comes from  Corollary \ref{C:counting3D} and the definition of $\Sigma_{\leq N_0}$).
\end{itemize}

Next we show a property analogous to $(ii)$ in  Lemma \ref{almostflatT}. We will use the notations of Subsections \ref{defoftelescopes2},\ref{defoftelescopes3}. The constants $\lambda$, $K$, $\epsilon$ will always be chosen depending only on $(M,g)$. Recall that each connected component of $\mathfrak{C}$ is an $(\epsilon,K)$-telescope $T= \bigcup_{i=1}^m A(z_i,r_i,2r_i)$ where the annuli $A(z_i,r_i,2r_i)$ belong to $\bar{\mathcal{A}}_\epsilon^{bis}$ and satisfy (\ref{inininbis}). By Lemma \ref{almostflat} $(i)$, the connected components of $\Sigma\cap \mathfrak{C}$ are disks or annuli. In fact by a similar argument, for $T$ as above, each connected component of $\Sigma \cap (B(z_m,2r_m) \setminus B(z_1,\frac{100r_1}{K}))$ is either an annulus or a disk.
If $\Omega$ is a connected component of $ \Sigma \cap T$, $\Omega$ belongs to a connected component $\hat{\Omega}$ of $\Sigma \cap (B(z_m,2r_m) \setminus B(z_1,\frac{100r_1}{K}))$. We say that $\Omega$ is of disk type (resp. of annulus type) if the extension $\hat{\Omega}$ is a disk (resp. an annulus).

Let 
$$\{d_j\}_j \quad \text{(resp. $\{e_k\}_k$)}$$ 
be the set of connected components of $\Sigma\cap \mathfrak{C}$ of disk type (resp. of annulus type) intersecting non-trivially the closure of the union  $D_1\cup...\cup D_{Q'}$, and let $\hat{d}_j$ be the extension corresponding to $d_j$. In particular,  $\{\hat{d}_j\}$ are disks while $\{e_k\}$ are disjoint annuli. A key point which explains why we introduced $\{d_j\}_j, \{e_k\}_k$ is that they satisfy the following properties when $K$, $\frac{1}{\epsilon}$ are large enough,. First, for any $(\epsilon,K)$-telescope $T$ and any annulus $\beta \in \{e_k\}_k$ with boundary in $\partial T$, it follows from the definition of pointed $\epsilon$-flat annuli that $\beta$ meets both boundary components of $\partial T$ at angles belonging to $(0.4\pi,0.6\pi)$. Moreover the disks in $\{\hat{d}_j\}$ are convex inside $\Sigma$ endowed with the induced metric.

Using those properties and (\ref{sbound}), (\ref{flllllat}), (\ref{review2}), if $\lambda,K,\frac{1}{\epsilon}$ are large enough, then 
for any $q_1,...,q_l\in \{1,...,Q'\}$ and any two different elements $\beta,\beta' \in  \{e_k\}_k$,  
\begin{equation}\label{like (ii)}
\begin{split}
& \bigcap_{i=1}^l D_{q_i} \cap \beta \text{ is diffeomorphic to a disk when non-empty},\\
& \text{and } \beta\cap \beta' =\varnothing.
\end{split}
\end{equation}
As for $\{\hat{d}_j\}$,  if $K$, $\frac{1}{\epsilon}$ are large enough, 
for any $q_1,...,q_l\in \{1,...,Q'\}$, for any elements $\gamma,\gamma' \in  \{\hat{d}_j\}_j$ and $\beta\in  \{e_k\}_k$,  
\begin{equation}\label{like (ii) bis}
\begin{split}
& \bigcap_{i=1}^l D_{q_i} \cap \gamma \text{ is diffeomorphic to a disk when non-empty}\\
& \gamma\cap \gamma' = \gamma\text{ or } \gamma' \text{ or } \varnothing,\quad \gamma\cap \beta =\varnothing\text{ or } \beta.
\end{split}
\end{equation}
The equalities above follow from disjointness of the telescopes $T_i$. In principle, it could happen that say $\gamma' \subset \gamma$ or $\beta\subset \gamma$ when $\gamma$ is the extension of a component of $\Sigma\cap T$, where $T= \bigcup_{i=1}^m A(z_i,r_i,2r_i)$ is a telescope,  and $\gamma$ intersects another strictly smaller telescope $T'$ which can be thought of as lying around a point of $\partial B(z_1,\frac{100 r_1}{K})$. 
Properties (\ref{like (ii)}) and (\ref{like (ii) bis}) serve as a replacement for  Lemma \ref{almostflatT} $(ii)$ in dimension $3$.


Set 
$$\mathcal{R} := D_1\cup...\cup D_{Q'} \cup \bigcup_j \hat{d}_j \cup \bigcup_k e_k.$$
By moving a bit the disks $D_1,...,D_Q'$ we can assume without loss of generality that $\mathcal{R}$ is a compact surface with piecewise smooth boundary. As listed in the properties of $\{D_q\}$ above, any one of the disks $D_q$ intersects at most $2$ different components of $\mathfrak{C}$ (i.e. telescopes); furthermore, a disk $D_q$ can intersect at most $1$ connected component of $\Sigma\cap T_i$ in each given telescope $T_i$. This means that $\{\hat{d}_j\}_j\cup \{e_k\}_k$ has at most $2Q'$ elements. 

Let $\mathcal{M}$ be the subfamily of maximal elements in $\{D_q\}_{q=1}^{Q'} \cup \{\hat{d}\}_j \cup\{e_k\}_k$, namely we remove from $\{D_q\}_{q=1}^{Q'} \cup \{\hat{d}\}_j \cup\{e_k\}_k$ any element entirely contained in another one.
A property of $\{D_q\}$ listed above was that, given a telescope $T_i$, a connected component of $\Sigma \cap T_i$ intersects at most $C=C(M,g)$ disks $D_q$. Combined with (\ref{like (ii)}), (\ref{like (ii) bis}), this implies that $\mathcal{M}$ is a covering by open sets of $\mathcal{R}$, with overlap bounded by  a constant depending only on $(M,g)$, and any intersection of elements of $\mathcal{M}$ is a disk, an annulus or empty. 
Now Lemma \ref{lemmetopo} in the Appendix and the bound on $Q'$ listed in the properties of $\{D_q\}$ above imply that 
$$\genus(\mathcal{R}) \leq C'Q' \leq  C''N_0(\ind(\Sigma)+1)$$
where $C',C''$ depend only on $(M,g)$.

Since $\Sigma_{\leq N_0}\setminus \mathcal{R}$ is at positive distance from $\mathcal{R}$ and included in $\Sigma\cap \mathfrak{C}$ (which is a disjoint union of annuli and disks), $\Sigma_{\leq N_0}\setminus \mathcal{R}$ has genus $0$ and so 
$$\genus(\Sigma_{\leq N_0})\leq C''N_0(\ind(\Sigma)+1)$$
where $C''$ depends only on $(M,g)$.

\end{proof}

\subsection{Area of $\Sigma_{\leq N_0}$}

We will call area of a closed subset of a $2$-dimensional surface $\Sigma$ its $2$-dimensional Hausdorff measure. The area of the non-sheeted part $\Sigma_{\leq N_0}$ can be bounded sublinearly by the Morse index as follows:

\begin{theo} \label{oideal}
There is a constant $C$ depending only on $(M,g)$ so that for any integer $N_0>0$, for any closed embedded minimal surface $\Sigma\subset (M,g)$,
$$\area(\Sigma_{\leq N_0}) \leq CN_0 (\ind(\Sigma)+1)^{\frac{1}{3}}.$$
\end{theo}

\begin{proof}
Let $N_0$ be a positive integer. Let $\mathfrak{C}$ be the almost flat region associated with $\Sigma$ (see Subsection \ref{almostflat1}). In this proof, $C$ denotes a constant only depending on $(M,g)$.

We first estimate the area of $\Sigma_{\leq N_0}\backslash \mathfrak{C}$. We will  use the notations of Subsections \ref{splash} and \ref{mainapf}, and we can assume that $\bar{r}=1/2$. Now as in Subsections \ref{splash} and \ref{mainapf}, we use the Besicovitch covering lemma to get an intger $c_0$ independent of $\Sigma$, and families $\mathcal{B}^{(i)}$ ($i\in \{1,\dots,c_0\}$) of balls of the form $B(x,\mathbf{r}_{\mathrm{stab}}(x))$ with $x\in \Sigma_{\leq N_0}\backslash \mathfrak{C}$, such that:
$$\forall i \in \{1,\dots,c_0\},\forall b_1\neq b_2 \in \mathcal{B}^{(i)}, \quad b_1\cap b_2=\varnothing,$$
$$\text{and}\quad \Sigma_{\leq N_0} \backslash \mathfrak{C} \subset \bigcup_{i=1}^{c_0}\bigcup_{b\in \mathcal{B}^{(i)}} b.$$
We can assume that $\mathcal{B}^{(1)}$ is chosen so that 
\begin{equation} \label{rtt}
\area(\Sigma_{\leq N_0}\cap \bigcup_{b\in \mathcal{B}^{(1)} } b) \geq \frac{1}{c_0} \area(\Sigma_{\leq N_0}\backslash \mathfrak{C}).
\end{equation}
Let $\rad(b)$ be the radius of $b$. By (\ref{sbound}) and the definition of $\Sigma_{\leq N_0}$, for any ball $b\in \mathcal{B}^{(1)}$, the area of $\Sigma\cap b$ is bounded above by $CN_0\rad(b)^2$. Then by (\ref{rtt}) and H\"{o}lder's inequality:
\begin{align*}
\area(\Sigma_{\leq N_0}\backslash \mathfrak{C}) &  \leq CN_0 \sum_{b\in \mathcal{B}^{(1)}} \rad(b)^2 \\
& \leq CN_0 \big(\sharp\mathcal{B}^{(1)}\big)^{1/3}\big(\sum_{b\in \mathcal{B}^{(1)}} \rad(b)^3\big)^{2/3}.
\end{align*}
We have shown in Corollary \ref{C:counting3D} that $\sharp \mathcal{B}^{(1)}$ is bounded above by $C(\ind(\Sigma)+1)$ and $\sum_{b\in \mathcal{B}^{(1)}} \rad(b)^3$ is clearly less than $C\Vol(M,g)$ (recall that the balls in $\mathcal{B}^{(1)}$ are pairwise disjoint). Therefore we obtain
\begin{equation}\label{estim1}
\area(\Sigma_{\leq N_0}\backslash \mathfrak{C}) \leq CN_0(\ind(\Sigma)+1)^{1/3}.
\end{equation}

Next, we estimate the area of $\Sigma_{\leq N_0}\cap \mathfrak{C}$. Let $T_1,\dots,T_L$ be the $(\epsilon,K)$-telescopes (i.e. connected components) in $\mathfrak{C}$, see Definition \ref{2020june}. The closure of each $(\epsilon,K)$-telescope $T_j$ has two boundary  components, one inner $2$-sphere $\partial_-T_j$ and one outer $2$-sphere $\partial_+T_j$. Now choose balls $b_1,\dots,b_L$ of the form $b_j = B(x_j,\mathbf{r}_{\mathrm{stab}}(x_j))$ where $x_j$ is a point of $\partial_+T_j$. By the upper bound in (\ref{review2}), these balls $b_j$ are pairwise disjoint since $\lambda,K$ are taken large.

Observe the following:
\begin{itemize}
\item for an annulus $A(p,t,2t)\in \bar{\mathcal{A}}_{\epsilon}^{bis}$ (see Subsection \ref{almostflat1}), for any point $y \in A(p,t,2t)$, we have
$$t/C\leq \mathbf{r}_{\mathrm{stab}}(y)\leq Ct,$$
and  in particular the radius of $b_j$ is comparable to the radius of $\partial_+T_j$ (see (\ref{review2})),
\item the maximal number of pairwise disjoint balls of the form $B(x,\mathbf{r}_{\mathrm{stab}}(x))$ with $x\in A(p,t,2t)$ is bounded by $C$ (this follows from the first point),
\item if $b'_1,\dots,b'_Q$ are pairwise disjoint balls of the form $B(x,\mathbf{r}_{\mathrm{stab}}(x))$ centered at points of a given $(\epsilon,K)$-telescope $T_j$, then 
$$\sum_{k=1}^Q \rad(b'_k)^2 \leq C\rad(b_j)^2.$$
That follows from the second point and properties of $(\epsilon,K)$-telescopes.
\end{itemize}
Recall that the balls $b_1,\dots,b_L$ are pairwise disjoint.
Moreover for each $j\in\{1,\dots,L\}$,
$$\area(\Sigma_{\leq N_0}\cap T_j) \leq CN_0 \rad(b_j)^2.$$
To see why this last inequality is true, one can apply the Besicovitch covering lemma to the balls $B(x,\mathbf{r}_{\mathrm{stab}}(x))$ with $x$ in the closure of $\Sigma_{\leq N_0}\cap T_j$ and use the previous observations. Hence by H\"{o}lder's inequality,
\begin{align*}
\area(\Sigma_{\leq N_0}\cap \mathfrak{C}) &\leq CN_0\sum_{j=1}^L\rad(b_j)^2 \\
&\leq CN_0 L^{1/3} \big(\sum_{j=1}^L\rad(b_j)^3\big)^{2/3}.
\end{align*}
Since the $b_j$ are pairwise disjoint, $\sum_{j=1}^L\rad(b_j)^3$ is bounded above by $C\Vol(M,g)$. Moreover by Corollary \ref{C:counting3D}, $L\leq C(\ind(\Sigma)+1)$, so we derive the following estimate:
\begin{equation}\label{estim2}
\area(\Sigma_{\leq N_0}\cap \mathfrak{C}) \leq CN_0(\ind(\Sigma)+1)^{1/3}.
\end{equation}
Putting (\ref{estim1}) and (\ref{estim2}) together, we conclude the proof.

\end{proof}

Recall that the sheeting number $N(\Sigma)$ of a minimal surface $\Sigma$ is the smallest integer $N_0$ such that $\Sigma=\Sigma_{\leq N_0}$. We end this section with a few remarks. If $N(.)$ happens to be uniformly bounded along a sequence of minimal surfaces, then Theorem \ref{oideal} and Ejiri-Micallef \cite{EjiriMicallef} imply that 
$$\area\leq C(\genus+1)^{1/3}.$$ 
In fact Theorem \ref{oideal} is optimal: H. Schwarz's triply periodic minimal surfaces give examples of minimal surfaces in a flat $3$-torus for which $N(.)$ is uniformly bounded and $\area$, $\ind^{1/3}$, go to infinity at the same speed, while Lawson's minimal surfaces in $\mathbb{S}^3$ \cite{Lawson} show that in general $N(.)$ is unbounded and $\area\leq C(\ind+1)^{1/3}$ does not hold. Theorem \ref{oideal} more generally yields a stronger result: for any sequence of minimal surfaces $\{\Sigma_k\}$ for which 
$$\lim_{k\to\infty} \area(\Sigma_k) =\infty \quad \text{and} \quad \ind(\Sigma_k) = o(\area(\Sigma_k)^3),$$ for any integer $N_0$, we have 
\begin{equation} \label{hint}
\lim_{k\to\infty}\frac{\area(\Sigma_{k,\leq N_0})}{\area(\Sigma_k)} = 0.
\end{equation}
The infinite sequence of minimal surfaces $\{\Gamma^{(p)}\}$ that we constructed in compact manifolds with strictly stable minimal boundary \cite[Theorem 10]{AntoineYau} generically satisfy these assumptions on the area and the index\footnote{See \cite{Song-Zhou20} where it is proved that these minimal surfaces generically  ``scar'' along stable surfaces.}. Thus Theorem \ref{oideal} sheds light on the geometric properties of the minimal surfaces $\{\Gamma^{(p)}\}$, which behave very differently from the min-max minimal surfaces constructed in \cite{MaNeinfinity,GasGua}. See section \ref{open} for further questions on the sheeting number $N(\Sigma)$.

\section{Singular set and Morse index in dimensions at least $8$} \label{undostres}

In this section we prove Theorem \ref{mainb}. The estimates on the singular set of area minimizing hypersurfaces of Naber-Valtorta \cite{NaberValtorta} will replace curvature estimates \cite{Sc,SchoenSimon} which were so useful in previous sections. The results of Naber-Valtorta \cite[Theorem 1.6 and Section 9]{NaberValtorta} actually hold for minimal hypersurfaces  smooth outside a set of Hausdorff codimension at least $7$ which are stable (instead of area minimizing), since only the $\epsilon$-regularity property is used \cite[Theorem 2.11]{NaberValtorta}, and the latter holds for stable hypersurfaces: let $\Sigma^n\subset B(p,2)\subset (\mathbb{R}^{n+1},g)$ be a stable embedded minimal hypersurface smooth outside a singular set of Hausdorff dimension at most $n-7$, with $\partial \Sigma\cap B(p,2) =\varnothing$, $\|g-g_{\text{Eucl}}\|_{C^5}\leq \epsilon_0$,
$\mathcal{H}^n(\Sigma)\leq \Lambda$, then there exists $\epsilon=\epsilon(n,\Lambda)>0$ such that if $\epsilon_0\leq \epsilon$ and $B(p,2)$ is $(n-7,\epsilon)$-symmetric, then $\Sigma$ is smooth around $p$ and
$$\max\{0\leq r\leq 1;\quad  \sup_{B(p,r)} |\mathbf{A}| \leq r^{-1}\} \geq 1,$$
see \cite[Definition 1.1, Theorem 2.11]{NaberValtorta} for the definition of being $(n-7,\epsilon)$-symmetric. This follows from the more general compactness and regularity theorem of Schoen-Simon \cite[Theorem 2, Theorem 1]{SchoenSimon}.

\subsection{The case of dimension $8$}

It happens that in order to prove Theorem \ref{mainb} in dimension 8, we can use arguments almost identical to those involved in the proof of Theorem \ref{maina}. Consider a closed manifold $(M^8,g)$ and a minimal hypersurface $\Sigma \subset M$ embedded smoothly outside of finitely many points, and of finite index (see the introduction for why these objects are natural). 
Let $\bar{r}>0$.
In Section \ref{sectionbetti}, we conveniently used a large $\lambda>2$ while here we can just take $\lambda =2$. Let $\mathbf{r}_{\mathrm{stab}}$ be the stability radius defined by (see Subsection \ref{stabrad})
$$\mathbf{r}_{\mathrm{stab}}(x):=\sup\{r\leq\bar{r}; B(x,2 r) \cap \Sigma \text{ is stable}\}.$$

The following lemma, which holds in all dimensions $n+1$, ensures that $\mathbf{r}_{\mathrm{stab}}$ is bounded away from $0$.
\begin{lemme}\label{boundedaway}
If $\Sigma$ is a minimal hypersurface in $(M^{n+1},g)$ whose singular set has dimension at most $n-7$ and has finite Morse index, then for all $x\in M$, we have $\mathbf{r}_{\mathrm{stab}}(x)>0$.
\end{lemme}

\begin{proof}
If $x$ is not on $\Sigma$ or is a smooth point of $\Sigma$ then the lemma is clear. If $x$ is in the singular set of $\Sigma$, suppose towards a contradiction that for any small $a>0$, $\Sigma\cap B(x,a)$ is unstable. Fix $a_0>0$ small enough, then by the usual cut-off  argument, we get a number $a_1<a_0$ so that $\Sigma\cap (B(x,a_0)\backslash \bar{B}(x,a_1))$ is unstable. We can continue with $a_1$ instead of $a_0$, etc, and get a decreasing infinite sequence of positive radii $\{a_k\}$ so that every piece $\Sigma\cap (B(x,a_k)\backslash \bar{B}(x,a_{k+1}))$ is unstable. \cite[Lemma 3.1]{Sharp} would imply that $\Sigma$ has infinite Morse index, a contradiction.
\end{proof}

Let $A>0$. By Naber-Valtorta \cite[Theorem 1.6]{NaberValtorta}, in dimension $n+1=8$ there is a constant $c_{NV}$  depending only on $(M,g)$ and $A$, so that for any closed minimal hypersurface $\Sigma$ as above with $7$-volume at most $A$, for any $x\in M$,
\begin{equation} \label{nv}
\sharp \Sing\big(\Sigma \cap B(x,\mathbf{r}_{\mathrm{stab}}(x))\big) \leq c_{NV}
\end{equation}
where $\sharp$ denotes as before the number of elements in a set, and $\Sing$ denotes the (finite) singular set.

Let $\mathbf{f}(\Sigma)$ be the folding number of $\Sigma$ (see Subsection \ref{folding}).
As already observed previously, the maximal number of disjoint balls of the form $B(x,2\mathbf{r}_{\mathrm{stab}}(x))$ with $\mathbf{r}_{\mathrm{stab}}(x)<\bar{r}$, at positive distance to one another, is at most $\mathbf{f}(\Sigma)$, which itself is at most $\ind(\Sigma)$.

In dimension $8$, $\mathbf{r}_{\mathrm{stab}}$ is still continuous:
\begin{lemme} 
Given $(M,g)$, $\Sigma$ as above, the stability radius $\mathbf{r}_{\mathrm{stab}}:M\to (0,\bar{r}]$ is a $\frac{1}{2}$-Lipschitz function.
\end{lemme}

\vspace{1em}

\subsubsection{Almost conical region $\mathfrak{C}$}

Let $A>0$ and define $\bar{A}$ correspondingly as in Subsection \ref{almostflat2}.
We will define an almost conical region $\mathfrak{C}$ as in Subsection \ref{almostflat2}. First, in the next lemma, we set $\bar{r}=\infty$ so that $\mathbf{r}_{\mathrm{stab}}$ can take values in $(0,\infty]$.

\begin{lemme} \label{aroundtheworld}
Let $\bar{A}>0$. If $\beta_0$ is large enough, for any minimal cone $\Gamma\subset \mathbb{R}^8$ tipped at $0$ smooth outside of $0$, with $\Theta_{\text{Eucl}}(0,1)\leq \bar{A}$ and with
$$\max_{y\in \Gamma\cap \partial B_{\text{Eucl}(0,1)} } |\mathbf{A}(y)| \geq \beta_0,$$
there are  two balls $b_1,b_2 \subset B_{\text{Eucl}}(0,2)\backslash B_{\text{Eucl}}(0,1/2)$ of the form 
$$b_1=B(x,\mathbf{r}_{\mathrm{stab}}(x)),\quad b_2=B(\frac{4x}{3},\mathbf{r}_{\mathrm{stab}}(\frac{4x}{3}))$$
such that $8b_1\cap 8b_2=\varnothing$ and the radii of $b_1,b_2$ are smaller than $2^{-1000}$. 
\end{lemme}
\begin{proof}
Let $\{\Gamma_m\}_{m>0}$ be a sequence of cones as in the statement but also with 
\begin{equation} \label{watts}
\max_{y\in \Gamma\cap \partial B_{\text{Eucl}(0,1)} } |\mathbf{A}(y)| \geq m.
\end{equation}
By the cone structure, it suffices to prove that $\mathbf{r}_{\mathrm{stab}}$ cannot be bounded away from $0$ on $\Gamma_m\cap \partial B_{\text{Eucl}}(0,1)$. If this was the case, then by the work of Schoen-Simon \cite{SchoenSimon} on stable minimal hypersurfaces, $\Gamma_m$ would subsequently converge in the sense of varifolds to a minimal cone $\Gamma_\infty$ tipped at $0$ with isolated singularities since $n=7$. In particular $\Gamma_\infty$ is smooth outside of $0$. Since $\mathbf{r}_{\mathrm{stab}}$ is bounded below away from $0$ along the sequence on $\Gamma_m\cap \partial B_{\text{Eucl}(0,1)}$, the sheeting theorem of Schoen-Simon \cite[Theorem 1]{SchoenSimon} implies that $\Gamma_m$ had second fundamental form uniformly bounded on points of $\partial B_{\text{Eucl}}(0,1)$  which contradicts (\ref{watts}).  
\end{proof}

A reverse bound like (\ref{reverse1}) still holds after possibly rescaling the metric $g$:
if for a point $x$ we have $\mathbf{r}_{\mathrm{stab}}(x) < \bar{r}$ then the ball $B(x,4\mathbf{r}_{\mathrm{stab}}(x)) $ is unstable so either $B(x,4\mathbf{r}_{\mathrm{stab}}(x))$ contains a singular point or 
$$\text{ there is } y'\in B(x,4\mathbf{r}_{\mathrm{stab}}(x)) \text{ such that  } \mathbf{r}_{\mathrm{stab}}(x) |\mathbf{A}(y')| > \varepsilon'$$
for some $\varepsilon'>0$ independent of $\Sigma$.

Let $\delta,K>0$. We define $\mathcal{G}_{\beta_0}$, ``$\delta$-close to a cone $\Gamma\in \mathcal{G}_{\beta_0}$'', ``pointed $\delta$-conical annulus'', $\mathcal{A}_\delta$, $\mathcal{A}^{bis}_\delta$, $(\delta,K)$-telescopes  as in Subsection \ref{almostflat2}. Part of Lemma \ref{almostflatT} holds in dimension $8$:

\begin{lemme} 

Let $(M^{8},g)$ be a closed 8-manifold and $A>0$.
If the constants $\varepsilon_{\mathrm{SS}}, \frac{1}{\lambda}, \bar{r}, \frac{1}{K}, \delta$ are chosen small enough then the following is true.

Let $(M^8,g)$ be a closed manifold of dimension $8$ and $\Sigma$ a closed minimal hypersurface embedded smoothly outside of finitely many points, of finite index. Then there exist disjoint $(\delta,K)$-telescopes $T_1,...,T_L$ such that $T_i$ are at positive distance from one another, and
$$\bigcup_{An\in \mathcal{A}_\delta} An \subset \bigcup_{i=1}^L T_i.$$

\end{lemme}

\begin{proof}
The proof is very similar to the one of Lemma \ref{almostflatT}, except that we use Lemma \ref{boundedaway} to get that $\mathbf{r}_{\mathrm{stab}}$ is bounded away from $0$ on $\Sigma$ and that
$$\hat{s}:=\inf\{s; \exists p , A(p,s,2s) \in \mathcal{A}_\delta\}>0.$$
\end{proof}

We can finally define the almost conical region $\mathfrak{C}$ as in Subsection \ref{almostflat2}:
$$\mathfrak{C}:=\bigcup_{i=1}^L T_i.$$
Note that by definition, $\Sigma \cap \mathfrak{C}$ is a \emph{smooth} hypersurface.
We fix $A$, $\bar{A}$, $\delta$, $K>2^{1000}$, $\bar{K}:=30\lambda K^2$.

\vspace{1em}

\subsubsection{No curvature concentration far from $\mathfrak{C}$}

With these definitions in place, one can check that the following analogue of Theorem \ref{noconcentration} holds for $n+1=8$, replacing $\lambda$ by $2$ in the statement. In the statement of Theorem \ref{noconcentration8}, we set $\bar{r}=\infty$ so that $\mathbf{r}_{\mathrm{stab}}$ can take values in $(0,\infty]$.
 
\begin{theo} \label{noconcentration8}
Let $n+1=8$. Set $\bar{K}:= 60K^2$.

There exist $\beta_1>1$, $\mu>0$ and $\bar{R}>1000$ depending only on $\bar{A},\delta,K$ such that the following is true. Let $g$ be a metric $\mu$-close to the Euclidean metric in the $C^5$-topology on $B_{Eucl}(0,\bar{K})\subset \mathbb{R}^8$, and let
$$(\Sigma, \partial \Sigma) \subset (B_{Eucl}(0,\bar{K}),\partial B_{Eucl}(0,\bar{K}))$$
be a compact embedded minimal hypersurface with respect to $g$, such that
\begin{itemize}
\item the $7$-volume of $\Sigma\cap B(0,\bar{K}) $ is at most $\bar{A} \bar{K}^7/2$,
\item $\mathbf{r}_{\text{stab}}(0)< \bar{K}^{-1}$,
\item there is $y'\in B(0,14)$ with $\mathbf{r}_{\text{stab}}(y')=1$,
\item $\Theta_g(0,20K) - \Theta_g(0,\frac{1}{3K}) \leq \beta_1^{-1}$.
\end{itemize}

Then 
\begin{enumerate} [label=(\roman*)]
\item either there are two balls $b_1= B(z',\mathbf{r}_{\text{stab}}(z'))$, $b_2 = B(z'',\mathbf{r}_{\text{stab}}(z''))$ such that
$$ z',z'' \in B(0,14K),\quad 6b_1 \cap 6b_2 = \varnothing,$$
$$\mathbf{r}_{\text{stab}}(z'), \mathbf{r}_{\text{stab}}(z'') \in [2^{-({R}+1)}, 2^{-R}) \text{ for some $R\in [1000,\bar{R}]$},$$
\item or $\Sigma \cap A(0,\frac{1}{2K},14K)$ is $\delta$-close to a cone $\Gamma\in \mathcal{G}_{\beta_0}$.
\end{enumerate}
\end{theo}

\begin{proof}
The proof is the same as Theorem \ref{noconcentration}, but instead of using the curvature bounds of Schoen-Simon in dimensions $3$ to $7$, we use their sheeting theorem as in the proof of Lemma \ref{aroundtheworld}. 
\end{proof}

As a consequence, the analogue of Corollary \ref{sotechnical} also holds true, as well as Proposition \ref{triplybis}, Proposition \ref{doublybis}, and Corollary \ref{C:counting} whose proofs are combinatorial (always replacing $\lambda$ by $2$).

\subsubsection{End of proof in dimension 8}

We can now complete the proof of Theorem \ref{mainb} in dimension $8$ as follows. Let $\Sigma\subset (M^8,g)$ be an embedded closed minimal hypersurface with a finite singular set, of finite Morse index and $7$-volume at most $A$. 
Let $\mathfrak{C}$ the almost conical region of $M$ defined above. In particular $\Sigma \cap \mathfrak{C}$ is smooth and does not contain any singular point.
Recall that $\mathbf{r}_{\mathrm{stab}}(x)>0$ for any $x$ by Lemma \ref{boundedaway}. By the Besicovitch lemma as in Subsection \ref{splash}, and by using (\ref{nv}), we get a finite family $\mathcal{B}^{(1)}$ of pairwise disjoint geodesic balls of the form $B(x,\mathbf{r}_{\mathrm{stab}}(x))$ such that $x\in\Sigma \backslash \mathfrak{C}$ and
$$\sharp \Sing(\Sigma) \leq c_2  \sharp \mathcal{B}^{(1)},$$
for some $c_2$ only depending on $(M,g)$ and $A$.

Applying Corollary \ref{C:counting} here (with $\lambda =2$), we conclude that 
$$\sharp \Sing(\Sigma) \leq c_2  \sharp \mathcal{B}^{(1)} \leq c_3 (\mathbf{f}(\Sigma) +1) \leq c_3 (\ind(\Sigma) +1)$$
for some $c_3$ only depending on $(M,g)$ and $A$, and Theorem \ref{mainb} is proved in dimension $8$.

\vspace{1em}

\subsection{The case of dimension $n+1\geq 9$} \label{plus grand que 9}

Let $\Sigma\subset M^{n+1}$ be a closed minimal hypersurface with $n$-volume bounded above by $A>0$, where $n\geq 8$. We assume $\Sigma$ to be embedded outside of a subset of dimension at most $n-7$. We rescale the metric so that $\bar{r} = \frac{1}{2}$. Let $\Sing(\Sigma)$ be its singular set. In high dimensions, we do not need to introduce an almost conical region $\mathfrak{C}$.

Let $\lambda=2$ and let $\mathcal{F}_k$ be defined as before.
We will make use of the notion of \emph{basis} of certain families of balls, whose definition is recalled here: let $\mathcal{B}'$ be a finite family of balls such that for all $k$ and $b\neq b'\in \mathcal{B}'\cap \mathcal{F}_k$, $6b\cap 6b'=\varnothing$, then a basis $\mathbf{B}'$ of $\mathcal{B}'$ is a subfamily such that for all $i$, $\mathbf{B}'\cap \mathcal{F}_i$ is a maximal subset of $\mathcal{B}'\cap \mathcal{F}_i$ satisfying for all $j<i$: 
$$\forall b\in \mathbf{B}'\cap \mathcal{F}_i, \quad \forall b'\in \mathcal{B}'\cap \mathcal{F}_j, \quad 6b\cap 6b'=\varnothing.$$

Here are some elementary properties of the notion of basis. Let the family $\mathcal{B}'$ and the basis $\mathbf{B}'$ be as above.
First, if $B\in \mathbf{B}'$ then $B$ is locally the largest ball in the following sense. Let $\beta\in \mathcal{B}'$. Let $j_B,j_\beta$ be the integers such that $B\in \mathcal{F}_{j_B} $ and $\beta \in  \mathcal{F}_{j_\beta} $,  then we have the implication:
\begin{equation} \label{implication}
\text{ if } 6B \cap 6\beta \neq \varnothing \text{ then } j_B < j_\beta.
\end{equation}
Secondly, by definition of basis, a ball $b_0\in  \mathcal{B}' \cap \mathcal{F}_{i_0}$ is not in  $\mathbf{B}'$ if and only if there are a sequence of strictly decreasing integers $i_0 > i_1>...> i_k \geq 0$ (where $k$ is a positive integer) and a chain of balls $b_1\in \mathcal{B}' \cap \mathcal{F}_{i_1}, ..., b_k\in \mathcal{B}' \cap \mathcal{F}_{i_k}$, so that
\begin{equation} \label{by definition of basis}
\forall \alpha\in \{0,...,k-1\},\quad 6b_{\alpha} \cap 6b_{\alpha +1} \neq \varnothing, \quad \text{ and } b_k \in \mathbf{B}'.
\end{equation}

\vspace{2em}

Let $\mathbf{r}_{\mathrm{stab}}$ be defined as before (with $\lambda =2$). By Lemma \ref{boundedaway}, $\mathbf{r}_{\mathrm{stab}}(x)>0$ for all $x$. By using a covering with balls of the form $B(x,\mathbf{r}_{\mathrm{stab}}(x))$, the Besicovitch lemma as in  Subsection \ref{splash} and an extraction argument as for (\ref{(23)}) in the proof of Lemma \ref{C:counting}, we get a finite family $\bigcup_{k\geq 0}\mathcal{B}_k$ of disjoint balls of the form $B(x,\mathbf{r}_{\mathrm{stab}}(x))$, $x\in \Sigma$, with $\mathcal{B}_k\subset \mathcal{F}_k$, such that for all $k$ and $b\neq b'\in \mathcal{B}_k$:
\begin{equation} \label{varnot}
6b\cap 6b' =  \varnothing,
\end{equation}
and for which there is a constant $c_1$ depending only on $(M,g)$ satisfying
$$\mathcal{H}^{n-7}(\Sing(\Sigma)) \leq c_1 \sum_{k\geq 0}\sum_{b\in \mathcal{B}_k} \mathcal{H}^{n-7}(\Sing(\Sigma)\cap b).$$
But since for each $b\in \mathcal{B}_k$ and any $\hat{\epsilon}\in(0,1)$, $\Sigma\cap 2(1-\hat{\epsilon})b$ is stable, the work of Naber-Valtorta \cite[Theorem 1.6]{NaberValtorta} implies that there is a constant $c_2$ depending only on $(M,g)$ and $A$ so that for all $k$, $b\in \mathcal{B}_k$, if $\rad(b)$ denotes the radius of $b$:
$$\mathcal{H}^{n-7}(\Sing(\Sigma)\cap b) \leq c_2 \rad(b)^{n-7}.$$

Fix an integer $m>1000$.
We can suppose that for a constant $c_3$ depending only on $(M,g)$, $A$ and $m$, and for an integer $k_0\in[0,m-1]$,  
\begin{equation} \label{bz}
\mathcal{H}^{n-7}(\Sing(\Sigma)) \leq c_3 \sum_{j\geq 0}\sum_{b\in \mathcal{B}_{k_0+mj}} \rad(b)^{n-7}.
\end{equation}
We can suppose that $k_0=0$ since the other cases are treated in the same way. 

\vspace{1em}
Let us modify once more the family $\bigcup_{j\geq 0}\mathcal{B}_{mj}$ by removing some of its elements. The modification involves finitely many steps: we start with a basis $\mathbf{B}_0$ of $\mathcal{B}^0:=\bigcup_{j\geq 0}\mathcal{B}_{mj}$. Let $B\in \mathbf{B}_0$, let $\mathcal{I}(B)\subset \mathbb{N}$ be the set of integers such that for $j \in \mathcal{I}(B)$, there is at most one ball $b\in \mathcal{B}^0\cap \mathcal{F}_{mj}$ with $6b\subset3B$.
We set 
$$\mathcal{B}^{1}: = (\mathcal{B}^{0}\backslash  \mathbf{B}_0)\backslash \{b; \quad  \exists B \in \mathbf{B}_0, \exists k_0\in\mathcal{I}(B),  b\in \mathcal{B}^{0}\cap \mathcal{F}_{mk_0} \text{ and } 6b\subset 3B\}.$$
Let $\mathbf{B}_1$ be a basis of $\mathcal{B}^{1}$. We can repeat this procedure with $\mathbf{B}_1,\mathcal{B}^{1}$ replacing $\mathbf{B}_0,\mathcal{B}^{0}$, remove some elements from $\mathcal{B}^{1}$, etc. 
Eventually we get a sequence of families $\mathcal{B}^{0},\mathcal{B}^{1},\mathcal{B}^{2},...,\mathcal{B}^{l}$ with associated bases $\mathbf{B}_0,\mathbf{B}_1,\mathbf{B}_2,...,\mathbf{B}_l$ where $l $ is the last integer for which $\mathcal{B}_{ml} \neq \varnothing$. 

Consider $\bigcup_{j\geq0} \mathbf{B}_{j}$. Note that at each step in the above procedure, we remove balls $b$ such that $6b$ is entirely contained in $3B$ for some $B$ in $\bigcup_{j\geq 0} \mathbf{B}_{j}$. Thus for each $i\geq 0$, $\mathbf{B}_i$ is a basis of $\mathcal{B}^{i+1} \cup \mathbf{B}_i$, and $ \mathbf{B}_{i}$ is also a basis of $\bigcup_{j\geq i} \mathbf{B}_{j}$.

The point of the above construction is the following lemma\footnote{I thank Giada Franz and Santiago Cordero Misteli for correcting a mistake in the original version of its proof, and the referees for pointing out the same mistake.}.
\begin{lemme}\label{fin}
For each $i\geq 0$ and $B\in \mathbf{B}_i$, 
\begin{enumerate}
\item either there is no ball $b\in \mathbf{B}_{j}$ with $j>i$ and $6b\cap 2B\neq \varnothing$,
\item or there are two ball $b,b'\in \mathbf{B}_{i+1}$ with $6b\cap 6b'=\varnothing$, $6b\subset 4B$, $6b'\subset 4B$. 
\end{enumerate}
\end{lemme}
\begin{proof}
Let $j_B$ be the integer such that $B\in \mathbf{B}_i\cap  \mathcal{F}_{mj_B}$. 
To check this lemma, suppose that Item (1) does not happen, then there is ${b_0}\in \bigcup_{j\geq i+1} \mathbf{B}_{j} $ such that $6b\cap 2B\neq \varnothing$.  Let $j_0$ be the integer such that $b_0\in \mathcal{F}_{mj_0}$. 
Then since $\mathbf{B}_i$ is a basis of $ \bigcup_{j\geq i} \mathbf{B}_{j}$, necessarily $j_0>j_B$ by (\ref{implication}).
Since $\mathbf{B}_{i+1}$ is a basis of $\mathcal{B}^{i+1}$ (which was defined before the lemma), by (\ref{by definition of basis}) there is a chain of balls $b_1\in \mathcal{B}^{i+1}\cap \mathcal{F}_{mj_1},...,b_k\in \mathcal{B}^{i+1}\cap \mathcal{F}_{mj_k}$ with $j_0>j_1>...>j_k$ so that 
$$\forall \alpha\in \{0,...,k-1\},\quad 6b_\alpha\cap 6b_{\alpha+1} \neq \varnothing, \quad \text{and } b_k \in \mathbf{B}_{i+1}.$$
We necessarily have $j_k>j_B$. Indeed if not, let $l\in \{1,...,k\}$ be the first index for which $j_l \leq j_B$. Since $m>1000$, we have $6b_{l-1} \subset 3B$ which implies that $6b_l\cap 6B \neq \varnothing$, contradicting (\ref{implication}), and the fact that $B\in \mathbf{B}_i$ and $\mathbf{B}_i$ is a basis of $\mathcal{B}^{i+1} \cup \mathbf{B}_i$. So $j_k>j_B$, and since $m>1000$, we have after setting $\hat{b}:=b_k \in \mathbf{B}_{i+1}\cap \mathcal{F}_{mj_{k}}$:
$$6\hat{b}\subset 3B.$$

By definition of $\mathcal{B}^{i+1}$, there must be another ball $b'_0\in \mathcal{B}^{i+1}\cap \mathcal{F}_{mj'_0}$ where $j'_0=j_k$, such that 
$$6b'_0\subset 3B \text{ and } 6\hat{b} \cap 6b'_0=\varnothing$$
(the equality above is just (\ref{varnot})). 
With the same arguments as the previous paragraph, using (\ref{by definition of basis}),
there is $\hat{b}' \in \mathbf{B}_{i+1} \cap \mathcal{F}_{mj'_{k'}}$ with
\begin{equation}\label{hatb2}
6\hat{b}' \subset 4B \text{ and } 6\hat{b} \cap 6\hat{b}'  =\varnothing.
\end{equation}
The equality above follows from (\ref{implication}) and the following: $\hat{b}'\in \mathcal{B}^{i+1}\cap \mathcal{F}_{mj'_{k'}}$, $b'_0\in \mathbf{B}_{i+1} \cap \mathcal{F}_{mj'_0}$ and  $j'_{k'}\leq j'_0=j_k$, while $\hat{b}$ belongs to the basis $\mathbf{B}_{i+1}$ of $\mathcal{B}^{i+1}$.

To conclude, $6\hat{b}\cap 6\hat{b}'=\varnothing$, $6\hat{b}\subset 4B$, $6\hat{b}'\subset 4B$ and so Item (2) is valid.
\end{proof}

Before continuing, we remark the following basic fact for $n>7$: 
$$\forall i\geq 0, \sum_{j\geq i+1} (2^{-mj})^{n-7} = \frac{1}{1-2^{-(n-7)m}} (2^{-m(i+1)})^{n-7}\leq 2 (2^{-m(i+1)})^{n-7}.$$
Hence, by construction of the $\mathbf{B}_{i}$:
\begin{equation} \label{by}
\sum_{j\geq 0}\sum_{b\in \mathcal{B}_{mj}} \rad(b)^{n-7} \leq 3 \sum_{b\in \bigcup_{i\geq 0}\mathbf{B}_{i}} \rad(b)^{n-7}.
\end{equation}

To $\bigcup_{i\geq 1}\mathbf{B}_{i}$ we associate a (not necessarily connected) tree $T$ constructed as follows: each element of $\bigcup_{i\geq 1}\mathbf{B}_{i}$ is a vertex, and for each $i\geq 0$, for a $B\in \mathbf{B}_i$ such that (2) in Lemma \ref{fin} occurs,  choose two balls $b,b'\in \mathbf{B}_{i+1}$ as in (2), and join $B$  to $b,b'$ with an edge for each. This indeed defines a tree $T$ and let $\{b_a\}_{a=1}^L$ be the balls corresponding to the leaves (vertices for which there are no $b,b'$ as above) of $T$. One checks that for $a\neq a'$, $\dist(2b_a,2b_{a'})>0$, and since any open region of $\Sigma$ containing the closure of $2b_a\cap \Sigma$ is unstable, this means that
\begin{equation} \label{bx}
\ind(\Sigma)\geq \mathbf{f}(\Sigma)\geq L.
\end{equation}
Now since each vertex of $T$ either has no descendants or has two of them, 
\begin{equation} \label{bxs}
L\geq \frac{1}{2}\sharp T = \frac{1}{2} \sum_{i\geq 1} \sharp \mathbf{B}_i
\end{equation}
where $\sharp T$ is the total number of vertices in $T$, which is $\sum_{i\geq 1} \sharp \mathbf{B}_i$.

By H\"{o}lder's inequality,
$$
\sum_{b\in \bigcup_{i\geq 0}\mathbf{B}_{i}} \rad(b)^{n-7} \leq \big(\sum_{{i\geq 0}}\sharp \mathbf{B}_{i}\big)^{\frac{7}{n}} \big(\sum_{b\in \bigcup_{i\geq 0}\mathbf{B}_{i}} \rad(b)^{n}\big)^{\frac{n-7}{n}}.
$$
We have just seen in (\ref{bx}), (\ref{bxs}) that there is $c_5$ depending only on $(M,g)$ so that
$$\sum_{{i\geq 0}}\sharp \mathbf{B}_{i} =\sharp \mathbf{B}_{0} + \sum_{{i\geq 1}} \sharp \mathbf{B}_{i} \leq c_5(\mathbf{f}(\Sigma)+1)\leq c_5(\ind(\Sigma)+1).$$ 
Moreover, since the balls of $\bigcup_{k\geq 0}\mathcal{B}_k$ are disjoint and centered at points of the minimal hypersurface $\Sigma$, by the monotonicity formula
$$\sum_{b\in \bigcup_{i\geq 0}\mathbf{B}_{i}} \rad(b)^{n}\leq c_5 \mathcal{H}^{n}(\Sigma) \leq c_5A.$$

Putting these estimates together with (\ref{bz}), (\ref{by}), we get a constant $C_A$ depending only on $(M,g)$ and the area upper bound $A$ such that
$$\mathcal{H}^{n-7}(\Sing(\Sigma)) \leq C_A (\mathbf{f}(\Sigma)+1)^{\frac{7}{n}}\leq C_A (\ind(\Sigma)+1)^{\frac{7}{n}}.$$
This finishes the proof when $n>7$.

\section{Open problems}\label{open}
We conclude with some conjectures and questions that naturally arose during our work.

\subsection{A general relation for area, index and genus in dimension $3$} 
In what follows, $\Ric$ (resp. $R$) denotes as usual the Ricci (resp. scalar) curvature. If $f_1,f_2$ are two functions defined on the set of closed minimal surfaces, write $f_1 \approx f_2$ if there is a constant $C>0$ depending only on $(M,g)$ so that $C^{-1} f_1\leq f_2\leq C f_1$ and define $\lesssim$ similarly. We conjecture the following, which generalizes the conjecture of Marques-Neves-Schoen (stated in the Introduction) and our results concerning bounded area minimal surfaces (see Theorems \ref{maina}, \ref{rolala}, \ref{oideal}). \newline

$\mathbf{C}_1:$
Let $(M^3,g)$ be a closed $3$-manifold. There exists a constant $C$ depending only on $(M,g)$ such that for any closed embedded minimal surface $\Sigma$,
\begin{enumerate}
\item $\genus(\Sigma)+ \area(\Sigma)  \approx  \ind(\Sigma)+ \area(\Sigma) $,
\item $\genus(\Sigma)+ \area(\Sigma) \approx \ind(\Sigma)+1$ if $R>0$,
\item $\genus(\Sigma) +1\approx \ind(\Sigma)$ if $\Ric>0$. \newline
\end{enumerate}

Note that Ejiri-Micallef \cite{EjiriMicallef} already proved half of these inequalities by showing that $\ind(\Sigma) \lesssim \genus(\Sigma)+\area(\Sigma)$, and if $\Ric\geq c>0$, Choi-Schoen \cite{ChoiSchoen} proved that $\area(\Sigma)\lesssim \genus(\Sigma) +1$ so Item (3) above is just a reformulation of the conjecture of Marques-Neves-Schoen (\ref{conjMNS}).  \cite{ChodKetMax} implies that for each item, if the right-hand side is bounded, then so is the left-hand side. The relevance of Item (1) for min-max theory is that the sequence of minimal surfaces $\{\Sigma_p\}$ produced by \cite{MaNeinfinity, GasGua, ChodMant} in any closed $3$-manifold should generically have genus growing linearly in $p$. This is reminiscent of Yau's conjecture for the size of nodal sets of eigenfunctions.

\subsection{Quantitative estimates for complexities of stable minimal hypersurfaces} 
In this paper, we showed that to get quantitative estimates for minimal hypersurfaces of high Morse index, it ``suffices'' in some sense to get quantitative estimates for stable minimal hypersurfaces. We used Schoen-Simon \cite{SchoenSimon} and Naber-Valtorta \cite{NaberValtorta}, which are estimates for uniformly bounded $n$-volume stable minimal hypersurfaces. How do these estimates depend on the $n$-volume of the stable minimal hypersurfaces?  Here are two concrete questions: \newline

$\mathbf{Q}_1:$
In dimensions $4 \leq n+1\leq 7$, if $(\Sigma,\partial \Sigma)\subset (B_{\text{Eucl}}(0,2),\partial B_{\text{Eucl}}(0,2) )\subset \mathbb{R}^{n+1}$ is a stable minimal hypersurface, is the total Betti number of $\Sigma \cap B_{\text{Eucl}}(0,1)$ bounded above linearly in terms of the $n$-volume of $\Sigma$?
\newline 

$\mathbf{Q}_2:$
In dimensions $n+1\geq 8$,  if $(\Sigma,\partial \Sigma)\subset (B_{\text{Eucl}}(0,2),\partial B_{\text{Eucl}}(0,2) )\subset \mathbb{R}^{n+1}$ is a stable minimal hypersurface with codimension $7$ singular set, is the $n-7$-dimensional Hausdorff measure of the singular set of $\Sigma \cap B_{\text{Eucl}}(0,1)$ bounded above linearly in terms of the $n$-volume of $\Sigma$?\newline

For $\mathbf{Q}_1$ in dimension $n+1=3$, the answer is positive and follows from Schoen's curvature bound for stable minimal surfaces \cite{Sc}, (see also Maximo \cite{Maximonote}, and alternatively Theorem \ref{maina}).

\subsection{Examples and counterexamples in higher dimensions} 
Answering the two following conjectures would be highly desirable. The first one is related to the well-known scarcity of examples of minimal hypersurfaces with small singular sets. \newline

$\mathbf{C}_2:$
There are examples showing that Theorem \ref{mainb} is optimal.
\newline 

$\mathbf{C}_3:$ 
There are counterexamples to Theorem \ref{maina} in some high dimensions.
\newline

A heuristic for Conjecture $\mathbf{C}_2$ is the following. Suppose that there is an $n$-dimensional minimal hypersurface in the flat $\mathbb{R}^{n+1}$ ($n+1\geq 8$) with finite index, finitely many smooth ends asymptotic to a same hyperplane and a singular set of positive $(n-7)$-Hausdorff measure (the existence of such objects seems plausible). Then by gluing together in a periodic way $k^n$ local models given by such hypersurfaces and rescaling, one might hope to construct a sequence of uniformly bounded $n$-volume minimal hypersurfaces $\Sigma_k$ in the flat torus $(S^1)^{n+1}$, with $\mathcal{H}^{n-7}(\Sing(\Sigma)) \approx k^n. \frac{1}{k^{n-7}} \approx k^7$ and $\ind(\Sigma) \approx k^n$.

We were motivated to formulate Conjecture $\mathbf{C}_3$ in analogy with the case of Einstein manifolds: Cheeger-Naber \cite{CheegerNabercodim4} proved that the family of closed $v$-noncollapsed Einstein $4$-manifolds with diameter bounded by $D$ contains finitely many diffeomorphism types, but Hein-Naber constructed in an unpublished work a sequence of Ricci-flat K\"{a}hler geodesic unit balls of complex dimension $n\geq 3$ uniformly noncollapsed but with $n$-th Betti number going to infinity. Note that variants of Conjecture $\mathbf{C}_3$ were considered by many authors previously (one of the last articles on the subject is a recent preprint of Simon \cite{Simon21} about the possible singular sets of stable minimal hypersurfaces).

\subsection{Sheeting number of min-max minimal surfaces in dimension $3$} 
About the sheeting number in dimension $3$, note that curiously in Theorem \ref{oideal} the powers for the index and the area are exactly the ones predicted by min-max theory for a sequence $\{\Sigma_p\}$ of minimal surfaces produced by usual min-max methods \cite{MaNeinfinity,GasGua}: the minimal surface $\Sigma_p$ has generically index $p$ and area growing as $p^{\frac{1}{3}}$ \cite{ChodMant,Zhoumultiplicityone,MaNemultiplicityone}.
 We are led to ask the following question.\newline

$\mathbf{Q}_3:$
If the metric $g$ is generic or if  $\Ric_g>0$,  is the sheeting number $N(\Sigma_p)$ uniformly bounded for the sequence $\{\Sigma_p\}$  of minimal surfaces constructed by min-max methods in \cite{MaNeinfinity}\cite{GasGua}?\newline

This question can be considered as an attempt to localize the Multiplicity One conjecture of Marques-Neves \cite{MaNeindexbound} (which was solved in \cite{ChodMant,Zhoumultiplicityone}). 

\subsection{Accumulation phenomenon around stable minimal surfaces}
In our joint work with Marques and Neves \cite{MaNeSong}, we formed the heuristic picture that minimal surfaces obtained by usual min-max methods \cite{MaNeinfinity} in a closed manifold $(M^3,g)$ should equidistribute. On the other hand, the sequence of minimal surfaces $\{\Gamma^{(p)}\}$ constructed in our proof of Yau's conjecture \cite[Theorem 10]{AntoineYau} seems to have the opposite behavior, as suggested by Theorem \ref{oideal} and (\ref{hint}) in Section \ref{reamrks}. We conjecture that these surfaces accumulate around a stable minimal surface. More generally\footnote{This has been proved by Xin Zhou and the author in \cite{Song-Zhou20}.}:\newline

$\mathbf{C}_4:$
For any strictly stable $2$-sided closed embedded minimal surface $\Sigma$ in a closed manifold $(M^3,g)$ with a generic metric, there should be a sequence $\{\Gamma_j\}$ of closed embedded minimal surfaces different from $\Sigma$ so that 
$$\frac{ [\Gamma_j]}{\area(\Gamma_j)} \to \frac{[\Sigma]}{\area(\Sigma)}$$
as varifolds when  $j\to \infty$.\newline

The reader can also compare this conjecture with the minimal surfaces constructed by Colding-De Lellis in \cite{ColdingDeLellis}, which accumulate around a stable minimal $2$-sphere.

\section{Appendix}

Let $F$ be a field and let $b^k(.)$ denote the Betti numbers of a set, namely the dimension over $F$ of the cohomology groups $H^k(.,F)$.
The following topological lemma is a well-known result stating that the Betti numbers of a manifold can be bounded from above by the number of open sets in a ``good'' cover with additional properties. See the section about \v{C}ech cohomology in \cite[Appendix 4]{Petersen}, see also Lemma 12.12 in \cite{BallGromSchr} where the authors prove it for the Betti numbers defined as the dimensions of the homology groups $H_k(.,F)$.

 \begin{lemme} \label{lemmetopo}
 Let $M$ be a connected manifold of dimension $d$ and let $\bar{\mu}$ be a positive integer. There is a constant $C=C(d,\bar{\mu})$ such that the following holds. Let $\{U_i\}_{i=1}^m$ be connected open sets of $M$ covering $M$ such that 
 \begin{itemize}
  \item for any different $i_1,,..., i_l \in \{1,...,m\} $ where $l\geq 1$, 
 $$\sum_{k=0}^d b^k(U_{i_1}  \cap ...\cap U_{i_l}) \leq \bar{\mu},$$
 \item
for any $j\in \{1,...,m\}$, $U_j$ intersects at most $\bar{\mu}$ other open sets in $\{U_i\}_{i=1}^m$.
\end{itemize}

 Then 
 $$\sum_{k=0}^d b^{i}(M) \leq C m.$$
 
 \end{lemme}
 \begin{proof}
 For the reader's convenience, we sketch the arguments given in \cite[Lemma 12.12]{BallGromSchr} .
Fix $a \in \{1,...,m\}$ and $k\geq 0$. Let 
$$X:= \bigcup_{i=1}^a U_i, \quad U:= \bigcup_{i=1}^{a-1} U_i, \quad V:= U_a.$$
Then by Mayer-Vietoris'  long exact sequence for singular cohomology over $F$, 
$$H^{k-1}(U \cap V) \rightarrow H^k(X) \rightarrow H^k(U)\oplus H^k(V) $$
so we have
$$b^k(X) \leq b^{k-1}(U\cap V) + b^k(U) + b^k(V). $$ 
Since $U\cap V = \bigcup_{i=1}^{a-1}(U_i \cap U_a)$ and $U_a$ intersects at most $\bar{\mu}$ other open sets $U_i$, from another Mayer-Vietoris' sequence argument using the uniform bound on the total Betti number of intersections of $U_i$, we obtain that $b^{k-1}(U\cap V)$ is bounded by a constant depending only on $\bar{\mu}$ and $k$. By assumption $b^k(V)$ is bounded by $\bar{\mu}$. From here, one can use an induction argument on the number $a$ of open sets to conclude.
 \end{proof}


\bibliographystyle{plain}
\bibliography{biblio20_06_04}

\end{document}